\documentclass[aop,MSNbibl,seceqn,citesort,dvips]{arximspdf}
\usepackage{mathbh}
\usepackage{graphicx}

\doi{10.1214/11-AOP654}
\volume{40}
\issue{4}
\pubyear{2012}
\firstpage{1377}
\lastpage{1435}

\makeatletter

\newtheorem{Theorem}{Theorem}[section]
\newtheorem{Lemma}[Theorem]{Lemma}
\newtheorem{Proposition}[Theorem]{Proposition}
\newproclaim{remark}[Theorem]{Remark}
\newtheorem{claim}[Theorem]{Claim}
\newproclaim{definition}[Theorem]{Definition}

\newcommand{\cA}{{\mathcal A}}
\newcommand{\cB}{{\mathcal B}}
\newcommand{\cC}{{\mathcal C}}
\newcommand{\cE}{{\mathcal E}}
\newcommand{\cF}{{\mathcal F}}
\newcommand{\cN}{{\mathcal N}}
\newcommand{\cQ}{{\mathcal Q}}
\newcommand{\cT}{{\mathcal T}}
\newcommand{\cU}{{\mathcal U}}
\newcommand{\cV}{{\mathcal V}}
\newcommand{\cW}{{\mathcal W}}

\newcommand{\bbC}{{\mathbb C}}
\newcommand{\bbD}{{\mathbb D}}
\newcommand{\bbE}{{\mathbb E}}
\newcommand{\bbN}{{\mathbb N}}
\newcommand{\bbP}{{\mathbb P}}
\newcommand{\bbR}{{\mathbb R}}
\newcommand{\bbZ}{{\mathbb Z}}

\newcommand{\e}{\varepsilon}
\newcommand{\s}{\sigma}
\newcommand{\z}{\zeta}
\newcommand{\D}{\Delta}
\newcommand{\G}{\Gamma}

\newcommand{\tc}{\mid}

\makeatother

\begin{document}
\begin{frontmatter}

\title{Universality in one-dimensional hierarchical coalescence processes\thanksref{T1}}
\runtitle{1D Hierarchical Coalescence Processes}

\thankstext{T1}{Supported by the European Research Council through the ``Advanced Grant'' PTRELSS 228032.}

\begin{aug}
\author[A]{\fnms{Alessandra} \snm{Faggionato}\ead[label=e1]{faggiona@mat.uniroma1.it}},
\author[B]{\fnms{Fabio} \snm{Martinelli}\corref{}\ead[label=e2]{martin@mat.uniroma3.it}},
\author[C]{\fnms{Cyril} \snm{Roberto}\ead[label=e3]{cyril.roberto@univ-mlv.fr}}\\
and
\author[D]{\fnms{Cristina} \snm{Toninelli}\thanksref{t2}\ead[label=e4]{cristina.toninelli@upmc.fr}}

\runauthor{Faggionato, Martinelli, Roberto and Toninelli}

\affiliation{Universita ``La Sapienza'', Universita Roma Tre,
Universit\'{e} de Marne-la-Vall\'{e}e, and Universit\'{e} de Paris
VI-VII and CNRS}

\address[A]{A. Faggionato\\
Dipartimento Matematica ``G. Castelnuovo''\\
Universita ``La Sapienza''\\
P. le Aldo Moro 2\\
00185 Roma\\
Italy\\
\printead{e1}}
\address[B]{F. Martinelli\\
Dipartimento Matematica\\
Universita Roma Tre\\
Largo S. L. Murialdo 00146\\
Roma\\
Italy\\
\printead{e2}\hspace*{6.32pt}}
\address[C]{C. Roberto\\
L.A.M.A.\\
Universit\'{e} de Marne-la-Vall\'{e}e\\
5 bd Descartes 77454 Marne-la-Vall\'{e}e\\
France\\
\printead{e3}}
\address[D]{C. Toninelli\\
L.P.M.A. and CNRS-UMR 7599\\
Universit\'{e} de Paris VI-VII\\
4 Pl. Jussieu 75252, Paris\\
France\\
\printead{e4}}
\end{aug}

\thankstext{t2}{Supported by the
French Ministry of Education through the Grant ANR BLAN07-2184264 and
ANR-2010-BLAN-0108.}

\received{\smonth{7} \syear{2010}}
\revised{\smonth{2} \syear{2011}}

%
\begin{abstract}
Motivated by several models introduced in the physics literature to
study the nonequilibrium coarsening dynamics of one-dimensional
systems, we consider a large class of ``hierarchical coalescence
processes'' (HCP). An HCP consists of an infinite sequence of
coalescence processes $\{\xi^{(n)}(\cdot)\}_{n\ge1}$: each process
occurs in a different ``epoch'' (indexed by $n$) and evolves for an
infinite time, while the evolution in subsequent epochs are
linked in such a way that the initial distribution of $\xi^{(n+1)}$
coincides with
the final distribution of $\xi^{(n)}$. Inside each epoch the process,
described by a suitable simple point process representing the
boundaries between adjacent intervals (domains), evolves as follows. Only
intervals whose length belongs to a certain epoch-dependent finite
range are
\textit{active}, that is, they can
incorporate their left or right neighboring interval with quite general
rates. Inactive intervals cannot incorporate their neighbors and can
increase their length only if they are incorporated by active
neighbors. The activity ranges are such that after a merging step the
newly produced interval always becomes inactive for that epoch but
active for some future epoch.

Without making any mean-field assumption we show that: (i) if the
initial distribution describes a renewal process, then such a property
is preserved at all later times and all future epochs; (ii) the
distribution of certain rescaled variables, for example, the domain
length, has a well-defined and universal limiting behavior as $n\to
\infty$ independent of the details of the process (merging rates,
activity ranges$,\ldots$). This last result explains the universality
in the
limiting behavior of several very different physical systems (e.g., the
East model of glassy dynamics or the Paste-all model) which was
observed in several simulations and analyzed in many physics papers.
The main idea to obtain the asymptotic result is to first write down a
recursive set of nonlinear identities for the Laplace transforms of
the relevant quantities on different epochs and then to solve it by
means of a
transformation which in some sense linearizes the system.
\end{abstract}

%

\end{frontmatter}

\tableofcontents[level=2]

\section{Introduction}
\label{intro}
There are several situations arising in one-dimensional physics in
which the nonequilibrium evolution of the system is dominated by the
coalescence of certain domains or droplets characterizing the
experiment (e.g., large vapor droplets
in breath figures or ordered domains in Ising and
Potts models at zero temperature) which leads to interesting
coarsening phenomena. As pointed out in the physics literature a
common feature of these phenomena is the appearance of a
scale-invariant morphology for large times. Many models, even very
simple ones, have
been proposed in order to capture and explain such a behavior (see, e.g.,
\cite{DBG,DGY1,DGY2} and
\cite{Pr}). Supported by computer simulations and under the key assumption
of a well-defined limiting behavior under suitable rescaling,
physicists have
derived some nontrivial limiting distributions for the relevant quantities.

In many cases the coalescence process dominating
the time evolution has a hierarchical structure which can,
informally, be described as follows.

Assume for simplicity that the state of the
system is described by an infinite sequence of adjacent intervals
(``domains'' in the physics language) with varying length and that
its time evolution is governed by the merging of two
consecutive intervals. Then there exist infinitely many
\textit{epochs} and in the $n$th epoch only those domains whose length belongs
to a suitable epoch-dependent characteristic range are
\textit{active} (or, better, $n$-active); that is, they can
incorporate their left or right neighbor interval with certain
(bounded) rates which could depend on the epoch and on the length of
the domain. Each epoch lasts a very long (mathematically infinite)
time so that at the end of the epoch there are no longer $n$-active
intervals, provided that the total merging rate is strictly positive
for any $n$-active domain. Then the next epoch takes over and the
process is repeated. Clearly, in order for the successive
coalescences to be able to eliminate domains created by previous
epochs and therefore to increase the domain length, some assumptions
about the active ranges should be made. If the $n$th active range
is the interval $[d_{\min}^{(n)},d_{\max}^{(n)})$, then we
require that $d_{\max}^{(n)}=d_{\min}^{(n+1)}$.

An interesting and highly nontrivial example of a hierarchical
coalescence process (HCP in the sequel) is represented by the high
density (or low temperature) nonequilibrium dynamics of the
\textit{East model} after a deep quench from a normal density state (see
\cite{EJ,SE} for physics motivations and discussions and
\cite{FMRT1} for a mathematical analysis).
The East model is a well-known example of kinetically constrained
stochastic particle system with site exclusion which evolves
according to a Glauber dynamics submitted to the following
constraint: the $0/1$ occupancy variable at a given site $x\in
\bbZ$ can change only if the site $x+1$ is empty. In this case, if a
domain represents a~maximal sequence of consecutive occupied sites,
and if the particle density is very high, then the characteristic
range of the length of active domains for the $n$th epoch is
$[2^{n-1},2^n)$, and active domains can only merge with their left
neighbor. Notice that
with this choice for the active range the merging of two $n$-\textit
{active} domains automatically
produces a $n$-\textit{inactive} domain. This is a
technical feature that will always be supposed true throughout the paper.

Another interesting HCP is given by the \textit{Paste-all model}
\cite{DGY2} which was devised to model breath figures formed by
coalescing droplets in one dimension. In this case all the domains are
sub-intervals of the integer
lattice, the $n$-\textit{active} length interval is $\{n\}$, and active domains
merge with their left/right neighbor with rate one.

In~\cite{SE} the authors, under the \textit{assumption} that the
scaled domain length has a well-defined limiting behavior as $n\to
\infty$, computed the exact form of the limiting distribution
for the above defined HCP corresponding to East (see Section C
of~\cite{SE}).
Under a
finite mean hypothesis they find that the limiting behavior
is exactly the same as the one computed in~\cite{DGY2} (always
assuming the limiting behavior and the mean filed hypothesis) for the
Paste-all model, a fact
that they describe as ``surprising.''

Our main result, stated in Theorem \ref{teo2},
solves completely this \textit{enigma}. In fact, 
without making any mean field hypothesis, we:
\begin{longlist}[(a)]
\item[(a)]{prove the existence of a well-defined limiting behavior
which is independent of the various merging rates};
\item[(b)]{classify the
limiting distribution according to the initial one (i.e., the
distribution at the beginning
of the first epoch)}.
\end{longlist}
Slightly more precisely the main content
of our contribution can be formulated as follows. Let $\xi$ denote the
random set of the
separation points between the domains (domain walls in physics
jargon). Then, under very general assumptions on
the merging rates and on the active ranges but always assuming
$d_{\max}^{(n)}=d_{\min}^{(n+1)}$ for each $n$:
\begin{longlist}
\item if at the beginning of the first
epoch $\xi$ is described by a renewal point process
(as implicitly done in the physics papers), then the same property
holds for all times and all epochs;
\item if $Z^{(n)}$ denotes the
domain length at the beginning of the $n$th epoch rescaled by a
factor $1/d_{\min}^{(n)}$, and if $g^{(n)}(\cdot)$ denotes its
Laplace transform,
then $g^{(n)}\to g_{c_0}^{(\infty)}$ where
%
%
\begin{equation}\label{tamponegola1}
g^{(\infty)}_{c_0}(s)= 1- \exp\biggl\{ - c_0 \int_1 ^\infty
\frac{e^{-sx}}{x} \,dx \biggr\},
\end{equation}
provided\vspace*{2pt} that $\lim_{s\downarrow0}
-s \,\frac{d}{ds}\,g^{(1)}(s)/(1-g^{(1)}(s))=c_0$ (necessarily $c_0\in[0,1]$).
Moreover, the above limit exists with
$c_0 =1$ when starting with a stationary renewal point
process (which has therefore a finite mean).
If instead the initial law is in the domain of attraction of
an ${\alpha}$-stable law with ${\alpha}\in(0,1)$, then
the limit exists with $c_0={\alpha}$.
\end{longlist}
The above results, which can be generalized to exchangeable point
processes, explain clearly why apparently very different physical
systems (i.e., with different merging rates and/or active ranges) show
the same asymptotic behavior.

We
want to stress here the crucial ideas behind the proof of our limit
theorem. The first step goes as follows. Inspired by the form of the
limiting distribution found by
the physicists, one uses the theory of
complete monotone functions and Laplace transform,
to show that for each $n$ there exists a~nonnegative Radon measure
$t^{(n)}$ on $(0,\infty)$ such that the Laplace transform for the
$n$th epoch, $g^{(n)}$, can be written as
%
%
\begin{equation}\label{tamponegola2}
g^{(n)} (s)=1- \exp\biggl\{ - \int_{[1,\infty)} \frac{e^{-sx}}{x}
t^{(n)} (dx) \biggr\} .
\end{equation}
Then one observes that the Laplace transforms $\{g^{(n)}\}_{n\ge
1}$ must satisfy a~nonlinear and highly nontrivial recursive system of
identities which, thanks to step one, translate into recursive
identities for the
measures $t^{(n)}$. In turn the latter can be solved to express the
measure $t^{(n)}$
in terms of $t^{(1)}$ in a~simple form. Finally, the explicit
form of $t^{(n)}$ allows us to pass to the limit $n\to\infty$
in the recursive identities and prove the main result.

Coalescence processes (also called coagulation or aggregation
processes) and their time-reversed analog given by fragmentation
processes have also been recently much studied in the mathematical
literature with different motivations and from different points of
view (see, e.g.,~\cite{A,Be} and references therein). Most of
the mathematical research focused on models with a certain
\textit{mean-field} character (i.e., the spatial position of the
coalescing objects does not play any role) with some exceptions (see,
e.g.,~\cite{Be2} and~\cite{LS}). Although our model shows indeed a mean-field
nature (see, e.g., Remark \ref{meanfield}) due to the fact that the
domain wall process $\xi$ is a renewal process or exchangeable at any future
time $t$ if it was so at time $t=0$, we have been able to explore
some dynamical aspects of the HCP for which the geometrical
alignment of the domains is relevant (see Section \ref{fondamenta}).

We conclude by mentioning that in~\cite{FMRT2} the methods developed
here have been successfully applied to other HCPs, where a domain
can also coalesce with both its neighboring domains as in
\cite{BDG}. In this class a particular interesting case is
represented by the model in which (roughly) the smallest interval
merges with its two neighbors. In the mean-field approximation and
by forgetting how much time elapses between and during the
merging events, one can derive a time evolution equation for the
domain size distribution in which the time variable $t$ is a
continuous approximation of the discrete label~$n$ of the epochs.
This equation has been rigorously analyzed in~\cite{GM}
(see also~\cite{Pe} for an interesting review) by means of nonlinear analysis
techniques.


\section{Model and results}
\label{annarella}

In this section we introduce the main objects of our analysis,
namely the \textit{simple point processes}, the \textit{one-epoch
coalescence processes} and the \textit{hierarchical coalescence
processes}. Then we expose our main results. We start by recalling
some basic notions of
simple point processes, referring to~\cite{DV} and~\cite{FKAS} for a
detailed treatment.

\subsection{Simple point processes}

We denote by $\cN$ the family of locally finite subsets $\xi
\subset\bbR$. $\cN$ is a measurable space endowed with the
$\s$-algebra of measurable subsets generated by
\[
\{ \xi\in\cN\dvtx|\xi\cap A_1|=n_1 ,\ldots, |\xi\cap A_k |=n_k
\} ,
\]
$A_1,\ldots, A_k$ being bounded Borel sets in $\bbR$ and
$n_1,\ldots, n_k \in\bbN$. On $\cN$ one can define a metric such
that the above measurable subsets correspond to the Borel sets
\cite{MKK}.
We call \textit{domains} the intervals $[x,x']$ between
nearest-neighbor points $x,x'$ in $\xi\cup\{-\infty, +\infty\} $.
Note that the existence of the domain $[-\infty,x']$ corresponds to
the fact that $\xi$ is bounded from the left and its leftmost point is
given by $x'$.
A similar consideration holds for $[x, \infty]$. Points of $\xi$ are
also called \textit{domain separation points}. Given a point $x\in
\bbR$, we define
\[
d_x^\ell:= \inf\{ t > 0 \dvtx x-t \in\xi\} ,\qquad
d_x ^r := \inf\{ t > 0 \dvtx x+t \in\xi\}
\]
with the convention that the infimum of the empty set is $\infty$.
Note that if $x \in\xi$, then $d_x^\ell$ ($d_x^r$) is simply the
length of the domain to the left (right) of~$x$.



%

%

We recall that a simple point process (shortly, SPP) is any
measurable map from a probability space to the measurable space
$\cN$. With a slight abuse of notation we will denote the realization
of a SPP by $\xi$ while we will usually denote by $\cQ$ its law on
the measurable space
$\cN$. In what
follows $\bbN$ ($\bbN_+$) will denote the set of nonnegative
(positive) integers.
\begin{definition}
(i) We say that a SPP $\xi$ is \textit{left-bounded} if it has a~leftmost point and
has infinite cardinality.
%

\mbox{\hphantom{i}}(ii) We say that a SPP $\xi$ is \textit{$\bbZ$-stationary} if $\xi
\subset\bbZ$ and its law $\cQ$ is
invariant by $\bbZ$-translations, that is, if for any $x \in\bbZ$ the
random set $\xi-x$ has law $\cQ$.

(iii) We say that a SPP $\xi$ is \textit{stationary} if its law $\cQ$ is
invariant under $\bbR$-translations,
that is, if for any $x \in\bbR$ the random set $\xi-x$ has law $\cQ$.
\end{definition}

Thanks to Theorem 1.2.2 in~\cite{FKAS} and its adaptation to the
lattice case, if $\xi$ is $\bbZ$-stationary or
stationary, then a.s. the following dichotomy holds: $\xi$ is
unbounded from the left and from the
right, or $\xi$ is empty. In the sequel we will always assume the
first alternative to hold a.s., and we will write $ \xi=\{x_k\dvtx k \in
\bbZ\}$ with the rules: $x_0 \leq0 < x_1$ and
$x_k < x_{k+1}$ for all $k \in\bbZ$. In the case of a left-bounded
SPP, we enumerate the points of $\xi$ as $\{x_k; k \in\bbN\}$ in
increasing order.
\begin{remark}
If $\xi$ is
$\bbZ$-stationary and a.s. nonempty, then \mbox{$\cQ( 0 \in\xi)>0$}, and
therefore the conditional probability $ \cQ( \cdot| 0 \in\xi)$ is
well defined. On the other hand, if $\xi$ is stationary, then $\cQ(0
\in\xi)=0$, the above conditional probability is therefore not
well defined and has to be replaced by the Palm distribution
associated to $\cQ$~\cite{DV,FKAS}.
We recall that, given the law $\cQ$ of a~stationary SPP with
finite intensity
\[
\lambda_{\cQ}:=\bbE_{\cQ} ( |\xi\cap[0,1]|)
\]
and such that
$\xi$ is nonempty $\cQ$-a.s., the Palm distribution $\cQ_0$
associated to~$\cQ$ is defined as the probability measure on the
measurable space $\cN$ such that
\[
\cQ_0( A)= (1/\lambda_{\cQ}) \bbE_{\cQ}
(
|\{ x \in\xi\cap[0,1] \dvtx{\tau}_x \xi\in A \}
| ) \qquad \forall A \subset\cN\mbox{ measurable}
\]
(see Section 1.2.1 in~\cite{FKAS}). Trivially, $\cQ_0$ has support in
%
%
\begin{equation}\label{camilloha3figlie}
\cN^\infty_0:=\{\xi\in\cN\dvtx0\in\xi, | \xi\cap
(-\infty,0] |= | \xi\cap[0,\infty) |=\infty\} .
\end{equation}
Moreover, $\cQ_0$ uniquely determines the law $\cQ$ since it
holds that
%
%
\begin{equation}\label{scimmiazimbo}
\bbE_\cQ[ f(\xi)
] = \lambda_{\cQ} \bbE_{\cQ_0}\biggl[ \int_0 ^{x_1(\xi) } f(
\xi-t )\,dt \biggr]
\end{equation}
for any nonnegative measurable function $f$ on $\cN$ (cf. Theorem
1.2.9 in~\cite{FKAS}, Theorem 12.3.II in~\cite{DV}). Notice that,
by taking $f=1$, one gets ${\lambda}_{\cQ}= 1/ \bbE_{\cQ_0} (x_1)$.
Consider now the space $(0,\infty)^\bbZ$ endowed with the product
topology with Borel measurable sets. Setting $d_k(\xi)= x_k
(\xi)-x_{k-1}(\xi)$ for $k\in\bbZ$ and $\xi
\in\cN_0^\infty$, the map $ \cN^\infty_0 \ni\xi\to
(0,\infty)^\bbZ$ is a measurable injection, with measurable image. In
particular, the Palm
distribution can be thought of as a probability measure on
$(0,\infty)^\bbZ$. As stated in Theorem 1.3.1 in~\cite{FKAS}, a
probability measure $Q$ on $(0,\infty)^\bbZ$ is the Palm distribution
associated to a stationary SPP with finite intensity and a.s. nonempty
configurations
if and only if $Q$ is shift invariant, and its marginal distributions
have finite
mean.
\end{remark}


We now describe the main classes of SPPs we are interested in.
\begin{definition}\label{zeroassoluto}
Given a probability measure $\mu$ on $(0,\infty)$, we say that~$\xi$
is a \textit{renewal SPP containing the origin and with interval
law $\mu$}, and write $\cQ= \operatorname{Ren}(\mu\tc0)$, if:
\begin{longlist}
\item$0 \in\xi$;
\item$\xi$ is unbounded from the left and from the right and,
labeling the points in increasing order with $x_0=0$, the random
variables $d_k= x_k-x_{k-1}$, $k \in\bbZ$, are i.i.d. with common
law $\mu$.
\end{longlist}
\end{definition}
%
%
%
%
\begin{definition}\label{olivietta}
Given probability measures $\nu$ and $\mu$ on $\bbR$ and $(0, \infty
)$, respectively, we say that $\xi$ is
a \textit{right renewal} SPP with first point law $\nu$ and interval law
$\mu$, and write $\cQ= \operatorname{Ren}(\nu, \mu)$, if:
\begin{longlist}
\item $\xi= \{x_k , k \in\bbN\} $ is a left-bounded SPP;
\item the first point $x_0$ has law $\nu$;
\item $ d_k = x_k -x_{k-1}$ ($k \in\bbN_+$)
has law $\mu$;
\item
the random variables $x_0, \{d_k\}_{k \in\bbN_+}$ are independent.
\end{longlist}
\end{definition}
%
%
%
\begin{definition}\label{levico}
Given a probability measure $\mu$ on $\bbN_+$ with finite mean, we
say that $\xi$ is a \textit{$\bbZ$-stationary renewal SPP with
interval law $\mu$}, and\vadjust{\goodbreak} write $\cQ=\operatorname{Ren} _{\bbZ} (\mu)$,
if:
\begin{longlist}
\item $\xi$ is $\bbZ$-stationary and a.s. nonempty;
\item w.r.t. the conditional probability $\cQ(\cdot|0
\in\xi)$ the random variables $d_k= x_k-x_{k-1}$, $k \in\bbZ$, are
i.i.d. with common law $\mu$.
\end{longlist}
\end{definition}

A basic example is the following. Consider a Bernoulli product measure
on $\{0,1\}^\mathbb{Z}$ with parameter $p$.
Any realization $(X_i)_{i\in\mathbb{Z}}$ can be identified with the
subset $\xi=\{i \in\mathbb{Z} \dvtx X_i=1\}$. The resulting SPP is a
$\bbZ$-stationary renewal SPP with geometric interval law.
\begin{remark}
As proven in Appendix \ref{puntini}, given a
probability measure~$\mu$ on~$\bbN_+$, the law
$\cQ=\operatorname{Ren}_{\bbZ}(\mu)$ is well defined iff $\mu$ has
finite mean.
Other properties of $\bbZ$-stationary renewal SPPs are also discussed
there.
\end{remark}
\begin{definition}\label{fataturchina}
Given a probability measure $\mu$ on $(0,\infty)$ with finite mean, we
say that $\xi$ is a \textit{stationary renewal SPP with interval law
$\mu$}, shortly $\xi=\operatorname{Ren} (\mu)$, if:
\begin{longlist}
\item $\xi$ is a stationary SPP with finite intensity and $\xi$
is nonempty a.s.;
\item the random
variables $d_k= x_k-x_{k-1}$, $k \in\bbZ$, are i.i.d. with common
law $\mu$ w.r.t. the Palm distribution associated to $\cQ$.
\end{longlist}
\end{definition}

A classical example of stationary renewal SPP is given by the
homogeneous Poisson point process, for which the interval law is an
exponential.

\begin{remark}
A stationary renewal SPP with interval law $\mu$ having infinite
mean cannot exist (see Proposition 4.2.I in~\cite{DV}). As
discussed after Theorem 1.3.4 in~\cite{FKAS}, $\cQ=\operatorname{Ren}
(\mu)$
if and only if the following holds: the random variables
$d_k=x_k-x_{k-1}$, $k\not=1$, are i.i.d. with law $\mu$ and are
independent from the random vector $(x_0,x_1)$, which satisfies
%
%
\begin{eqnarray}
\cQ( -x_0> u, x_1> v) = {\lambda} _\cQ\int_{u+v}^\infty\bigl( 1-
F(t)\bigr) \,dt ,\nonumber\\[-8pt]\\[-8pt]
&&\eqntext{F(t):= \mu((0,t]) ,
u,v>0 .}
\end{eqnarray}
\end{remark}

We conclude with the definition of two large classes of
``exchangeable'' point processes.
\begin{definition}\label{baratto}
We say that $\xi$ is a \textit{left-bounded exchangeable SPP containing the
origin} if:
\begin{longlist}
\item$\xi=\{x_k, k\in\bbN\}$ is a left-bounded SPP containing the
origin;
\item$\cQ$, thought of as probability measure on $(0,\infty)^{\bbN
_+}$ by the map $\xi\to
( x_k-x_{k-1}\dvtx k \in\bbN_+)$, is exchangeable (i.e.,
invariante under permutations~\cite{D,K}).
\end{longlist}
\end{definition}
\begin{definition}\label{scambioequo}
We say that $\xi$ is a \textit{stationary exchangeable SPP} if:
\begin{longlist}
\item $\xi$ is a stationary SPP with finite intensity and $\xi$
is nonempty
a.s.;
\item the Palm distribution $\cQ_0$, thought of as probability
measure on $(0,\infty)^\bbZ$
by the map $\xi\to( x_{k}-x_{k-1}\dvtx k \in\bbZ)$, is
exchangeable.
\end{longlist}
\end{definition}
\begin{remark}
Any left-bounded or stationary renewal SPP is also
exchangeable.
\end{remark}
%

\subsection{One-epoch coalescence process} \label{secone-epoch}
We describe here the class of coalescence processes
which will represent the modular unity of the, yet to be defined,
hierarchical coalescence process (HCP).
For a reason that will become clear in the next section, we call it
\textit{one-epoch coalescence process} (OCP).

This process depends on two constants $0<d_{\min}<d_{\max}$ and
on nonnegative bounded functions ${\lambda}_\ell, {\lambda}
_r$ defined on $[d_{\min} , \infty]$ which, with $ {\lambda} (d):=
{\lambda}_\ell
(d)+ {\lambda}_r (d)$, satisfy the following assumptions:
\begin{longlist}[(A2)]
\item[(A1)] ${\lambda}(d) >0$ if and only if $d \in[d_{\min},
d_{\max})$;
\item[(A2)] if $d,d' \geq d_{\min}$, then $d+d'\geq d_{\max}$.
\end{longlist}
Trivially, (A2) is equivalent to the bound $2d_{\min}\geq d_{\max}$.

The one-epoch
coalescence process is a Markov process with state space $ \cN
(d_{\min})$ given by the configurations $\xi\in\cN$ having only domains
of length not smaller than $d_{\min}$, that is,
%
%
\begin{equation}\label{scudetto} \cN(d_{\min})= \{\xi\in\cN\dvtx
d_x^{ \ell} \geq d_{\min} , d_x^{ r} \geq d_{\min}\mbox{ }
\forall x \in
\xi\} .
\end{equation}
The stochastic evolution is given by a jump dynamics
with
c\`{a}dl\`{a}g paths $( \xi(t) \dvtx t \geq0 )$ in the
Skohorod space $D([0,\infty),
\cN(d_{\min}) )$ (cf.~\cite{B}), and at each jump a point is
removed. Formally, the Markov generator of
the coalescence process is given by
%
%
\begin{equation}\label{pescedaprile} Lf (\xi) = \sum_{x \in\xi}
\bigl( {\lambda}_\ell(d_x^{ \ell} )+
{\lambda}_r (d_x ^{ r} )\bigr) [ f( \xi\setminus\{x\})-
f(\xi)
] .
\end{equation}
We will write $\bbP_{\cQ}$ for the law on
$D([0,\infty), \cN(d_{\min}) )$ of the one-epoch coalescence
process with initial law $\cQ$ on $\cN(d_{\min})$ and $\cQ_t$ for its
marginal at time~$t$.

We keep the discussion of the
Markov generator at a formal level, since we prefer to give a
constructive definition of the coalescence process. Here we give two
rough alternative descriptions of the dynamics as a random process of
points or as a random process of domains (intervals).

\subsubsection*{Point dynamics} To each point $x\in\xi(0) $ we
associate two exponential random variables $T_{x,\ell}$ and
$T_{x,r}$ of parameter ${\lambda}_\ell(d_x^{ \ell} )$ and $
{\lambda}_r (d_x
^{ r} )$, respectively. We stress that $d_x^{ \ell}$ and $ d_x
^{ r} $ refer to the configuration $\xi(0) $: at time $0$ the
domains on the left and on the right of the point $x $ are,
respectively, $[x-d_x^{ \ell},x]$ and $[x, x+ d_x^{ r}]$. All
random variables must be independent. If $t= T_{x, \ell}\leq
T_{x,r}$ and at time $t-$ the point $x - d^{ \ell}_x$ still
exists, then we set $\xi(t)= \xi(t-) \setminus\{x\} $. Moreover, we
say that the two domains having $x$ as separation point merge or
coalesce at time $t$, and that the domain on the left of $x$
incorporates the domain on the right of $x$ at time $t$. If $t=
T_{x, r}< T_{x,\ell}$,
and at time $t-$ the point $x +
d^{ r}_x$ still exists, then we set $\xi(t)= \xi(t-) \setminus
\{x\} $. Moreover, we say that the two domains having $x$ as
separation point merge or coalesce at time~$t$, and that the domain
on the right of $x$ incorporates the domain on the left of $x$ at
time $t$. Finally, if $t= T_{x, r}\wedge T_{x,\ell}$, but the above
two cases do not take place, then we set $\xi(t)= \xi(t-)$. See
Figure \ref{figblocking} for an illustration.

In order to formalize the above construction, we proceed as
follows. Given \mbox{$t>0$}, we define $ \Upsilon_t$ as the set of points
$x\in\xi(0)$ such that $T_{x,\ell} \wedge T_{x,r}\leq t $. On the
set $ \Upsilon_t$ we define a graph structure putting an edge
between points \mbox{$x,y\in\Upsilon_t$} if and only if $x$ and $y$ are
consecutive points in $\xi(0)$. Since the functions~${\lambda}_\ell,
{\lambda}_r$
are bounded from above, a.s. for any fixed time $t$, the above graph
$\Upsilon_t$ has only connected components (clusters) of finite
cardinality. Then, $ \xi(0)\setminus\Upsilon_t $ is included in
$\xi(s)$ for all $s \in[0,t]$, while the evolution of $
(\xi(s)\dvtx s \in[0,t]) $ restricted to each cluster of
$\Upsilon_t $ follows the rules stated at the beginning, which are
now meaningful a.s. since clusters have finite cardinality a.s.

\subsubsection*{Domain dynamics} We give here only a rough description
of the
dynamics. In Section \ref{secdomaindyn} we will discuss in detail a
basic coupling leading to the definition on the same
probability space of the domain dynamics for all initial
configurations $\xi(0)\in\cN(d_{\min})$.

One assigns to each
domain $\D=[x,x']$ with length $d$ present in $\xi(0)$ an
exponential random variable $T_\D$ of parameter ${\lambda} (d)$ and a coin
$C_\D$ with faces $-1,1$ appearing with probability $ {\lambda}_r
(d)/{\lambda}(d)$ and ${\lambda}_\ell(d)/{\lambda}(d)$,
respectively. All random
variables must be independent. If $t = T_\D$ and if at time $t-$ the
domain $\D$ still exists, then at time $t$ the domain $\D$
incorporates its left domain [i.e., $\xi(t)= \xi(t-) \setminus
\{x\}$] if $C_\D=-1$, while $\D$ incorporates its right domain
[i.e., $\xi(t)= \xi(t-) \setminus\{x'\}$] if $C_\D=1$.

%
%
\begin{figure}

\includegraphics{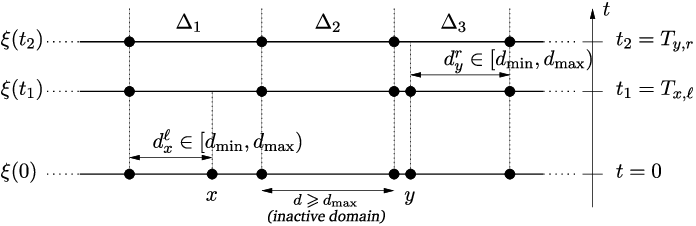}

\caption{An example of the one-epoch coalescence process starting
from $\xi(0)$. At time $t=0$, the domain of length $d$ is inactive
since $d \geq d_{\max}$. At time $t_1$, site $x$ disappears and since
$t_1=T_{x,\ell}$, the domain on the left of $x$ incorporates the
domain on the right of $x$. Analogously at $t_2$ point $y$ disappears
since the domain on the right of $y$ incorporates the domain on the
left of $y$. The domain $\Delta_1$ and $\Delta_3$ are
inactive since they are resulting from a coalescence. The domain
$\Delta_2$ is frozen for $t>t_2$, due to the presence of $\Delta_1$ and
$\Delta_3$. This illustrates the blocking effect.}
\label{figblocking}
\end{figure}

We can now explain the dynamical meaning of assumptions (A1) and~(A2):
\begin{itemize}
\item
(A1) means that a domain is \textit{active}, that is, it can incorporate
another domain, \textit{iff} its length $d$ lies
in $[d_{\min},d_{\max})$.
\item
(A2) means that a domain resulting from a
coalescence is not active.
\end{itemize}
As consequence, the following \textit{blocking effect} appears: given
three nearest-neighbor inactive
domains $\D_1, \D_2, \D_3$, the intermediate domain $\D_2$ is
frozen, in the sense that its extreme points cannot be erased; see Figure
\ref{figblocking}.

By definition of the one-epoch coalescence process, points can only be
removed. Therefore, on any finite interval $I$, $\xi(t) \cap I$
converges as $t\to\infty$, and the following lemma follows at once.
\begin{Lemma}\label{volarelontano}
For any given initial condition $\xi\in\cN(d_{\min})$ the
following hold:
\begin{longlist}
\item $\xi(t)
\subset\xi(s)$ if $s \leq t$;
\item
the configuration $\xi(t) $ is constant on
bounded intervals eventually in~$t$;
\item there exists a unique element $\xi(\infty)$ in $\cN
(d_{\max})$
such that
$\xi(t)\cap I=\xi(\infty)\cap I$ for all large enough $t$ (depending
on $I$) and all bounded intervals~$I$.
\end{longlist}
\end{Lemma}

Due to the above lemma, $\xi(\infty)$, the SPP representing the
asymptotical state of the coalescence process, is well defined. Our
first main result is given by the following two theorems.
It states that,
starting from a left-bounded renewal (resp.,
a~$\mathbb{Z}$-stationary 
or stationary renewal) SPP $\xi$,
at a later time~$t$ the coalescence process $\xi(t)$ remains of the
same type. Moreover, there
exists a key identity between the Laplace transform of the interval
law at time $t=0$ and time $t=\infty$. This equation, that we call
\textit{one-epoch recursive equation}, will play a crucial role in a
recursive scheme for the hierarchical coalescence process.

\begin{Theorem}[(Renewal property)]
\label{teo1}
Let $\nu, \mu$ be probability measures on~$\bbR$ and $[d_{\min},
\infty)$, respectively. Then, for all $t\in[0,\infty]$ there exist
probability measures $\nu_t,\mu_t$ on $\bbR$ and $[d_{\min},
\infty)$,
respectively, such that $\nu_0=\nu$, $\mu_0=\mu$ and:
\begin{longlist}
\item if $\cQ=\operatorname{Ren}(\nu,\mu)$, then
$\cQ_t=\operatorname{Ren} (\nu_t, \mu_t)$;
\item if $\cQ= \operatorname{Ren} (\mu)$, then
$\cQ_t= \operatorname{Ren} (\mu_t)$;
\item if $\cQ= \operatorname{Ren}_\bbZ(\mu)$, then
$\cQ_t=\operatorname{Ren}_\bbZ(\mu_t)$;
\item If $\cQ= \operatorname{Ren} ({\delta}_0, \mu)$, then $\cQ_t(\cdot
\tc0\in
\xi)= \operatorname{Ren}({\delta}_0,\mu_t)$;
\item$\lim_{t\to\infty}\nu_t=\nu_\infty$ and $\lim_{t\to
\infty}\mu_t=\mu_\infty$ weakly.
\end{longlist}
\end{Theorem}
\begin{Theorem}[(Recursive
identities)]
\label{teo1bis} Let $\nu, \mu$ be probability measures on $\bbR$
and $[d_{\min},
\infty)$, respectively, and let $\nu_t,\mu_t$ be the probability
measures introduced in Theorem \ref{teo1}.
\begin{enumerate}[(ii)]
\item[(i)] Consider the Laplace/characteristic functions
%
%
\begin{eqnarray}
\label{orchestra1}
G_t(s) &=& \int_{[d_{\min}, \infty) } e^{-s x} \mu_t (dx) ,\qquad
s \in\bbR_+\cup i\bbR, \\
\label{orchestra2}
H_t(s) &=& \int_{ [d_{\min}, d_{\max})} e^{-sx} \mu_t (dx) ,\qquad
s \in
\bbR_+\cup i \bbR.
\end{eqnarray}
Then, for any $s \in\bbR_+\cup i \bbR$, the following \textit
{one-epoch recursive equation} holds:
%
%
\begin{equation}\label{pietra}
1- G_\infty(s)= [ 1-G_0(s)] e^{H_0 (s)}.
\end{equation}
\item[(ii)] Consider the
Laplace/characteristic function
%
\[
L_t(s) = \int e^{-s x} \nu_t(dx),\qquad s\in\bbR_+ \cup\in i
\bbR
.
\]

%
\begin{enumerate}[(b)]
\item[(a)]
If ${\lambda}_r \equiv0$, then $\nu_t=\nu_0$ for all $t\geq0 $.
Hence $L_t(s)=L_0(s)$ for all $t \geq0$.
\item[(b)] If $ {\lambda}_\ell= \gamma{\lambda}_r$ for some
$\gamma\in
[0,\infty)$, then, for any $s \in\bbR_+ \cup i \bbR$,
%
%
\begin{equation}\label{pietrabis}
L_\infty(s)= L_0 (s) \exp\biggl\{\frac{H_0 (s)-H_0(0)}{1+\gamma}
\biggr\}.
\end{equation}
Moreover, if $\cQ=\operatorname{Ren}(\nu,\mu)$
%
%
\begin{equation}\label{troppacocacola}
\mathbb{P}_ \cQ\bigl( x_0(0) \in\xi(\infty) \bigr) =
e^{-{H_0(0)}/({1+\gamma})} ,
\end{equation}
where $x_0(0)$ denotes the first point of the initial configuration
$\xi(0)$.
\end{enumerate}
\end{enumerate}
\end{Theorem}

The proofs of Theorems \ref{teo1} and
\ref{teo1bis} are given in Sections \ref{fondamenta} and
\ref{brunatecchio}.
\begin{remark}
In (ii) we have analyzed two cases [(a) and (b)] motivated by the
East model and by the Paste-all model. The arguments used in the
proof of point (ii) could, however, be applied to other cases as
well. We stress that the Laplace transform $L_t (s)$, $s \in
\bbR_+$, could diverge since $\nu_t$ has support on $\bbR$.
Therefore, the above identities in point (ii) have to be thought
of as identities in the extended space $[0,\infty]$.
\end{remark}

We point out that the one-epoch recursive equation (\ref{pietra})
uniquely determines $\mu_\infty$ when knowing $\mu_0, d_{\min},d_{\max
}$. In
particular, these three elements are the unique traces of the
dynamics that asymptotically survive. In other words, the precise form
of the rates ${\lambda}_\ell$ and ${\lambda}_r$ is irrelevant. In
the case
of a left-bounded renewal point process the limiting first point law
$\nu_\infty$ does not share such a universality, although the trace
of ${\lambda}_\ell$ and ${\lambda}_r$ on $\nu_\infty$ is only partial.
\begin{remark}
\label{meanfield}
Assume for simplicity that $\mu$ is
concentrated on $\bbN_+$, so that the domains have
integer length at any time. After properly constructing the Markov generator
(\ref{pescedaprile}) one could prove that
%
%
\begin{equation}\label{sparaalcuore}\qquad
\partial_t \mu_t(d)=-{\lambda}(d) \mu_t (d)+ \sum_{x=1}
^{d-1}[ {\lambda}_\ell( x)+{\lambda}_r (d-x) ]\mu_t
(x)\mu_t (d-x) .
\end{equation}
Note that if $d$ is active, then only the first addendum in the
right-hand side is present, while if $d$ is inactive this first addendum is
absent. From this observation, one easily obtains that $\partial_t
G_t= (1-G_t)\partial_t H_t$, and therefore
%
%
\begin{equation}\label{tempigenerici}
1-G_t(s)= \bigl( 1-G_0(s)\bigr)\exp\{H_0(s)-H_t(s)\} \qquad
\forall t,s \geq0 .
\end{equation}
Taking the limit $t \to\infty
$ one gets (\ref{pietra}).
This
strategy has been applied in~\cite{SE}, where the treatment is not
rigorous, and will be formalized in~\cite{FMRT2} in order to treat
other coalescence processes as in~\cite{BDG}. It could be applied to
derive (\ref{pietrabis}). While the Smoluchoswski-type equation
(\ref{sparaalcuore}) has a mean-field structure (see, e.g.,
\cite{A}), in proving (\ref{pietra}) and (\ref{pietrabis}), we have
followed here a more constructive strategy, and we have investigated
how a domain of given length can emerge at the end of the epoch or
how a given point can become the first point for the configuration
$\xi(\infty)$ at the end of the epoch.
\end{remark}

\subsection{The hierarchical coalescence process}


We can finally introduce the \textit{hierarchical coalescence process}
(HCP). The dynamics depend on the
following parameters and functions: a strictly increasing sequence of
positive numbers
$\{d^{(n)}\}_{n\ge1}$ and a family of uniformly bounded functions
${\lambda}_\ell^{(n)}, {\lambda}_r^{(n)}\dvtx[d^{(n)},\allowbreak\infty
]\rightarrow[0,A] ,n
\geq1$. Without loss of generality we may assume that $d^{(1)}=1$.
We set as before ${\lambda}^{(n)}:= {\lambda}^{(n)}_\ell+ {\lambda
}^{(n)}_r$, and we
assume:
\begin{longlist}[(A3)]
\item[(A1)] for any $n \in\bbN_+$, ${\lambda}^{(n)} (d) >0$ if and
only if $d \in[d^{(n)}, d^{(n+1)})$;
\item[(A2)] for any $n \in\bbN_+$, if $d,d'\geq d^{(n)}$, then
$d+d'\geq d^{(n+1)}$ (i.e., $2d^{(n)} \geq
d^{(n+1)}$);
\item[(A3)] $\lim_{n\to\infty}d^{(n)}=\infty$.
\end{longlist}
For example, one could take $d^{(n)}=n$ or $d^{(n)} =
a^{n-1}$ with $a \in(1,2]$.

The HCP is then given by a sequence of one-epoch coalescence
processes, suitably linked. More precisely, the stochastic evolution
of the HCP is described by the sequence of paths $\{ \xi^{(n) }
(\cdot)\}_{n\ge1}$ where each $\xi^{(n)}$ is the random path
describing the evolution of the one-epoch coalescence process with
rates~${\lambda}^{(n)}_\ell,{\lambda}^{(n)}_r$, active domain lengths
$d^{(n)}_{\min}=d^{(n)}, d^{(n)}_{\max}=d^{(n+1)}$ and initial
condition $\xi^{(n)}(0)=\xi^{(n-1)}
(\infty)$, \mbox{$n\geq2$}.
Informally we refer to $\xi^{(n)}$ as describing the evolution in the
$n$th epoch. See Figure \ref{fighcp} for a graphical
illustration.\looseness=1

%
%
\begin{figure}

\includegraphics{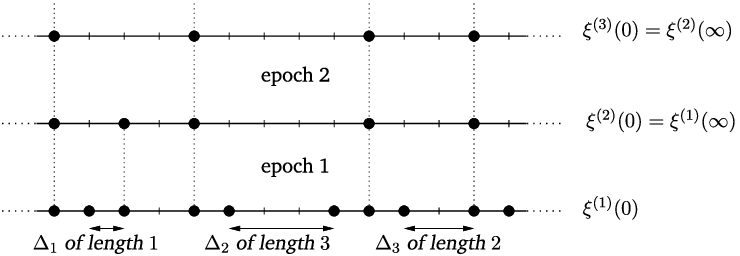}

\caption{An example of HCP dynamics, with $d^{(n)}=n$. The
distances between the points are, from left to right, $1$, $1$
(corresponding to $\Delta_1$), $2$, $1$, $3$ (corresponding to
$\Delta_2$), $1$, $1$, $2$ (corresponding to $\Delta_3$)$\ldots.$ At the
beginning of epoch 1, only the
domains of length in $1$ are active. In particular, $\Delta_1$ is
active while $\Delta_2$ and $\Delta_3$ are inactive. At the end of
epoch~1, there are no more domains of length less than
$2$ (see Lemma \protect\ref{volarelontano}). At the beginning of
epoch 2,
domains of length $2$ are active and at the end, there are no more
domains of length less than
$3$, and so on. Note
that an inactive domain as $\Delta_2$ can increase its length.}
\label{fighcp}
\end{figure}

Theorem \ref{teo1} gives us information on the evolution and its
asymptotic inside each epoch when the initial condition is a SPP of
the renewal type. If, for example, the initial distribution $\cQ$ for the
first epoch is $\operatorname{Ren}(\nu,\mu)$, we can use Theorem \ref{teo1}
together with the link $\xi^{(n+1)}(0)=\xi^{(n)}
(\infty)$ between two consecutive epochs to recursively define the measures
$\mu^{(n)},\nu^{(n)}$ by
%
%
\begin{eqnarray}
\label{eq1}
\mu^{(n+1)}&=&\mu_\infty^{(n)},\qquad
\mu^{(1)}=\mu,\nonumber\\[-8pt]\\[-8pt]
\nu^{(n+1)}&=&\nu_\infty^{(n)},\qquad \nu^{(1)}=\nu.\nonumber
\end{eqnarray}
With this position it is then natural to ask if, in some suitable
sense, the measures
$\mu^{(n)},\nu^{(n)}$ have a well-defined limiting behavior as $n\to
\infty$. The affirmative
answer is contained in the following theorem, which is the core of
the paper.
Before stating it we need a result on the Laplace transform of
probability measures on $[1,\infty]$.
\begin{Lemma}
\label{prelimmainth}
Let $\mu$ be a probability measure on $[1, \infty)$, and let $g(s)$ be
its Laplace transform, that is, $g(s)= \int e^{-s x}
\mu(dx) , s >0$.
%
\begin{enumerate}[(ii)]
\item[(i)] If
%
%
\begin{equation}
\label{limarrosto}
\lim_{s \downarrow0} -\frac{s
g'(s)}{1-g(s)}=c_0,
\end{equation}
then necessarily $0\le c_0\le1$.
\item[(ii)] The existence of limit (\ref{limarrosto}) holds if:

\begin{enumerate}[(a)]
\item[(a)]
$\mu
$ has finite mean and then $c_0=1$, or

\item[(b)] for some ${\alpha}\in(0,1)$ $\mu$ belongs to the domain of
attraction of an
${\alpha}$-stable law or, more generally, $\mu((x,\infty)
)=x^{-{\alpha}} L(x)$
where $L(x)$ is a~slowly varying\setcounter{footnote}{2}\footnote{A function $L$ is said to
be slowly varying at infinity, if for all $c>0$, $\lim_{x \to
\infty} L(cx)/L(x)=1$.} function at $+\infty$, ${\alpha}\in[0,1]$,
and in
this case $c_0=\alpha$.
\end{enumerate}
\end{enumerate}
\end{Lemma}
\begin{remark}
One could wonder if limit
(\ref{limarrosto}) always exists. The answer is negative and an
example is given in Appendix \ref{controcampo}.
\end{remark}

The proof of Lemma \ref{prelimmainth} is discussed in Appendix \ref{appA}.
\begin{Theorem}\label{teo2}
Let $\nu, \mu$ be probability measures on
$\bbR$ and $[1, \infty)$, respectively, and let $g(s)$ be the Laplace
transform of $\mu$. Let $\cQ$ be the initial law of $\xi^{(1)}$, and
suppose that $\cQ$ is either $\cQ=\operatorname{Ren} (\nu, \mu)$ or $\cQ=
\operatorname{Ren} (\mu)$ or $\cQ=\operatorname{Ren}_\bbZ(\mu)$. For
any $n\in
\bbN_+$ let $X^{(n)}$ be a random variable with law $\mu^{(n) }$
defined in (\ref{eq1})
so that $g(s):=\bbE[ e^{-s X^{(1)}}]$.

If (\ref{limarrosto}) holds for $g$,
then the rescaled variable $Z^{(n)}:=X^{(n)}/d^{(n)}$ weakly
converges to the random variable $Z^{(\infty)}\equiv Z^{(\infty)}_{c_0}$
whose Laplace transform is given by
%
%
\begin{equation}\label{macedonia}
g^{(\infty)}_{c_0}(s)= 1- \exp\biggl\{ - c_0 \int_1^\infty
\frac{e^{-sx}}{x} \,dx \biggr\}.
\end{equation}
The corresponding probability density is of the
form $z_{c_0} (x)\mathbh{1}_{x \geq1}$, where $z_{c_0}$ is
the continuous function on $[1, \infty)$ given by
%
%
\begin{equation}\label{vonnegut}
z_{c_0}(x)= \sum_{k=1}^\infty\frac{(-1)^{k+1}c_0^k}{k!} \rho_k(x)
\mathbh{1}_{x\geq k } ,
\end{equation}
where $ \rho_1(x)= 1/x$ and
\[
\rho_{k+1} (x)= \int_1^\infty d x_1
\cdots\int_1 ^\infty dx _k \frac{1}{ x-\sum_{i=1}^{k} x_i } \prod
_{j=1}^{k} \frac{1}{x_j}
\mathbh{1}_{\sum_{i=1}^k x_i\leq x-1}
,\qquad k\geq1 .
\]
\end{Theorem}

The proof
of Theorem \ref{teo2} is discussed in Section \ref{secproofteo2}.
\begin{remark}
The remarkable fact of the above result is that the only reminiscence
of the initial distribution in the limiting law is through the
constant $c_0$ which, as
proved in Lemma \ref{limarrosto}, is ``universal'' for a large class
of initial laws~$\mu$. Hence the term \textit{universality} in the title.
We also stress that, starting with a stationary or $\bbZ$-stationary
renewal SPP, the weak limit of~$Z^{(n)}$ always exists
and it is universal ($c_0=1$), not depending on the rates~${\lambda}^{(n)}_\ell$, ${\lambda}^{(n)}_r$.
\end{remark}
\begin{remark}
We point out that the asymptotic Laplace distribution~$g^{(\infty)}_{c_0}$ can be written also as
\[
g^{(\infty)}_{c_0}(s)= 1-
\exp\biggl\{ - c_0 \int_s^\infty\frac{e^{-x}}{x} \,dx \biggr\} = 1-
\exp\{ - c_0 \operatorname{Ei}(s) \} ,
\]
where $\operatorname{Ei}(\cdot)$ denotes the exponential integral function. This
is indeed the form appearing in~\cite{DGY2} and~\cite{SE} with $c_0
=1$ (see previous remark).
\end{remark}

If the law $\mu$ has finite mean then by the above Theorem combined
with~(ii) of Lemma \ref{prelimmainth} we know that
$Z^{(n)}$ weakly converges to the random variable~$Z_{1}^{(\infty)}$.
Actually we can
improve our result to higher moments.
\begin{Proposition} \label{propkmoment}
In the same setting of Theorem \ref{teo2} assume that
$d^{(n)}=a^{n-1}$ for some $a \in(1,2]$, and that $\mu$ has finite
$k$th moment, $k\in\bbN_+$. Then, for any function $f\dvtx[0,\infty)
\to\bbR$ such that $|f(x)| \leq C + C x^k$ for some constant $C$,
it holds
%
%
\begin{equation}\label{unmomento}
\lim_{n \to\infty} \bbE\bigl[ f\bigl(Z^{(n)}\bigr)
\bigr] = \bbE\bigl[ f\bigl(Z_{1}^{(\infty)}\bigr)\bigr].
\end{equation}
\end{Proposition}
\begin{remark}
The choice $d^{(n)}=a^{n-1}$ in Proposition \ref{propkmoment} is
technical and could be relaxed, but at the price of extra hypotheses
(that would not include the case $d^{(n)}=n$, e.g.). In order
to keep the computations as simple as possible we decided to focus
on this particular example which is of interest for applications to
the East model.
\end{remark}

The proof of
Proposition \ref{propkmoment} can be found in Section
\ref{seckmoment}.
Next we concentrate on the asymptotic behavior of the first point law
when starting with a left-bounded renewal SPP.
\begin{Theorem}\label{teo3}
Let $\nu, \mu$ be probability measures on
$\bbR$ and $[1, \infty)$, respectively, and consider the hierarchical
coalescence
process such that the
initial law $\cQ$ of $\xi^{(1)}$ is $\operatorname{Ren} (\nu, \mu)$. Assume
%
%
\begin{equation}\label{turisti}
{\lambda}^{(n)}_\ell= \gamma{\lambda}^{(n)}_r\qquad \forall n
\geq
1 ,
\end{equation}
for some $\gamma\in[0,\infty)$, and let, for any
$n\in\bbN_+$, $X_0^{(n)}$ be
the position of the first
point of the HCP at the beginning of the $n$th epoch.\vadjust{\goodbreak}

If limit (\ref{limarrosto}) exists for the Laplace transform
$g$ of $\mu$ then, as $n\to\infty$, the rescaled random variable
$Y^{(n)}:=X_0^{(n)}
/ d^{(n)}$ weakly converges to the positive random variable
$Y^{(\infty)}_{c_0}$ with Laplace transform given by
%
%
\begin{equation}\label{thenero}
\bbE\bigl( e^{-s Y^{(\infty)}_{c_0}} \bigr)=\exp\biggl\{-
\frac{c_0}{1+\gamma} \int_{(0,1)} \frac{1-e^{-sy} }{y} \,dy
\biggr\},\qquad
s\in\bbR_+.
\end{equation}
\end{Theorem}

We point out that if ${\lambda}_r ^{(n)} \equiv0$ for all $n \geq1$,
the first point does not
move.\footnote{This is the case for the HCP associated to the West
version of the East model, that is, to the kinetically constrained
model with Glauber dynamics for which the occupation variable at $x$
can be updated iff $x-1$ is empty.} In particular, its asymptotic is
trivial. Theorem \ref{teo3}
is proven in Section~\ref{secteo3}.

Finally, we evaluate the surviving probability of a given point:
\begin{Theorem}\label{teo4}
Let $\nu, \mu$ be probability measures on
$\bbR$ and $[1, \infty)$, respectively, and consider the hierarchical
coalescence
process with initial law~$\cQ$. Assume
%
%
\begin{equation}
{\lambda}^{(n)}_\ell= \gamma{\lambda}^{(n)}_r\qquad \forall n
\geq
1 ,
\end{equation}
for some $\gamma\in[0,\infty)$, and let, for any
$n\in\bbN_+$, $X_0^{(n)}$ be
the position of the first
point of the HCP at the beginning of the $n$th epoch.

If the limit (\ref{limarrosto}) exists for the Laplace transform
$g$ of $\mu$ then, as $n\to\infty$:
\begin{longlist}
\item if $\cQ= \operatorname{Ren} (\nu, \mu)$, then:
\[
\mathbb{P}_{\cQ} \bigl( X_0^{(n)} = X_0^{(1)}\bigr) =
\bigl(1/d^{(n)} \bigr)^{({c_0}/({1+\gamma}))(1+o(1))};
\]
\item if $\cQ=\operatorname{Ren}(\mu\tc0)$, then
\[
\mathbb{P}_{\cQ} \bigl(0 \in
\xi^{(n)}(0) \bigr) = \bigl(1/ d^{(n)} \bigr)
^{c_0(1+o(1))}.
\]
\end{longlist}
\end{Theorem}

Note that (ii) does not depend on the value of
$\gamma$. Theorem \ref{teo4} is proven in Section \ref{secteo4}.

Extension of
the above results to one-epoch coalescence process or hierarchical
coalescence process with initial law $\cQ$ describing an exchangeable
SPP will be discussed in Appendix \ref{puntiniscambio}.


\section{\texorpdfstring{Renewal property in the OCP: Proof of Theorem \protect\ref{teo1}}{Renewal property in the OCP: Proof of Theorem 2.13}}\label{fondamenta}
In this section and in the next one we will prove our
results concerning the one-epoch
coalescence process (Theorems \ref{teo1} and \ref{teo1bis}) in
a more general setting, namely when the interval $[d_{\min},d_{\max})$
is replaced by a more general set
$\cA\subset(0,\infty)$.
More precisely, let ${\lambda}_\ell, {\lambda}_r$ be
bounded nonnegative functions on $(0,\infty]$, and set ${\lambda
}={\lambda}_\ell+{\lambda}_r$. We assume that: 
%
\begin{longlist}[(A2$'$)]
\item[(A1$'$)] ${\lambda}(d)>0$ if and only if $d \in\cA$;

\item[(A2$'$)] if $d,d' \geq\inf(\cA)$, then $d+d'\notin\cA$.
\end{longlist}
Above, $d_{\min}:=\inf(\cA)$ denotes the infimum of the set $\cA
$. When $\cA=
[d_{\min},\allowbreak d_{\max})$, (A1$'$) and (A2$'$) coincide with assumptions
(A1) and
(A2), respectively. A~domain is called active if its length
belongs to $\cA$. The initial distribution~$\cQ$ of the one-epoch
coalescence must be
supported in $[\inf(\cA), \infty)$. In~(\ref{orchestra1})
and~(\ref{orchestra2}) the integration domains become
$[\inf(\cA),\infty)$ and $\cA$, respectively.

The proof of Theorem \ref{teo1} requires the definition of a
universal coupling, that is, the construction on the same
probability space of all one-epoch coalescence processes obtained
by varying the initial configuration. This coupling will be relevant
also in the proof of Theorem \ref{teo4}(ii).

\subsection{Universal coupling for the domain dynamics} \label{secdomaindyn}
In Section \ref{annarella} we have introduced some enumerations of
the points in $ \xi\in\cN$, depending on the property of $\xi$ to
be unbounded both from the left and from the right, or only from the
left. It is convenient here to have a universal enumeration. To
this aim, given $ \xi\in\cN$, we enumerate its points in increasing
order, with the rule that the smallest positive one (if it exists)
gets the label $1$, while the largest nonpositive one (if it exists)
gets the label $0$. We write $N(x, \xi)$ for
the integer number labeling the point $x \in\xi$. This allows to
enumerate the domains of $\xi$ as follows: a domain $[x,x']$ is said
to be the $k$th domain if (i)~$x$ is finite and $N(x, \xi)=k$,
or (ii) $x=-\infty$ and $N(x', \xi)=k+1$. Recall that if
$x=-\infty$, then $\xi$ is unbounded from the left and $x'$ is the
smallest number in~$\xi$.\looseness=-1

We set $\|{\lambda}\|_\infty= \sup_{d \in\mathcal{A}} {\lambda}
(d)$, and we
define ${\lambda}^*_\ell= {\lambda}_r, {\lambda}^*_r ={\lambda
}_\ell$. Obviously ${\lambda} =
{\lambda}_\ell+ {\lambda}_r={\lambda}_\ell^*+{\lambda}_r^*$. This
change of notation should
help the reader. Indeed, in the point dynamics a point $x$ is
erased by the action of its left (right) domain of length $d$ with
rate ${\lambda}_\ell(d)$ (${\lambda}_r(d)$). 
On the other hand, as explained
again below, if we formulate the model in terms of a domain dynamics
then a domain of length $d$
disappears because of the annihilation of its left (right) extreme
with probability rate ${\lambda}_\ell^* (d)$ ($ {\lambda}_r^*(d) $).

We consider now a probability space $({\Omega}, \cF,P)$ on
which the following random objects are defined and are all
independent: the Poisson processes $ \cT^{(k)} = \{ T_m
^{(k)} \dvtx m\in\bbN\}$ and $ \bar\cT^{(k)} = \{\bar T_m
^{(k)} \dvtx m\in\bbN\}$ of parameter $\| {\lambda}\|_\infty$,
indexed by
$k \in\bbZ$, and the random variables $U^{(k)}_m$, $\bar
U^{(k)}_m$, uniformly distributed in $[0,1]$, indexed by $k \in
\bbZ$ and $m \in\bbN$.

Next, given $\z\in\cN(d_{\min})$ and ${\omega}\in{\Omega}$, to
each domain $\Delta$ that belongs to $\z$ we associate the Poisson
process $\cT^{(k)}$
if $\Delta$ is the $k$th domain in $\z$. In this case, we write $\cT
^{(\D)}$
instead of $\cT^{(k)}$. Similarly we define $\bar\cT^{(\D)} $,
$U^{(\D)}_m$, $\bar
U^{(\D)}_m$.
We define $\cW_t[{\omega},\z]$ as
the set of domains $\D$ in $\z$ such that
\begin{eqnarray*}
&&\bigl\{ s\in[0,t] \dvtx
s\in\cT^{(\D)}\cup\bar\cT^{(\D)}, \mbox{ or } s \in
\cT^{(\D')}\cup\bar\cT^{(\D')} \\
&&\hspace*{61.7pt}\mbox{for some domain } \D'
\mbox{ neighboring } \D\bigr\}\not=\varnothing.
\end{eqnarray*}
On $\cW_t[{\omega},\z]$
we define a graph structure putting an edge between domains~$\D$ and
$\D'$ if and only if they are neighboring in $\z$. Since the
function ${\lambda}$ is bounded from above, we deduce that the set
\[
\cB(\z)\!:=\! \{ {\omega} \dvtx\cW_t[{\omega},\z] \mbox{ has
all connected
components of finite cardinality } \forall t \!\geq\!0 \}
\]
has
$P$-probability equal to $1$. Note that the event $ \cB(\z)$
depends on $\z$ only through the infimum and the supremum of the set
$ \{ N(x,\z)\in\bbZ\dvtx x \in\z\}$. By a~simple argument based
on countability, we conclude that
$ P(\cB)=1$, where~$\cB$ is defined as the family of elements
${\omega} \in{\Omega}$ belonging to
$
\bigcap_{\z\in\cN(d_{\min})} \cB(\z)$ and such that all the sets
$\cT^{(k)}[{\omega}]$ and $\bar\cT^{(k)}[{\omega}]$, $k \in\bbZ
$, are
disjoint.

In order to define the path $\{\xi(s)\}_{s\ge0}\equiv\{\xi^\z
(s,{\omega})\}_{s\ge0}$
we first fix a time $t>0$ and define
the path up to time $t$. If ${\omega} \notin\cB$, then we set
\[
\xi(s)=\z\qquad \forall s \in[0,t] .
\]
If ${\omega} \in\cB$, recall
the definition of the graph $\cW_t[{\omega},\z] $. Given a set of domains
$ V$ we write $\bar V$ for the set of the associated extremes, that is,
$x \in\bar V$ if and only if there exists a domain in $V$ having
$x$ as left or right extreme. Moreover, we write $\cV_t[{\omega},\z]
$ for
the set of all domains in $\z$ that do not belong to
$\cW_t[{\omega},\z]$. We define
%
%
\begin{equation}\label{heidi}
\xi(s)\cap\overline{ \cV_t[{\omega},\z]}:= \overline{
\cV_t[{\omega},\z]} \qquad \forall s \in[0,t] ,
\end{equation}
that is, up to time $t$ all points in $\overline{ \cV_t[{\omega},\z
] }$
survive.
Let us now fix a cluster~$\cC$ in
the graph $\cW_t[{\omega},\z]$.
The path $( \xi(s )\cap\bar\cC\dvtx s \in[0,t])$ is
implicitly defined by
the following rules (the definition is well posed since
${\omega} \in\cB$).
If $s \in[0,t]$ equals~$T ^{(\D)}_m$ with $\Delta=[x,x'] \in\cC$
and $x ,x'\in\xi(s-)$, then the ring at time~$T^{(\D)}_m$ is
called \textit{legal} if
%
%
\begin{equation} \label{equniformbis}
U_m^{(\D)} \leq\frac{ {\lambda}_\ell^* ( x'-x)}{\|{\lambda}\|
_\infty},
\end{equation}
and in this case we set $\xi(s) \cap\bar\cC:=(\xi(s-) \cap\bar
\cC) \setminus\{x\}$, otherwise we set $\xi(s) \cap\bar\cC=
\xi(s-) \cap\bar\cC$. In the first case, we say that $x$ is erased
and that the domain $[x,x']$ has incorporated the domain on its
left. Similarly, if $s \in[0,t]$ equals~$\bar T ^{(\D)}_m$ with
$\Delta=[x,x'] \in\cC$ and $x,x' \in\xi(s-)$, then the ring at
time $\bar T^{(\D)}_m$ is called \textit{legal} if
%
%
\begin{equation} \label{equniform2bis}
\bar U_m^{(\D)} \leq\frac{ {\lambda}_r^* ( x'-x)}{\|{\lambda}\|
_\infty},
\end{equation}
and in this case we set $\xi(s) \cap\bar\cC:=(\xi(s-) \cap\bar
\cC) \setminus\{x'\}$, otherwise we set $\xi(s) \cap\bar\cC=
\xi(s-) \cap\bar\cC$. Again, in the first case we say that $x'$
is erased and that the domain $[x,x']$ has incorporated the domain
on its right.

We point out that $\bar\cC\cap\bar\cC'= \varnothing$ if $ \cC$ and
$ \cC'$ are distinct clusters in $\cW_t[{\omega},\z]$. On the other hand,
it could be $\bar\cC\cap\overline{ \cV_t[{\omega},\z]}\not=
\varnothing$. Let $x $ be a point in the intersection, and suppose for
example that $[a,x] \in\cC$ while $[x,b]\in\cV_t[{\omega},\z] $. Then,
by definition of $\cW_t [{\omega}, \z]$, one easily derives that the
Poisson processes associated to the domains $[a,x] $ and $[x,b]$ do
not intersect $[0,t]$, while at least one of the Poisson processes
associated to the domain on the left of $[a,x] $ intersects $[0,t]$.
In particular, $x \in\xi(s) \cap\bar\cC$ for all $s \in[0,t]$,
in agreement with (\ref{heidi}). The same conclusion is reached if
$[a,x]\in\cV_t[{\omega},\z] $ and $[x,b] \in\cC$. This allows to
conclude that the definition of
the path $\{\xi(s)\}_{s\ge0}$ up to time $t$ is well posed. We point
out that this definition
is $t$-dependent. The reader can easily check
that, increasing $t$, the resulting paths coincide on the
intersection of their time domains. Joining these paths together we get
$\{\xi(s)\}_{s\ge0}$.

At this point, it is simple to check that, given a configuration $\z
\in\cN(d_{\min})$, the law of the corresponding
path $\{\xi(s)\}_{s\ge0}$ is that of the one-epoch
coalescence process defined in Section \ref{annarella}
with initial condition $\z$. The advantage of the above construction
is that all one-epoch coalescence processes, obtained by varying
the initial configuration, can be realized on the the same
probability space. Given a probability measure $\cQ$ on $\cN(d_{\min
})$, the one-epoch coalescence process with initial distribution
$\cQ$ can be realized by the random path $\{\xi^{\cdot} (s,\cdot
)\}_{s\ge0}$, defined on the product space ${\Omega} \times\cN
(d_{\min})$, endowed with the probability measure $P\times\cQ$.

\subsection{\texorpdfstring{Proof of Theorem \protect\ref{teo1}\textup{(i)--(iii)}}{Proof of Theorem 2.13(i)--(iii)}}\label{dimteo1i}

Before presenting the proof of Theorem \ref{teo1}(i)--(iii) we
state and prove a key lemma.
\begin{Lemma}[(Separation effect)]\label{sepeff}
For any $x \in\bbR$, any configuration $\z\in\cN( d_{\min} )$
with $x \in\z$, any event $\mathcal{A}$ in the $\s$-algebra
generated by
$\{\xi(s)\cap(-\infty,x) \}_{s\leq t}$, any event $\mathcal{B}$ in
the $\s$-algebra generated by
$\{\xi(s)\cap(x,\infty) \}_{s\leq t} $, it holds
%
%
\begin{eqnarray}\label{dd}
&&\bbP_{\z}\bigl(\mathcal{A}\cap\cB\cap
\{ x\in\xi(t)\}\bigr)\nonumber\\[-8pt]\\[-8pt]
&&\qquad=
\bbP_{\z\cap(-\infty, x]}\bigl(\cA\cap\{ x\in\xi(t)\}\bigr)
\bbP_{\z\cap[x, \infty)}\bigl(\cB\cap\{x\in
\xi(t)\}\bigr) .\nonumber
\end{eqnarray}
\end{Lemma}
\begin{pf} We set
$\z_\ell:= \z\cap(-\infty,x]$, $\z_r := \z\cap[x, \infty)$,
$k:=\cN(x,\z)$, $j:=\mathcal{N}(x,\z_\ell)$ and
$u:=\mathcal{N}(x,\z_r )$.
The desired result (\ref{dd}) is implied by the following facts (i)
and (ii):

\begin{longlist}
\item
For any $\omega\in{\Omega}$ such $x\in\xi^\z(t,\omega)$ the
following holds. At each time $s \in[0,t]$ one has
\[
\xi^{\z_\ell}
(s,\hat\omega)= \xi^{\z}(s, \omega)\cap(-\infty,x],
\]
if $\hat
{\omega}$ satisfies for any $i <k$ and $m \in\bbN$
%
%
\begin{equation}\label{chiaulis1}
\cT^{(i)}({\omega}) =\cT^{(i+j-k)}(\hat{\omega}) ,\qquad
%
U_{m}^{(i+j-k)}({\omega})=U_m^{(i+j-k)}(\hat{\omega}) ,
\end{equation}
and the same identities with $\cT$ and $U_m^{(\cdot)}$ replaced by
$\bar\cT$
and $\bar U_m^{(\cdot)}$. Similarly, at each time $s\in[0,t]$ one has
\[
\xi^{\z_r } (s,\tilde\omega)= \xi^{\z}(s, \omega)\cap[x, \infty),
\]
if
$\tilde{\omega}$ satisfies for any $i \geq k $ and $m \in\bbN$
%
%
\begin{equation}\label{chiaulis2}
\cT^{(i)}({\omega})=\cT^{(i+u-k )}(\tilde{\omega}) ,\qquad
U_{m}^{(i)}({\omega})=U_m^{(i+u-k )} (\tilde{\omega}) ,
\end{equation}
and the same identities with $\cT$ and $U^{(\cdot)}_m$ replaced by
$\bar\cT$
and $\bar U^{(\cdot)}_m$.

\item Take $\hat\omega,\tilde\omega\in{\Omega}$ such that $x\in
\xi^{\z_\ell}(t,\hat
\omega)$ and
$x\in\xi^{\z_r}(t,\tilde\omega)$.
At each time $s \in[0,t]$ it holds
\[
\xi^\z(s,\omega)= \xi^{\z_\ell} (s,\hat\omega)\cup\xi^{\z_r}
(s,\tilde\omega),
\]
if ${\omega}\in{\Omega} $ satisfies (\ref{chiaulis1}) and the
same identities with $\cT$ and $U^{(\cdot)}_m$ replaced by~$\bar\cT
$ and $\bar U^{(\cdot)}$ for any $i< k$ and $m \in\bbN$,
and ${\omega}$ satisfies
(\ref{chiaulis2}) and the same identities with $\cT$ and $U^{(\cdot)}_m$
replaced by $\bar\cT$ and $\bar U^{(\cdot)}_m$ for any $i \geq k$
and $m \in
\bbN$.~\qed
\end{longlist}
\noqed\end{pf}

We first prove the renewal property for the OCP with initial
distribution $\cQ= \operatorname{Ren} (\nu, \mu)$. We take the special
realization of the process defined by means of the universal
coupling at the end of
the previous section. We concentrate on the joint
distribution of the random variables $x_0(t), d_1(t), d_2 (t)$,
proving that they are independent and giving an expression of their
marginal distributions. We recall that $x_0(t)$ is the leftmost point
of $\xi(t)$, while $ d _k (t)$ is the length of the $k$th domain
to the right of $x_0(t)$ in $ \xi(t)$.

While $d_1(t), d_2 (t)$ are nonnegative random variables and their
Laplace transforms are always finite, $x_0(t)$ is a real random
variable and its Laplace transform could diverge. Hence, it is
convenient to work with characteristic functions instead of Laplace
transforms. Given imaginary numbers $s_0,s_1, s_2 \in i \bbR$,
we have
%
%
\begin{eqnarray}\label{schianto}
&&\bbE_{\cQ} \bigl( e^{-s_0 x_0 (t)- s_1 d_1(t)-s_2 d_2 (t)}
\bigr) \nonumber\\
&&\qquad=\sum_{i_0 <i_1<i_2 \in\bbN} \bbE_{\cQ} \bigl(e^{-s_0 x_0
(t)-s_1 d_1(t)-s_2 d_2 (t)}; x_0(t)= x_{i_0}(0);\nonumber\\[-8pt]\\[-8pt]
&&\hspace*{137pt} x_1(t)=
x_{i_1}(0); x_2(t)= x_{i_2} (0)\bigr)
\nonumber\\
&&\qquad= \sum_{i_0 <i_1<i_2 \in\bbN} \int\cQ(d\z) e^{-s_0 x_{i_0}
-s_1 (x_{i_1}-x_{i_0} )-s_2(x_{i_2}-x_{i_1} ) }
f_{i_0,i_1,i_2}(\z) ,\nonumber
\end{eqnarray}
where $\z=\{x_k \dvtx k \in\bbN\}$ and the function
$f_{i_0,i_1,i_2} (\z)$ is defined as the $P$-probability of the
event $\cU$ in ${\Omega}$ given by the elements ${\omega}$
satisfying the
following properties:
\begin{eqnarray*}
\mbox{(P1)}\quad \xi^\z(t,{\omega}) \cap(-\infty, x_{i_0}] &=& \{ x_{i_0}
\} ,\\
\mbox{(P2)}\quad \xi^\z(t,{\omega}) \cap[x_{i_0}, x_{i_1}] &=& \{
x_{i_0},x_{i_1} \} ,\\
\mbox{(P3)}\quad \xi^\z(t,{\omega}) \cap[x_{i_1}, x_{i_2}] &=& \{
x_{i_1},x_{i_2} \} .
\end{eqnarray*}
Let us now set
\begin{eqnarray*}
\z_0&=&\z\cap(-\infty, x_{i_0}] ,\qquad \z_{0,1}=\z\cap[x_{i_0}, x_{i_1}]
, \\
\z_{1,2}  &=&\z\cap[x_{i_1}, x_{i_2}] , \qquad \z_2=\z\cap
[x_{i_2}, \infty) .
\end{eqnarray*}
Then, by the separation effect described in Lemma \ref{sepeff},
one has
%
%
\begin{equation}
\label{rev1}
f_{i_0,i_1,i_2}(\z)= P( \cU(\z) )= \prod_{i=1}^4 P\bigl(
{\omega}\in{\Omega}\dvtx{\omega}
\mbox{ fulfills } \mbox{(P$i'$)}\bigr) ,
\end{equation}
where
\begin{eqnarray*}
&&\mbox{(P1$'$)}\quad \hspace*{5.9pt}\xi^{\z_0 }(t,{\omega}) = \{ x_{i_0} \} ,\\
&&\mbox{(P2$'$)}\quad \xi^{\z_{0,1} }(t,{\omega}) = \{ x_{i_0},x_{i_1} \} ,\\
&&\mbox{(P3$'$)}\quad \xi^{ \z_{1,2}} (t,{\omega}) = \{ x_{i_1},x_{i_2} \}
,\\
&&\mbox{(P4$'$)}\quad \hspace*{32pt}x_{i_2} \in\xi^{ \z_2} (t,{\omega}) .
\end{eqnarray*}

We stress that the factors in (\ref{rev1}) are $\z$-dependent, although
we have omitted $\z$ from the notation. In particular, the
probability $P( {\omega}\in{\Omega}\dvtx{\omega} \mbox{
fulfills } \mbox{(P$i'$)})$ depends on $\z$ only through the first point
$x_0$ and the domain lengths $d_1,d_2,\ldots,\allowbreak d_{i_0}$ if $i=1$, the
domain lengths $d_{i_0+1},\ldots, d_{i_1}$ if $i=2$, the domain
lengths $d_{i_1+1},\ldots, d_{i_2}$ if $i=3$ and the domain lengths
$d_{i_2+1}, d_{i_2+2},\ldots$ if $i=4$. Thinking of $\z$ as a random
configuration sampled by $\cQ$, all the above domain lengths are
i.i.d. with law $\mu$ and are independent from $x_0$ which has law
$\nu$. In particular, the random variables $\z\to P( {\omega}\in
{\Omega}\dvtx{\omega}
\mbox{ fulfills (P$i'$)})$ are independent for $i=1,\ldots, 4$. Using
the consequent factorization and integrating over $\z$ in
(\ref{schianto}), we conclude that
%
%
\begin{eqnarray}\label{evviva}
&&
\bbE_{\cQ} \bigl(e^{-s_0 x_0 (t)- s_1 d_1(t)-s_2 d_2 (t)}
\bigr)\nonumber\\
&&\qquad = \sum_{i_0 <i_1 <i_2 \in\bbN} \int\cQ(d\z)
P\bigl( {\omega}\in{\Omega}\dvtx{\omega} \mbox{ fulfills
(P4$'$)}\bigr)\nonumber\\
&&\hspace*{41.6pt}\qquad\quad{} \times
\int\cQ(d\z) e^{-s_0 x_{i_0}} P\bigl( {\omega}\in{\Omega}\dvtx
{\omega} \mbox{
fulfills
(P1$'$)}\bigr)\\
&&\hspace*{41.6pt}\qquad\quad{} \times\int\cQ(d\z)
e^{-s_1(x_{i_1}-x_{i_0}) } P\bigl( {\omega}\in{\Omega}\dvtx{\omega}
\mbox{ fulfills
(P2$'$)}\bigr) \nonumber\\
&&\hspace*{41.6pt}\qquad\quad{} \times\int\cQ(d\z)
e^{-s_2(x_{i_2}-x_{i_1}) } P\bigl( {\omega}\in{\Omega}\dvtx{\omega}
\mbox{ fulfills
(P3$'$)}\bigr) .\nonumber
\end{eqnarray}
By simple computations and using that $\cQ= \operatorname{Ren}(\nu,\mu)$,
from the above identity we derive that
%
%
\begin{equation}\label{bellalatorta}
\bbE_{\cQ} \bigl( e^{-s_0 x_0 (t)- s_1 d_1(t)-s_2 d_2 (t)}
\bigr)=\hat L_t (s_0) \hat G_t (s_1) \hat G_t (s_2),
\end{equation}
where
%
%
\begin{eqnarray}
\label{gormito1zac}
\hat L_t (s)&=& \bbP_{\operatorname{Ren} ({\delta}_0,\mu) } \bigl( 0 \in\xi
(t) \bigr)L_0(s)
\sum_{n\geq0}
\bbE_{\otimes_n \mu} \bigl( e^{- s x_n (0) }; \xi(t)= \{ x_n(0)
\}
\bigr) ,\hspace*{-40pt}
\\
\label{gormito2zac}
\hat G_t(s)&=&\sum_{n\geq1} \bbE_{\otimes_n \mu} \bigl( e^{- s x_n
(0) } ; \xi(t)=\{ 0,x_n(0)
\}\bigr) .\hspace*{-40pt}
\end{eqnarray}
Above $L_0(s)$ denotes the characteristic function of $\nu$, while
$\otimes_n \mu$ denotes the law of the SPP given by $n+1$ points
$0=x_0<x_1<\cdots<x_n$ such that the random variables
$d_i=x_{i}-x_{i-1}$, $1\leq i \leq n$, are i.i.d. with common law
$\mu$.

Note that in the derivation of (\ref{gormito1zac}) one has to
keep the contribution of both the first and the second expectation
in the right-hand side of (\ref{evviva}).

By similar arguments, one obtains
%
%
\begin{equation}\qquad
\bbE_{\cQ} \bigl( e^{-( s_0 x_0 (t) + s_1 d_1(t)+\cdots+ s_k
d_k(t)) }\bigr)=\hat L_t (s_0)\prod_{i=1}^k \hat G_t(s_i)\qquad
\forall k \geq0
\end{equation}
with the convention that the last product over $k$ is equal to $1$
if $k=0$. The above formula implies that the random variables $x_0
(t), d_1(t), d_2(t), \ldots$ are all independent, $x_0(t)$ has
characteristic function $\hat L_t$ and $d_k(t)$ has characteristic
function $\hat G_t$
for each $k \geq1$. 
Note that the above arguments remain valid for $s_0,s_1,\ldots, s_k
\geq0$ (and one speaks of Laplace transforms instead of
characteristic functions), but if $\bbE( e^{-s_0 x_0(0) })=\infty$
we get the trivial identities $\infty=\infty$.\looseness=-1

(ii)--(iii) We consider now the case $\cQ= \operatorname{Ren}(\mu)$. Points
are now labeled in
increasing order with the convention that $x_0$ denotes the largest
nonpositive point.
Similarly to the
above proof, one can show that the random variables $d_k (t)$,
$k\not=1$, are i.i.d. and are independent from the random variable
$x_1(t)-x_0(t)$. Moreover, their common law has Laplace transform
(\ref{gormito2zac}). On the other hand, due to the definition of
the dynamics, $\xi_t$ must be a stationary SPP. As a byproduct, we
conclude that the law of $\xi_t$ is $\operatorname{Ren} (\mu_t)$,
$\mu_t$ being a probability measure on $(0, \infty)$ with Laplace
transform (\ref{gormito2zac}). The case $\cQ=\operatorname{Ren}_\bbZ(\mu)$
can be treated analogously.


It is convenient to isolate a technical fact derived in the above
proof, which will be the starting point in the proof of Theorem
\ref{teo1bis}:
\begin{Lemma}\label{argentina}
Recall that $G_t(s)= \int_{[\inf(\cA),\infty)} \hspace*{-0.6pt}e^{-sx}\mu_t(x) $
and
$L_t(s)=\break \int_{\cA} e^{-s x} \hspace*{-0.6pt}\times$ $\nu_t (x) $ $(s\in\bbR_+\cup i\bbR)$.
Then
%
%
\begin{eqnarray}
\label{gormito1}
L_t (s)&=& \bbP_{\operatorname{Ren} ({\delta}_0,\mu) } \bigl( 0 \in\xi(t) \bigr)L_0(s)
\sum_{n\geq0}
\bbE_{\otimes_{n}\mu} \bigl( e^{- s x_n (0) } ; \xi(t)= \{ x_n(0)
\}
\bigr) ,\hspace*{-40pt}
\\
\label{gormito2}
G_t(s)&=&\sum_{n\geq1} \bbE_{\otimes_{n}\mu} \bigl( e^{- s x_n (0)
} ; \xi(t)=\{ 0,x_n(0)
\}\bigr) ,\hspace*{-40pt}
\end{eqnarray}
where $\otimes_{n}\mu$ denotes the law of the SPP given
by $n+1$ points $0=x_0<x_1<\cdots<x_n$ such that the random
variables $d_i=x_{i}-x_{i-1}$, $1\leq i \leq n$, are i.i.d. with
common law $\mu$.
\end{Lemma}

\subsection{\texorpdfstring{Proof of Theorem \protect\ref{teo1}\textup{(iv)}}{Proof of Theorem 2.13(iv)}} Suppose that $\cQ=
\operatorname{Ren}({\delta}_0,\mu)$. Then we can write
%
%
\begin{eqnarray}\label{risoinbianco}
&&\bbE_{\cQ} \bigl( e^{- s d_1 (t) } ; 0\in\xi(t) \bigr)\nonumber
\\
&&\qquad= \sum_{i
\in\bbN_+} \bbE_{\cQ}\bigl( e^{-s x_i(0) } ; 0 \in\xi(t) ;
x_1(t)= x_i (0) \bigr)\\
&&\qquad=\sum_{i\in\bbN_+} \int\cQ(d \z) e^{-s
x_i} \bbP_\z\bigl( \xi(t)\cap[0, x_i]=\{0,x_i\} \bigr) ,
\nonumber
\end{eqnarray}
where $\z=\{x_k \dvtx k \geq0 \}$.
By the separation
effect described in Lemma \ref{sepeff},
we can write the last probability inside the integrand in
(\ref{risoinbianco}) as
%
\[
\bbP_{\z\cap[0, x_i]} \bigl( \xi(t) =\{0, x_i\} \bigr) \bbP_{ \z
\cap[x_i, \infty)} \bigl( x_i \in\xi(t)\bigr) .
\]
We observe that the last two factors, as functions of $\z$, are
$\cQ$-independent. Moreover, for all $i\in\bbN_+$, it holds
\[
\int\cQ( d\z) \bbP_{ \z\cap[x_i, \infty)} \bigl( x_i \in\xi
(t)\bigr)
= \bbP_{\cQ} \bigl(0\in\xi(t)\bigr) .
\]
Therefore, coming back to
(\ref{risoinbianco}), using the renewal property of $\cQ$ and
(\ref{gormito2}), we get
\begin{eqnarray*}
&&\bbE_{\cQ} \bigl( e^{- s d_1 (t) } \mid 0\in\xi(t) \bigr)\\
&&\qquad=\sum
_{i \in\bbN_+}\int\cQ(d \z) \bbP_{\z\cap[0, x_i]} \bigl( \xi(t)
=\{0, x_i\} \bigr) = G_t (s) .
\end{eqnarray*}
By similar arguments, one gets
\[
\bbE_\cQ\bigl(
e^{-\sum_{j=1}^k s_j d_j(t) } \mid 0\in\xi(t) \bigr)= \prod_{j=1}^k
G_t (s_j) ,\qquad s_1,\ldots, s_k \in\bbR_+\cup i\bbR,
\]
thus concluding the proof of Theorem \ref{teo1}(ii).

\subsection{\texorpdfstring{Proof of Theorem \protect\ref{teo1}\textup{(v)}}{Proof of Theorem 2.13(v)}}

From (\ref{gormito1}) and (\ref{gormito2}) we get that $L_t(s)$ and
$G_t(s)$ converge to $L_\infty(s)$ and $G_\infty(s)$ as $t \to
\infty$. This implies the weak convergence to $\nu_t$ and $\mu_t$ to
$\nu_\infty$ and $\mu_\infty$.

\section{\texorpdfstring{Recursive identities in the OCP: Proof of Theorem \protect\ref{teo1bis}}{Recursive identities in the OCP: Proof of Theorem 2.14}}\label{brunatecchio}
The proof is based on the
identities (\ref{gormito1}) and (\ref{gormito2}) in Lemma
\ref{argentina}. We first point out a blocking phenomenon in the
dynamics that will be frequently used in what follows.\vadjust{\goodbreak} Due to
assumption 
(A1$'$), a separation point $x$ between two inactive
domains cannot be erased. As simple consequence, we obtain that the
points between two nearest neighbor inactive domains cannot all be
erased: if 
there exists $s\geq0$ s.t. $[a,b]$ and $[c,d]$ are inactive domains
(including the
cases $a=-\infty$, $d=\infty$) with $b\leq c$, then $\xi(\infty)
\cap[b,c]\not= \varnothing$.
Indeed the set $[b,c]\cap\xi$ is nonempty (since $b$ and $c$ belongs
to it) and if we assume that all points in this set are killed, then
the last one to be killed is for sure a~separation point between two
inactive domains and a contradiction arises.
We will frequently use this fact below.

By Lemma \ref{argentina} we can write, for $s \in\bbR_+\cup i \bbR$,
%
%
\begin{eqnarray}\label{colazione1}
G_\infty
(s) &=& \sum_{k=0}^\infty A_k (s) ,\nonumber\\[-8pt]\\[-8pt]
A_k(s) &=&\bbE_{\otimes
_{k+1}\mu} \bigl( e^{- s x_{k+1} (0) } ; \xi(\infty)=\{
0,x_{k+1}(0) \}\bigr) .\nonumber
\end{eqnarray}
We explicitly compute $A_k(s)$. To this aim we consider the
one-epoch coalescence process with law
$\bbP_{\otimes_{k+1}\mu} $.
We observe that, due to the blocking phenomenon, the
event $\xi(\infty)=\{ 0, x_{k+1}(0)\}$ implies that (i) $k\geq1$,
and the $k+1$ initial domains are all active, or (ii) $k\geq0$, and
initially there are $k $ active domains and one inactive domain.
Therefore, given $k \geq0$ and $1\leq j \leq k+1$, we introduce the
following events:
\begin{eqnarray*}
F_k &=&\{ d_1 (0),d_2(0),\ldots, d_{k+1} (0) \in\cA\}
\cap\bigl\{ \xi(\infty)=\{0, x_{k+1}(0)\} \bigr\} ,\\
E_{k, j} & =&\bigl\{ d_i(0) \in\cA\mbox{ }\forall i \in\{1,\ldots,
k+1\} \setminus\{j\}\bigr\}
\cap\{ d_j(0)\notin\cA\}\\
&&{}\cap\bigl\{ \xi(\infty)= \{0, x_{k+1}(0)\} \bigr\} .
\end{eqnarray*}
By the above discussion, it holds
%
%
\begin{equation}\label{colazione2}\qquad A_k(s)= \bbE_{\otimes_{k+1}\mu}
\bigl( e^{- s x_{k+1} (0) } ; F_k
\bigr)\mathbh{1}_{k \geq1}+ \sum_{j=1}^{k+1} \bbE_{\otimes
_{k+1}\mu} \bigl( e^{- s x_{k+1} (0) } ; E_{k,j} \bigr) .
\end{equation}
The exact computation of the two addenda in the right-hand side is given in
the following lemmas:
\begin{Lemma}\label{pesante1} For each $k \geq1$, it holds
%
%
\begin{equation} \bbE_{\otimes_{k+1}\mu} \bigl( e^{- s x_{k+1} (0)
} ; F_k \bigr)
= \frac{
[\int\mu(dx) e^{-sx } \mathbh{1}_{x \in\cA}]^{k+1} }{
(k+1)\cdot(k-1)!} .
\end{equation}
\end{Lemma}
\begin{Lemma}\label{pesante2} For each $k \geq0$, it holds
%
%
\begin{eqnarray}
&& \sum_{j=1}^{k+1} \bbE_{\otimes_{k+1}\mu} \bigl(
e^{- s x_{k+1} (0) } ; E_{k,j} \bigr)\nonumber\\[-8pt]\\[-8pt]
&&\qquad= \int\mu(dx) e^{-s x}
\mathbh{1}_{x \notin\cA} \frac{ [\int\mu(dx) e^{-sx }
\mathbh{1}_{x \in\cA}]^k }{ k!} .\nonumber
\end{eqnarray}
\end{Lemma}

We postpone the proof of the lemmas in order to end the proof of
point~(i) of Theorem \ref{teo1bis}. Due to (\ref{colazione1}),
(\ref{colazione2}), Lemmas \ref{pesante1} and \ref{pesante2}
we obtain
%
%
\begin{eqnarray}\label{cassanata}
G_\infty(s)
& = &
\sum_{k=1}^\infty\frac{ [\int\mu(dx) e^{-sx } \mathbh
{1}_{x \in\cA}]^{k+1} }{ (k+1) \cdot(k-1)!} \nonumber\\
&&{} +
\sum_{k=0}^\infty\int\mu(dx) e^{-s x} \mathbh{1}_{x \notin\cA}
\frac{ [\int\mu(dx) e^{-sx } \mathbh{1}_{x
\in\cA}]^k }{ k!} \nonumber\\
& = &
\sum_{k=1} ^ \infty\frac{ H_0(s) ^{k+1} }{(k+1) \cdot(k-1)!}+ \sum
_{k=0} ^\infty\bigl(
G_0(s)-H_0(s) \bigr) \frac{H_0(s)^k}{ k!}\\
& = &
- H_0(s)- \sum_{j=2}^\infty\biggl[ \frac{1}{(j-1)!}- \frac{1}{j
\cdot(j-2)!}\biggr] H_0(s)^j + G_0(s) e^{H_0(s)}\nonumber\\
&=&
- \sum_{j=1}^\infty\frac{ H_0(s)^j}{j!} + G_0(s) e^{H_0(s)} =
1-e^{H_0(s)}+ G_0(s) e^{H_0(s)}
.\nonumber
\end{eqnarray}
This concludes the proof of (\ref{pietra}) (and hence of
point (i) of Theorem \ref{teo1bis}).

Now we give the proofs of Lemmas \ref{pesante1} and \ref{pesante2}.
\begin{pf*}{Proof of Lemma \ref{pesante1}}
From now on we work with the one-epoch coalescence process whose
initial distribution is given by $\otimes_{k+1} \mu$.

Let us suppose that $d_1(0), d_2(0) ,\ldots, d_{k+1}(0) \in\cA$:
we want to understand how the event $F_k$ takes place, that is, how
points $x_1(0),\ldots, x_k(0)$ are erased while $x_0(0) =0$ and
$x_{k+1}(0)$ survive. The event
$F_k$ must be realized as follows:
\begin{longlist}
\item the first erased point must be of the
form $x_{i} (0)$ with $1\leq i \leq k$;
\item after the disappearance of $x_i (0)$, restricting the
observation on
the left of $ x_i(0) $, one sees that $x_{i-1}(0), x_{i-2} (0),\ldots
,x_1(0)$ disappear one after the other, from the rightmost
point to the leftmost point;
\item after the disappearance of $x_i (0)$, restricting the
observation on the right
of $x_i(0)$, one sees that $x_{i+1}(0), x_{i+2} (0),\ldots,x_k(0)$
disappear one after the other, from the leftmost point to the
rightmost point.
\end{longlist}
%
(ii) and (iii) follow from the blocking phenomenon and the fact that
the disappearance of $x_i(0)$ creates an inactive domain,
$[x_{i-1}(0),x_{i+1}(0)]$. Since the initial configuration has a
finite number of points, the coalescence process can be realized as
follows:
each domain of
initial length $d$ waits independently from the other domains an
exponential time of parameter
${\lambda}(d)$, afterwards
if both the its extremes are still present we say that the ring is
effective and
with probability ${\lambda}_r(d)/{\lambda} (d)$ its left
extreme is erased otherwise the right extreme is erased, and after
this jump the dynamics start afresh. We can therefore describe the
jumps in the coalescence process (disregarding the jump times) by a
string $\s=(\s_1, \s_2,\ldots, \s_m)$, where each entry $\s_i$ is a
couple $\s_i=(N_i,L_i) $ with $N_i \in\{1,2,\ldots, k+1\}$,
$N_i\neq N_j$ for $i\neq j$ and $L_i
\in\{\ell, r\}$ ($N$ stands for ``number'' and $L$ stands for
``letter''). The meaning of $\s_i$ is the following:
the domain which rings at the $i$th effective ring is given by
$[x_{N_i-1}(0), x_{N_i}(0) ]$, while after its
ring the erased extreme is the left one if $L_i=\ell$ or the right
one if $L_i=r$. See Figure \ref{figcascata} for an example. We
say that the number $N_i$ is associated to the letter $L_i$. Given
such a string $\s$ we denote by $\cB(\s)$ the event that the jumps
of the coalescence process are indeed described by the string $\s$
in the sense specified above.

%
%
\begin{figure}

\includegraphics{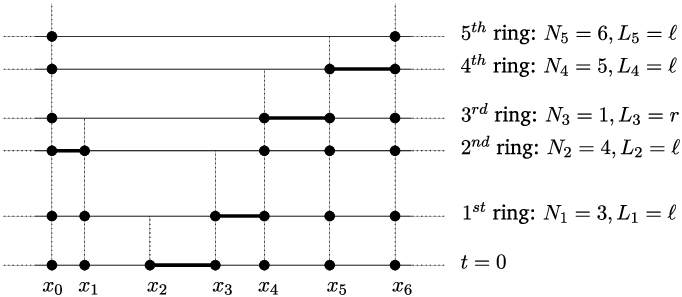}

\caption{Example of a trajectory in $F_k$, with $k=5$.}
\label{figcascata}
\end{figure}


Due to our previous considerations it holds
\[
F_k = \bigcup_{\s\ \mathrm{admissible}} \cB(\s),
\]
where a string
$\s=(\s_1, \s_2,\ldots, \s_m)$ is called \textit{admissible} if the
following properties are satisfied:
\begin{longlist}[(P2)]
\item[(P1)] if $L_1=\ell$, then $N_1 \in[2,k+1]$;
the numbers $N_i $ associated to the letter
$\ell$ are all the integers in $[N_1,k+1]$, and they appear in the
string in increasing order;
the numbers $N_i $ associated to the letter $r$ are all the positive
integers in $[1,N_1-2]$, and they appear in
the string in decreasing order;
\item[(P2)] if $L_1=r$, then $N_1 \in[1, k]$,
the numbers $N_i $ associated to the letter~$\ell$ are all the integers in $[N_1+2,k+1]$, and they appear in the
string in increasing order; the numbers $N_i $ associated to the
letter $r$ are all the integers in $[1,N_1]$, and they appear in the
string in decreasing order.
\end{longlist}
Observe that an admissible string must have $k$ entries, that is,
$m=k$, and that
the knowledge of $(L_i)_{1\leq i \leq k } $ allows to
determine uniquely the numbers $(N_i)_{1\leq i \leq k}$.

Recall that $ {\lambda}^*_\ell(d)= {\lambda} _r(d)
$, ${\lambda}_r^* (d ) ={\lambda} _\ell(d) $. Writing $d_i(0)$ as
$d_i$ (for
simplicity of notation), if $\s$ is admissible we get
%
%
\begin{equation} \bbE_{\otimes_{k+1} \mu}\bigl[ e^{- s d_1
(\infty) } ; \cB(\s) \bigr]= \bbE_{\otimes_{k+1}
\mu}[ F(d_1, d_2,\ldots, d_{k+1} ,\s)] ,
\end{equation}
where
%
%
\begin{eqnarray}
&&
F(d_1,d_2,\ldots, d_{k+1}, \s)\nonumber\\
&&\qquad=\Biggl(\prod_{i=1}^{k+1} e ^{- s d_i }\mathbh{1}_{ d_i \in\cA}\Biggr)
\frac{{\lambda} (d_{N_1}) }{{\lambda} (d_1)+
\cdots+ {\lambda}(d_{k+1}) }\nonumber\\[-8pt]\\[-8pt]
&&\qquad\quad{}\times\frac{ {\lambda}^*_{L_1}
(d_{N_1})}{{\lambda} (d_{N_1}) }
\prod_{i=2} ^k \frac{ {\lambda}( d_{N_i}) }{ \sum_{j=i}^k {\lambda
}( d_{N_j})
}\frac{ {\lambda} ^*_{L_i}(d_{N_i }) }{{\lambda}( d_{N_i} ) }\nonumber\\
&&\qquad=\Biggl(\prod_{i=1}^{k+1} e ^{- s d_i }\mathbh{1}_{d_i \in\cA}\Biggr)
\frac{{\lambda}^*_{L_1} (d_{N_1}) }{{\lambda} (d_1)+
\cdots+ {\lambda}(d_{k+1}) } \prod_{i=2} ^{k} \frac{ {\lambda
}^*_{L_i}( d_{N_i})
}{ \sum_{j=i}^k {\lambda}( d_{N_j}) }
\nonumber
\end{eqnarray}
(the last factor is defined as $1$ if $k=1$).

Observe that the
law $\otimes_{k+1} \mu$ is exchangeable, that is, it is left
invariant by permutations of $d_1, d_2,\ldots, d_{k+1}$. This
symmetry leads to the identity
\[
\bbE_{\otimes_{k+1} \mu}[ F(d_1, d_2,\ldots, d_{k+1}
,\s)]=\bbE_{\otimes_{k+1} \mu}[ G(d_1, d_2,\ldots, d_{k+1} ,
(L_i)_{1\leq i \leq k} )],
\]
where
\begin{eqnarray*}
&&G(d_1, d_2,\ldots, d_{k+1} , (L_i)_{1\leq i \leq k} )\\
&&\qquad=
\Biggl(\prod_{i=1}^{k+1} e ^{- s d_i }\mathbh{1}_{d_i \in\cA}\Biggr)
\frac{{\lambda}^*_{L_1} (d_1) }{{\lambda} (d_1)+
\cdots+ {\lambda}(d_{k+1}) } \prod_{i=2} ^{k} \frac{ {\lambda
}^*_{L_i}( d_i) }{
\sum_{j=i}^k {\lambda}( d_j) } .
\end{eqnarray*}
Recall that an admissible string $\s$ is uniquely determined by
its letter string $(L_i)_{1\leq i \leq k }$, and observe that
each string in $\{\ell, r \}^{[1, k]}$ is the letter string
$(L_i)_{1\leq i \leq k }$ for some admissible $\s$. Therefore we
have
%
%
\begin{eqnarray}\label{malika}
&&\bbE_{\otimes_{k+1} \mu}\bigl[ e^{- s d_1 (\infty) } ;
F_k \bigr] \nonumber\\
&&\qquad = \sum_{\s\ \mathrm{admissible }} \bbE_{\otimes_{k+1}
\mu}[ F(d_1, d_2,\ldots, d_{k+1}, \s)] \nonumber\\[-8pt]\\[-8pt]
&&\qquad =
\sum_{L_1,\ldots, L_k \in\{\ell, r\} }
\bbE_{\otimes_{k+1} \mu}[ G(d_1, d_2,\ldots, d_{k+1},
(L_i)_{1\leq i \leq k} )] \nonumber\\
&&\qquad = \bbE_{\otimes_{k+1}
\mu}[ H(d_1, d_2,\ldots, d_{k+1} )],\nonumber
\end{eqnarray}
where
\[
H(d_1, d_2,\ldots, d_{k+1}) = \Biggl(\prod_{i=1}^{k+1} e ^{- s d_i
}\mathbh{1}_{ d_i \in\cA} \Biggr)
\frac{{\lambda} (d_1) }{{\lambda} (d_1)+
\cdots+ {\lambda}(d_{k+1}) } \prod_{i=2} ^{k} \frac{ {\lambda}(
d_i) }{ \sum
_{j=i}^k {\lambda}( d_j) } .
\]
Applying Lemma \ref{lemexchangeable} in Appendix \ref{apppp} with $k+1$
instead of $k$, $m=\mu$, $f(x)=e ^{- s x }\mathbh{1}_{ x \in\cA} $
and $g(x)={\lambda}(x)$, we end up with
%
%
\begin{equation}\label{aiane}
\bbE_{\otimes_{k+1} \mu} [H(d_1,\ldots, d_{k+1})]=
\frac{1}{(k+1)\cdot(k-1)!} \biggl[ \int\mu(dx) e^{-sx
}\mathbh{1}_{ x \in\cA} \biggr]^{k+1}.\hspace*{-35pt}\vspace*{-2pt}
\end{equation}
This ends the proof of Lemma \ref{pesante1}.\vspace*{-3pt}
%
\end{pf*}
\begin{pf*}{Proof of Lemma \ref{pesante2}}
The proof follows the main arguments in the proof of Lemma
\ref{pesante1}, hence we skip some details.
As in the proof of
Lemma~\ref{pesante1} we work with the one-epoch coalescence process
with law $\bbP_{ \otimes_{k+1} \mu}$.

Denoting the jumps of the coalescence process
(disregarding the jump times) with the same rule used in the proof
of Lemma \ref{pesante1}, that is, by means of the string~$\s$, we get
that
%
%
\begin{equation}\label{sonora} E_{k,j}=\bigcup_{\s\
j\mbox{-}\mathrm{admissible} } \cB(\s),\vspace*{-2pt}
\end{equation}
where now \textit{$j$-admissible} means that the numbers $N_i$
associated to the letter~$\ell$ appear in the string $\s$ in
increasing order from $j+1$ to $k+1$, while the numbers $N_i$
associated to the letter $r$ appear in the string $\s$ in decreasing
order from $j-1$ to~$1$. Note that in particular $\s$ contains $j-1$
letters ``$r$'' and ``$k+1-j$'' letters $\ell$, and therefore $\s$ has
length $k$.

As in the previous proof we set $d_r=d_r(0)$. We then compute the
expectation
%
%
\begin{eqnarray}\label{fornaciari}
&&\bbE_{\otimes_{k+1} \mu} \bigl[ e^{-s d_1 (\infty)} ;
\cB(\s) \bigr]\nonumber\\[-3pt]
&&\qquad= \bbE_{\otimes_{k+1} \mu}
\Biggl[ e^{-s d_j }\mathbh{1} _{d_j\notin\cA} \prod_{i=1}^k
\biggl\{ e^{-s d_{N_i}}
\mathbh{1}_{d_{N_i} \in\cA}\frac{{\lambda} ^*_{L_i}(d_{N_i} )
}{\sum_{r=i}^k {\lambda}
(d_{N_r} )}\biggr\}\Biggr] \nonumber\\[-9.5pt]\\[-9.5pt]
&&\qquad=\bbE_{\otimes_{k+1} \mu}
\Biggl[ e^{-s d_{k+1} }\mathbh{1} _{d_{k+1} \notin\cA} \prod
_{i=1}^k \biggl\{ e^{-s d_{i}}
\mathbh{1}_{d_i \in\cA} \frac{{\lambda} ^*_{L_i}(d_i ) }{\sum
_{r=i}^k {\lambda}
(d_r
)}\biggr\}\Biggr]\nonumber\\[-3pt]
&&\qquad= \biggl(\int e^{-s x} \mathbh{1}_{x \notin\cA} \mu(dx)\biggr)
\bbE_{\otimes_{k} \mu}
\Biggl[ \prod_{i=1}^k \biggl\{ e^{-s d_{i}}
\mathbh{1}_{d_i \in\cA}\frac{{\lambda} ^*_{L_i}(d_i ) }{\sum
_{r=i}^k {\lambda}
(d_r )}\biggr\}\Biggr] ,
\nonumber\vspace*{-2pt}
\end{eqnarray}
where in the second identity we have used the exchangeability of
$\otimes_{ k+1} \mu$, and in the third identity we have simply
factorized the probability measure.

Summing over $j$ allows us to remove the constraint that $\s$ must
have
$j-1$ letters ``$r$'' and ``$k+1-j$'' letters $\ell$, hence
%
%
\begin{eqnarray}\qquad
&&\sum_{j=1}^{k+1} \bbE_{\otimes_{k+1} \mu} \bigl[ e^{-s
d_1 (\infty)} ; E_{k,j}\bigr]\nonumber\\[-2pt]
&&\qquad= \sum_{L_1,\ldots, L_k \in\{\ell,
r\} }\mbox{ r.h.s. of  (\ref{fornaciari})}\\[-2pt]
&&\qquad=\biggl(\int e^{-s x}
\mathbh{1}_{x \notin\cA} \mu(dx)\biggr) \bbE_{\otimes_k \mu}
\Biggl[ \prod_{i=1}^k \biggl\{ e^{-s d_{i}}
\mathbh{1}_{d_i \in\cA} \frac{{\lambda} (d_i ) }{\sum_{r=i}^k
{\lambda} (d_r
)}\biggr\}\Biggr] .
\nonumber
\end{eqnarray}
Applying point (b) of Lemma \ref{lemexchangeable} [with $f(x)=
e^{-s x} \mathbh{1}_{x \in\cA}$ and $g(x)={\lambda}(x)$] completes
the proof
of Lemma \ref{pesante2}.
\end{pf*}

\subsection{\texorpdfstring{Proof of Theorem \protect\ref{teo1bis}\textup{(ii)}}{Proof of Theorem 2.14(ii)}}
The proof of point 
(ii)(a) is trivial, since ${\lambda}_r \equiv0$,
then $x_0(t)=x_0(0)$ for any time $t \geq0$. Indeed the
first point $x_0(t)$ of~$\xi(t)$ cannot be erased from the left due
the infinite domain, and from the right due to the assumption ${\lambda}_r
\equiv0$.

We now concentrate on point
(ii)(b).
Due to (\ref{gormito1}), we can
write (for $s\in\bbR_+ \cup i\bbR$)
%
%
\begin{equation}\label{cena1}
L_\infty(s)= \bbP_{\operatorname{Ren} ({\delta}_0,\nu) } \bigl( 0 \in\xi
(\infty) \bigr) L_0(s) \sum_{k=
0}^\infty
B_k(s) ,
\end{equation}
where
$ B_k(s)
= \bbE_{\otimes_{k}\mu} ( e^{- s x_k (0)
} ; \xi(\infty)= \{ x_k(0) \} )$.
\begin{Lemma}\label{lavatrice}
$B_0(s)= 1$ while, for any $k \geq1$, it holds
%
%
\begin{equation}\label{mostra}
B_k(s)=\mathbb{E}_{\otimes_{k} \mu} \Biggl( \prod_{i=1}^k
e^{-sd_i} \frac{{\lambda}_\ell^*(d_i) \mathbh{1}_{d_i \in
\mathcal{A}}}{\sum_{j=i}^k {\lambda}(d_j)} \Biggr) .
\end{equation}
\end{Lemma}
\begin{pf}
We work with the one-epoch coalescence process with law
$\bbP_{\otimes_{k}\mu}$. The case $k=0$ is trivial. We
take $k \geq1$. Due to the blocking phenomenon,
there is only one possible way to realize the event $\{\xi(\infty)= x_k
(0)\}$:
only the points $x_0(0),x_1(0),\ldots,x_{k-1}(0)$ must disappear, one
after the other from the left to the right. Setting $d_i= d_i(0)$,
this implies that $d_1,\ldots, d_k$ belong to $\cA$. In this case,
knowing $\xi(0)$, the above event has probability
\[
\frac{{\lambda}_\ell^*(d_1)}{\sum_{j=1}^k {\lambda}(d_j)} \times
\frac{{\lambda}_\ell^*(d_2)}{\sum_{j=2}^k {\lambda}(d_j)} \times
\cdots\times
\frac{{\lambda}_\ell^*(d_k)}{ {\lambda}(d_k)} = \prod_{i=1}^k
\frac{{\lambda}_\ell^*(d_i)}{\sum_{j=i}^k {\lambda}(d_j)} .
\]
Since $x_k(0)=d_1+d_2+\cdots+d_k$, we get (\ref{mostra}).
\end{pf}

Since $\lambda_\ell= \gamma
\lambda_r$ we have ${\lambda}_r^*=\gamma{\lambda}_\ell^*$. In particular
${\lambda}={\lambda}_\ell^*+{\lambda}_r^*=(1+\gamma){\lambda
}_\ell^*$. Hence, due to (\ref{cena1}) and
Lemma \ref{lavatrice}, we get
\[
L_\infty(s) = C L_0(s) \sum_{k=0}^\infty
\frac{1}{(1+\gamma)^k}\bbE_{\otimes_{k} \mu} \Biggl(
\prod_{i=1}^k e^{-sd_i} \frac{\lambda(d_i)\mathbh{1}_{d_i \in
\mathcal{A}}}{\sum_{j=i}^k {\lambda}(d_j)} \Biggr) ,
\]
where $C:= \bbP_{\operatorname{Ren} ({\delta}_0,\nu) } ( 0 \in\xi(\infty
) )$ and
where, in the last series, the addendum with $k=0$ is defined as
$1$. Applying point (b) of Lemma \ref{lemexchangeable} [with
$f(x)= e^{-s x} \mathbh{1}_{x \in\mathcal{A}}$ and $g(x)={\lambda}(x)$],
and recalling that $H_0(s)= \int e^{-sx} \mathbh{1}_{x \in
\mathcal{A}} \mu(dx)$, we end up with
\[
L_\infty(s)
=
C L_0(s) \sum_{k=0}^\infty\frac{H_0(s)^k}{(1+\gamma)^k \cdot k!} =
C L_0(s) \exp\biggl\{ \frac{H_0(s)}{1+\gamma} \biggr\} .
\]
Since $L_0(0)=L_\infty(0)=1$, the latter identity applied to $s=0$
leads to $C=\exp\{- \frac{H_0(0)}{1+\gamma} \}$
which in turn leads to (\ref{pietrabis}).
Then (\ref{troppacocacola}) follows immediately by noticing that
$\bbP_{\operatorname{Ren} (\nu,\mu)}(x_0(0)\in\xi(\infty))=\bbP
_{\operatorname{Ren}(\delta_0,\mu)}(0\in\xi(\infty))$ and from the
above definition
of $C$.


\section{\texorpdfstring{Analysis of the recursive identity (\protect\ref{pietra}) in OCP}{Analysis of the recursive identity (2.8) in OCP}} \label{gattomagico}

As mentioned in the \hyperref[intro]{Introduction}, a crucial tool to prove Theorem
\ref{teo2} is given by a special integral representation of certain
Laplace transforms, which makes identities~(\ref{pietra}) and
(\ref{pietrabis}) finally treatable. We first consider
(\ref{pietra}), focusing our attention on the one-epoch
coalescence process in the same setting of Section \ref{annarella}
(i.e., the active domains have length in $[d_{\min},d_{\max})$).
In what
follows, we present an overview of the global scheme, postponing
proofs to the end of the section. It is convenient to work with
rescaled random variables. More precisely, in the same setting of
Theorem \ref{teo1}, we call $X_0, X_\infty$ some generic random
variables with law $\mu, \mu_\infty$, respectively. Then we define
\[
Z_0=X_0/d_{\min} \quad\mbox{and}\quad Z_\infty=X_\infty/d_{\max}
\]
as the rescaled random
variables. Setting for $s>0$
%
%
\begin{eqnarray}\label{ridono}
g_0(s) &=& \bbE( e^{-s Z_0}) ,\qquad g_\infty(s)=\bbE
(e^{- s Z_\infty} ) ,\nonumber\\[-8pt]\\[-8pt]
h_0 (s) &=& \bbE(
e^{-s Z_0 }; Z_0< a ) ,\qquad a = \frac{d_{\max}}{d_{\min}} ,\nonumber
\end{eqnarray}
equation (\ref{pietra}) becomes equivalent to
%
%
\begin{equation}\label{pietrabisbis}
1- g_\infty(a s)=\bigl( 1- g_0(s) \bigr) e^{h_0(s)} .
\end{equation}
By definition, and because of assumption (A2), we have $Z_0\ge1$, $
Z_\infty\ge1$
and $a \in[1,2]$.
These bounds will turn out to be crucial later on.

For later use, we point out some simple identities. We recall the
definition of the \textit{exponential integral} function
$\operatorname{Ei}(s)$, $s>0$,
\[
\operatorname{Ei}(s)= \int_s ^\infty\frac{e^{-t} }{t} \,dt = \int_1
^\infty
\frac{
e^{-sx}}{x} \,dx .
\]
Given a Radon measure $t$ on $[0,\infty)$ (i.e., a Borel
nonnegative measure, giving finite mass to any bounded Borel set),
by Fubini's theorem it is simple to check that
%
%
\begin{equation}\label{pizzette}
\int_0 ^\infty\frac{e^{-s (1+x) }}{1+x} t (dx)= \int_s ^\infty du\,
e^{-u} \int_0 ^\infty e^{-u x} t (dx) .
\end{equation}
Above and in what follows, we will write $\int_c^\infty$ instead of
$\int_{[c,\infty)}$ for $c \geq0$.
If $t(dx)= c_0 \,dx$, the quantity
in (\ref{pizzette}) is simply the exponential integral $
\operatorname{Ei}(s)$ and the right-hand side of (\ref{pizzette}) gives an
alternative integral representation of~$\operatorname{Ei}(s)$. In
particular, the limit points in Theorem \ref{teo2} have Laplace
transform of the form
%
%
\begin{eqnarray}\label{polenta}
g_\infty^{(c_0)} (s)&=& 1- \exp
\biggl\{ -\int_0 ^\infty\frac{e^{-s (1+x) }}{1+x} t (dx) \biggr\}\nonumber\\[-8pt]\\[-8pt]
&=& 1-
\exp\biggl\{- \int_s ^\infty du\, e^{-u} \int_0 ^\infty e^{-u x} t
(dx) \biggr\},\nonumber
\end{eqnarray}
where $ t (dx) = c_0 \,dx$.

This observation suggests to write the Laplace
transforms $g_0$, $g_\infty$ in the form
(\ref{polenta}) for suitable Radon measures $t_0$ and $t_\infty$.
The
following result guarantees that such an integral representation
exists.
\begin{Lemma}\label{gola}
Let $Z$ be a random variable such that $Z \geq1$, and define
$g(s)=\bbE[ e^{-sZ} ]$, $s \geq0$. Let $w\dvtx(0, \infty)
\rightarrow\bbR$ be the unique function such that
%
%
\begin{equation}\label{pompiere0}
g(s)=1-\exp\biggl\{ - \int_s ^\infty du\, e^{-u} w(u)\biggr\},\qquad
s>0 ,
\end{equation}
that is,
%
%
\begin{equation}\label{pompiere1}
w (s)=
- \frac{e^s g'(s)}{1-g (s)} ,\qquad s >0 .
\end{equation}
Then the function $w$
is completely monotone. In particular, there exists a~unique Radon
measure $t(dx)$ on $[0, \infty)$ (not
necessarily of finite total mass) such that
%
%
\begin{equation}\label{cotoletta}
w(s)= \int_0 ^\infty e^{-sx} t(dx) ,\qquad s>0 ,
\end{equation}
and therefore
%
%
\begin{equation}\label{animainpena}
g(s)= 1- \exp\biggl\{ -\int_0^\infty\frac{e^{-s(1+x) }}{1+x} t
(dx)\biggr\} ,\qquad s\geq0 .
\end{equation}
Moreover,
%
%
\begin{equation}\label{fuoco}
\limsup_{s \downarrow0}-\frac{s g'(s)}{1-g (s)}\in[0 ,1] .
\end{equation}
\end{Lemma}

We recall that a function $f \dvtx(0,\infty)\rightarrow\bbR$ is
called \textit{completely monotone} if it possesses derivatives $D^n f
$ of all orders and
\[
(-1)^n D^n f (x) \geq0 \qquad \forall x >0 .
\]
Due to the above lemma, there exist two uniquely determined
Radon measures $t_0$ and $t_\infty$ on $[0,\infty)$, such that
$g_0$ and $g_\infty$ admit the integral representation
(\ref{animainpena}) with $t$ replaced by $t_0$ and $t_\infty$,
respectively.

In order to rewrite (\ref{pietrabisbis}) as identity in
terms of $t_0$ and $t_\infty$, we need to express the function
$h_0$ in terms of $t_0$. The following result gives us the
solution:
\begin{Lemma}\label{limonata}
Let $Z$ be a random variable such that $Z \geq1$, and let $g(s)$ be
its Laplace transform. Let $t$ be the unique Radon measure on
$[0,\infty)$ satisfying~(\ref{animainpena}) and call $m(dx)$ the
Radon measure with support in $[1,\infty)$ such that\looseness=-1
%
%
\begin{equation}\label{bilancia}
m (A)= \int_0 ^\infty\frac{\mathbh{1}_{1+x \in A} }{1+x} t(dx) .
\end{equation}\looseness=0
For each $k \geq1$, consider the convolution measure $ m ^{(k)}$
with support in $[k,\infty)$ defined as
%
%
\begin{equation}
m^{(k)} (A)= \int_1^\infty
m(dx_1)\int_1^\infty m(dx_2)\cdots\int_1^\infty m(dx_k)
\mathbh{1}_{x_1 +x_2+ \cdots+ x_k\in A} .\hspace*{-35pt}
\end{equation}
Then the law of $Z$ is given by
%
%
\begin{equation}\label{barrocciaio}
\sum_{k=1}^\infty\frac{(-1)^{k+1} }{k!} m^{(k)} .
\end{equation}
In particular
%
%
\begin{equation} \label{brucogiallo} \bbE[ e^{-sZ}; Z< a
]= \int_{[0,a-1) } \frac{e^{-s(1+x)}}{1+x} t(dx)
,\qquad s \geq0 .
\end{equation}
\end{Lemma}

We point out that, given a bounded Borel set $A$, the series
\[
m_*(A)=\sum_{k=1}^\infty\frac{(-1)^{k+1} }{k!} m^{(k)}(A)
\]
is a finite sum, since $m^{(k)}$ has support in $[k,\infty)$. The
thesis includes that this sum is a nonnegative number and that the
set-function $A\mapsto m_*(A)$, defined on bounded Borel sets, extend
uniquely to a Radon measure on all Borel sets.

Equation (\ref{brucogiallo}) above allows us to write $h_0(s)$ in terms
of $t_0$. Collecting the above observations we get for $s \geq0$
\begin{eqnarray*}
g_0(s)&=& 1- \exp\biggl\{ -\int_0^\infty\frac{e^{-s(1+x) }}{1+x} t_0
(dx)\biggr\} ,\\
g_\infty(s)&=& 1- \exp\biggl\{ -\int_0^\infty\frac{e^{-s(1+x)
}}{1+x} t_\infty
(dx)\biggr\} ,\\
h_0(s)&=&\int_{[0,a-1) } \frac{e^{-s(1+x)}}{1+x} t_0(dx) .
\end{eqnarray*}
Due to the above identities,
(\ref{pietrabisbis}) is equivalent to
%
%
\begin{equation}\label{florida}\qquad
\int
_0^\infty\frac{e^{-a s(1+x) }}{1+x}t_{\infty}(dx) =
\int_{[a-1,\infty) } \frac{e^{- s(1+x) }}{1+x}t_0(dx) ,\qquad
s\geq0.
\end{equation}

It is convenient now to introduce the following notation.
Given
an increasing function $\phi\dvtx[0,\infty) \rightarrow[0,\infty)$ and
a Radon measure $\mathfrak{m}$ on $[0,\infty)$, we denote by
$\mathfrak{m}\circ\phi$ the new Radon measure on $[0,\infty)$
defined by
%
%
\begin{equation}\label{gianni}
\mathfrak{m} \circ\phi(A)= \mathfrak{m}( \phi(A)) ,\qquad
A\subset\bbR\mbox{ Borel} .
\end{equation}
Note that $\mathfrak{m}\circ\phi$ is indeed a measure, due to the
injectivity of $\phi$. Moreover, it holds
%
%
\begin{equation}\label{regoletta}
\int_0^\infty f(x) \mathfrak{m} \circ\phi(dx)= \int
_{[\phi(0),\infty)} f( \phi^{-1} (x) ) \mathfrak{m}(dx) .
\end{equation}
%

We are finally able to give a simple characterization of
(\ref{florida}), which we know to be equivalent to
(\ref{pietrabisbis}):
\begin{Theorem}\label{cannone}
Consider the linear function $\phi\dvtx[0,\infty) \rightarrow
[0,\infty)$ defined as $\phi(x)= a(1+x)-1$. Then, equation
(\ref{florida}) [and therefore also (\ref{pietrabisbis})] is
equivalent to the relation
%
%
\begin{equation}\label{bianchini}
t_\infty= (1/a) t_0 \circ\phi.
\end{equation}
\end{Theorem}

\subsection{\texorpdfstring{Proof of Lemma \protect\ref{gola}}{Proof of Lemma 5.1}}
First we prove that $w$ is a
completely monotone function. Since $g(s)<1$ for $s>0$, we can
write $w= f \sum_{k=0}^\infty g^k$ where $f(s)= -e^s g '(s) $.
Trivially, $g$ is a completely monotone function. Since the product
of completely monotone functions is again a completely monotone
function (see Criterion 1 in Section XIII.4 of~\cite{Fe2}), we
conclude that $g^k$ is a completely monotone function. Since the sum
of completely monotone functions is trivially completely monotone,
we conclude that $\sum_{k=0}^\infty g^k$ is completely monotone. It
remains to prove that $f$ is completely monotone. To this aim we
observe that, by the Leibniz rule,
\begin{eqnarray*}
D^n f (s)
& = &
- \sum_{k=0}^n \pmatrix{n\cr k} D^{n-k} (e^s) D^k(g'(s))
=
- e^s \sum_{k=0}^n \pmatrix{n\cr k} D^{k+1} g (s) \\
& = &-e^s\sum_{k=0}^n \pmatrix{n\cr k} (-1)^{k+1} \bbE( e^{-s Z}
Z^{k+1} )
=
e^s \bbE\Biggl(e^{-s Z} Z \sum_{k=0}^n \pmatrix{n\cr k}(-Z)^k\Biggr) \\
& = &e^s \bbE\bigl( e^{-s Z} Z(1-Z)^n\bigr).
\end{eqnarray*}
Since $1-Z \leq0$, the sign of the $n$th derivative $D^n f $ is\vadjust{\goodbreak}
$(-1)^n$.

At this point, we can apply Theorem 1a in Section XIII.4 of
\cite{Fe2} to get that there exists a Radon measure $t(dx)$ on
$[0, \infty)$ (not necessarily of finite total mass) satisfying
(\ref{cotoletta}). Moreover, the above measure $t$ is uniquely
determined due to the \textit{inversion formula} given in Theorem 2, Section
XIII.4 of~\cite{Fe2}. Finally, we derive (\ref{animainpena}) for $s
>0$ from
(\ref{pizzette}), (\ref{pompiere0}) and (\ref{cotoletta}). The
extension to $s=0$ follows from the monotone convergence theorem.

In order to prove (\ref{fuoco}) we observe that $y e^{-y} \leq1-
e^{-y}$ for all $y \geq0$, thus implying that
\[
- s g'(s)= \bbE( sZ e^{-sZ}) \leq1-\bbE( e^{-sZ} ) =1- g(s) \qquad
\forall s>0 .
\]
In particular, the ratio in (\ref{fuoco}) is bounded by $1$. On
the other hand $-s g'(s)= \bbE( sZ e^{-sZ})>0$ while $1- g(s)
>0$, thus implying that the ratio in (\ref{fuoco}) is positive.

\subsection{\texorpdfstring{Proof of Lemma \protect\ref{limonata}}{Proof of Lemma 5.2}}
Due to the definition of $m(dx)$, we can write
%
%
\begin{equation}\label{succhetto}
\int_0^\infty\frac{e^{-s(1+x) }}{1+x} t (dx) = \int_0^\infty
e^{-sx} m(dx) .
\end{equation}
By (\ref{animainpena}), since $g(s)<1$ for $s>0$, we get that the
above quantities are finite as $s>0$. Using the series expansion
of the exponential function we can write
%
%
\begin{equation}\label{preparocolazione}\qquad
1- \exp\biggl\{ -\int_0^\infty\frac{e^{-s(1+x) }}{1+x} t (dx)
\biggr\}= \sum_{k=1}^\infty\frac{(-1)^{k+1}}{k!}\biggl(\int_0^\infty
e^{-sx} m(dx) \biggr)^k .
\end{equation}
%
Since
%
%
\begin{equation} \label{kindercolazione}\biggl( \int_0^\infty
e^{-sx} m(dx) \biggr)^k = \int_0^\infty e^{-sx } m^{(k)} (dx)
,
\end{equation}
we can rewrite (\ref{preparocolazione}) as
%
%
\begin{equation}\label{tacchi}
\sum_{k=1}^\infty\frac{(-1)^{k+1}}{k!} \int_0^\infty e^{-s x} m^{(k)}
(dx)= \sum_{k=1}^\infty\Biggl(\sum_{j=k}^\infty a_{k,j}\Biggr),
\end{equation}
where
\[
a_{k,j} = \frac{(-1)^{k+1}}{k!} \int_{I_j} e^{-s x} m^{(k)}
(dx) ,\qquad I_j=[j,j+1) \qquad\mbox{for } j \geq1 .
\]
Using again the series expansion of the exponential function and
also (\ref{kindercolazione}), we conclude that
%
%
\begin{eqnarray}
\sum_{k=1}^\infty\sum_{j=1 } ^\infty|a_{j,k}|&=&\sum_{k=1}^\infty
\frac{1}{k!}\int_0^\infty e^{-s x} m^{(k)} (dx)\nonumber\\[-8pt]\\[-8pt]
&=&\exp
\biggl\{\int_0^\infty e^{-sx} m(dx) \biggr\}-1<\infty.\nonumber
\end{eqnarray}
In particular, we can arrange arbitrarily the terms in the series
given by the right-hand side of (\ref{tacchi}), getting always the same
limit. This fact
implies that
%
%
\begin{eqnarray} \label{radio}\qquad
\mbox{r.h.s. of } (\ref{tacchi}) & = & \sum_{j=1}^\infty
\Biggl(\sum_{k=1}^j a_{k,j} \mathbh{1}_{k \ \mathrm{odd}} \Biggr)
+ \sum_{j=1}^\infty\Biggl(\sum_{k=1}^j a_{k,j} \mathbh{1}_{k\
\mathrm{even}} \Biggr)\nonumber\\[-8pt]\\[-8pt]
& = &\int_0^\infty e^{-sx} \nu_+(dx)-\int_0^\infty e^{-sx}
\nu_-(dx) ,\nonumber
\end{eqnarray}
where the Radon measures $\nu_+$
and $\nu_-$ on $[0, \infty)$ are defined as follows:
\begin{eqnarray*}
\nu_+(A)&=& \sum_{k=1}^\infty\frac{\mathbh{1}_{k \ \mathrm{odd}}}{k!}
m^{(k)} (A) ,\\
\nu_-(A)&=& \sum_{k=1} ^\infty\frac{\mathbh{1}_{k\ \mathrm{even}}}{k!} m^{(k)} (A) .
\end{eqnarray*}
We point out that for any bounded Borel subset $A \subset
[0,\infty)$ the above series are indeed finite sums since each
$m^{(k)}$ has support in $[k,\infty)$. In addition, $\nu_+$ and
$\nu_-$ have support contained in $[1, \infty)$ and $[2,\infty)$,
respectively.

Collecting (\ref{animainpena}), (\ref{preparocolazione}),
(\ref{tacchi}) and (\ref{radio}), we obtain that
\[
g(s)= \int
_0^\infty e^{-sx} \nu_+(dx)-\int_0^\infty e^{-sx} \nu_-(dx)
\]
for
all $s>0$. Writing $p_Z$ for the law of $Z$, the above identity
implies that the Laplace transforms of the measures $p_Z+\nu_-$ and
$\nu_+$ coincide on $(0,\infty)$. Due to Theorem 2 in Section XIII.4
of~\cite{Fe2}, this implies that $p_Z+\nu_-=\nu_+$. It follows that
\[
p_Z(A)= \nu_+(A)-\nu_-(A)\qquad \forall A \subset\bbR\mbox{
bounded and
Borel} .
\]
Since for $A$ as above we can write $\nu_+(A)-\nu_-(A)=\sum
_{k=1}^\infty\frac{(-1)^{k+1} }{k!} m^{(k)}(A) $, we get that the
law $p_Z$ coincides with (\ref{barrocciaio}).

It remains now to prove (\ref{brucogiallo}). To this aim we observe
that, since $m^{(k)}$ has support contained in $[k, \infty)$, measure
(\ref{barrocciaio}) equals $m$ on $[1, 2)$. Since $a\leq2$,
and using the definition of the measure $m$ given by
(\ref{bilancia}), we obtain that
\begin{eqnarray*}
\bbE[
e^{-sZ} ; Z< a]&=& \int_{ [1,a) } e^{-sx } p_Z(dx)= \int_{
[1,a) } e^{-sx } m(dx)\\
&=& \int_{[0,a-1)} \frac{e^{-s(1+x)}}{1+x}
t(dx) .
\end{eqnarray*}
This concludes the proof of (\ref{brucogiallo}).

\subsection{\texorpdfstring{Proof of Theorem \protect\ref{cannone}}{Proof of Theorem 5.3}}
We write $\rho(dx)$ for the measure in the right-hand side of
(\ref{bianchini}). Using that $a[\phi^{-1}(x)+1]= 1+ x $, we obtain
for $s\geq0$ that
\begin{eqnarray*}
\int_0^\infty\frac{e^{-a s(1+x) }}{1+x}\rho(dx)
&=&
a^{-1} \int_{[\phi(0), \infty)}\frac{e^{-s
(1+x)}}{a^{-1}(1+x)}t_0(dx)\\
&=&
\int_{[a-1,\infty)}\frac{e^{-s (1+x)}}{1+x} t_0 (dx) .
\end{eqnarray*}
The above identity implies that (\ref{florida}) holds if and only
if
%
%
\begin{equation}\label{petrolio}
\int_0^\infty\frac{e^{-as(1+x) }}{1+x}t_{\infty}(dx)
=
\int_0^\infty\frac{e^{-a s(1+x)}}{1+x}\rho(dx) \qquad \forall s
\geq0 .
\end{equation}
We write $m_\infty$ and $m'$ for the measures on $[1,\infty)$
such that
\[
m_\infty(A)= \int_0 ^\infty\frac{\mathbh{1}_{1+x \in A} }{1+x}
t_{\infty}(dx) ,\qquad m'
(A)= \int_0^\infty\frac{\mathbh{1}_{1+x \in A} }{1+x} \rho(dx)
\]
for bounded Borel subsets $A \subset[1,\infty)$. Then, by
(\ref{petrolio}), we get that (\ref{florida}) holds if and only if
the Laplace transforms of the measures $m_\infty$ and $m' $ coincide
on $(0,\infty)$. By Theorem 2 in Section XIII.4 of~\cite{Fe2}, this
last property is equivalent to the identity $m_\infty=m'$, which is
equivalent to $t_\infty= \rho$.

\section{Hierarchical Coalescence Process: Proofs}
\subsection{\texorpdfstring{Application of the recursive identity (\protect\ref{pietra}) to the HCP}{Application of the recursive identity (2.8) to the HCP}}\label{secpower}

We begin by collecting some useful formulae for the hierarchical
coalescence process that we derive from results obtained for the
one-epoch coalescence process in the previous section. These
formula will be used throughout the whole section.

We use notation and definitions of Theorem \ref{teo2}. In
particular $\mu$ and $\nu$ are probability measures on $[1,\infty)$
and $\mathbb{R}$, respectively. We define here $X^{(n)}$, $n \in
\bbN_+$, as the length of the leftmost domain inside $(0,\infty)$
at the beginning of the $n$th epoch, that is, $X^{(n)}= x^{(n)} _2
(0)- x^{(n)}_1 (0)$. Moreover we set $Z^{(n)}=X^{(n)}/d^{(n)}$.
Note that $X^{(n)}$ has law $\mu^{(n)}$.
Also, $\mathbb{E}$ stands for the expectation with respect to the
hierarchical coalescent process starting indifferently from
$\cQ=\operatorname{Ren}(\nu,\mu)$, $\cQ=\operatorname{Ren}(\mu)$ or
$\cQ=\operatorname{Ren}_{\mathbb{Z}}(\mu)$. For any $n \in\mathbb{N}_+$ and
any $s \geq0$ let
%
%
\begin{equation}\label{sciopero} g^{(n)}(s)=\mathbb{E}
\bigl(e^{-sZ^{(n)}}\bigr) ,\qquad
h^{(n)}(s)=\mathbb{E}\bigl(e^{-sZ^{(n)}} \mathbh{1}_{1 \leq Z^{(n)}
<a_n}\bigr),
\end{equation}
where
$a_n=d^{(n+1)}/d^{(n)}$. Thanks to Theorem \ref{teo1}, [see also
(\ref{pietrabisbis})], we get a~system of recursive identities
%
%
\begin{equation} \label{hofreddo1}
1-g^{(n)}(s a_{n-1})=\bigl(1-g_{n-1}(s)\bigr)e^{h^{(n-1)}(s)} \qquad\forall n
\geq2 .
\end{equation}
These recursive identities will be essential in the subsequent
computations. Since $Z^{(n)} \geq1$, by Lemma \ref{gola} there exists a
unique measure $t^{(n)}$ on $[0,\infty)$ such that
%
%
\begin{equation}\label{hofreddo2bis}
g^{(n)} (s) = 1- \exp\biggl\{ - \int_0 ^\infty\frac{e^{-s (1+x) }
}{1+x} t^{(n)} (dx ) \biggr\},\qquad n \geq1.
\end{equation}
Invoking now Theorem \ref{cannone} we conclude that
%
%
\begin{equation}\label{ventoburlonebis}
t^{(n)} = (1/a_{n-1}) t^{(n-1)} \circ\phi_{n-1},\qquad n \geq2,
\end{equation}
where
$\phi_n(x)= a_n (1+x)-1$.

Up to now we have only moved from the
system of recursive identities~(\ref{hofreddo1}) to the new system
(\ref{ventoburlonebis}). But while the former is highly nonlinear and
complex, the latter is solvable. Indeed if we define
%
%
\begin{equation}\label{vacanzina}
\psi_n (x):= \phi_1 \circ\phi_2 \circ\cdots\circ\phi_n (x),
\end{equation}
then $\psi_n(x)=d^{(n+1)} (1+x)-1$ and (\ref{gianni}) together with
(\ref{ventoburlonebis}) imply
%
%
\begin{equation}\label{brumsporco}
t^{(n)} = \frac{1}{d^{(n)}} t^{(1)} \circ\psi_{n-1} ,\qquad n
\geq2 .
\end{equation}
%
Finally, using (\ref{brucogiallo}) and (\ref{brumsporco}), it is
simple to check that
%
%
\begin{equation}\label{nuvoletta}
h^{(n)}(s) =
\int_{[d^{(n)}-1, d^{(n+1)}-1)} e^{-s(1+x)/d^{(n)} } (1+x)^{-1} t^{(1)}
(dx) ,\qquad n \geq1,\hspace*{-35pt}
\end{equation}
%
where we used the identity $(1+\psi^{-1}_{n-1}(x))=(1+x)/d^{(n)}$.
\subsection{\texorpdfstring{Asymptotic of the interval law in the HCP: Proof of Theorem \protect\ref{teo2}}{Asymptotic of the interval law in the HCP: Proof of Theorem 2.19}}
\label{secproofteo2}

Section \ref{secpower} provides us with most of the tools necessary
for the
proof of Theorem~\ref{teo2}. In particular our starting point is
identity (\ref{hofreddo2bis}):
%
%
\begin{equation} \label{hofreddo2}
g^{(n)} (s) = 1- \exp\biggl\{ - \int_0 ^\infty\frac{e^{-s (1+x) }
}{1+x} t^{(n)} (dx ) \biggr\},\qquad n \geq1 .
\end{equation}
Defining
%
%
\begin{equation}\label{boltino}
U^{(n)}(x)= \cases{
t^{(n)} ([0,x]), &\quad if $x \geq0$,\cr
0, &\quad otherwise,}
\end{equation}
we get that $U^{(n)}$ is a c\`{a}dl\`{a}g function, $dU^{(n)} =
t^{(n)}$ and
$U^{(n)}(x)=0$ for $x<0$. By (\ref{brumsporco}) it holds that
%
%
\begin{eqnarray}\label{pescefritto}
U^{(n)}(x) &=& \frac{1}{d^{(n)}}\bigl[ U^{(1)}(\psi_{n-1} (x))-
U^{(1)}(\psi_{n-1} (0)-)\bigr] ,\nonumber\hspace*{-40pt}\\[-8pt]\\[-8pt]
&=& \frac{1}{d^{(n)}} \bigl[ U^{(1)} \bigl(
d^{(n)} (1+x)-1\bigr)- U^{(1)} \bigl(\bigl(
d^{(n)}-1\bigr)-\bigr)\bigr] ,\qquad n\geq1 .\nonumber\hspace*{-40pt}
\end{eqnarray}
If we fix $n\geq2$, integrate by parts and use
$U^{(n)}(0-)=0$, we can rewrite the integral in (\ref{hofreddo2}) as
%
%
\begin{eqnarray}\label{ananas}
\int_0^\infty\frac{e^{-s(1+x) }}{1+x} t^{(n)}(dx)
&=&
\int_0^\infty\frac{e^{-s(1+x) }}{1+x} \,dU^{(n)}(x) \nonumber\\
&=& \lim_{y \uparrow\infty} \frac{e^{-s(1+y) }}{1+y} U^{(n)}(y)\\
&&{} -
\int
_0^\infty\biggl( \frac{d}{dx} \biggl( \frac{e^{-s(1+x) }}{1+x}\biggr)
\biggr)
U^{(n)}(x) \,dx .\nonumber
\end{eqnarray}
We now use (\ref{limarrosto}), the key hypothesis. Since
$g^{(1)}(s)= g(s)$ because $d^{(1)}=1$, if $w^{(1)}$ denotes the
Laplace transform of $t^{(1)}$ [i.e., $w^{(1)}(s) =
\int_0^\infty e^{-s x} t^{(1)}(dx) $], then (\ref{limarrosto})
together with (\ref{pompiere1}) implies
%
%
\begin{equation}\label{brio2}
\lim_{s\downarrow0} s w^{(1)}(s)= c_0 .
\end{equation}
Finally, Tauberian Theorem 2 in Section XIII.5 of~\cite{Fe2} shows
that (\ref{brio2}) gives
%
%
\begin{equation}\label{marina}
\lim_{y\uparrow\infty} \frac{ U^{(1)}(y)}{y}= c_0 .
\end{equation}
The above limit together with (\ref{pescefritto}) implies that there
exists a
suitable constant $C>0$ such that
%
%
\begin{equation}\label{bound} U^{(n)}(x)
\leq C (1+x) ,\qquad n\geq1 , x \geq0 .
\end{equation}
In
particular, the limit in the right-hand side of (\ref{ananas}) is zero and
%
%
\begin{eqnarray}
\label{miserve}
\int_0^\infty\frac{e^{-s(1+x) }}{1+x} t^{(n)}(dx)=- \int_0^\infty
\biggl( \frac{d}{dx} \biggl( \frac{e^{-s(1+x) }}{1+x}\biggr) \biggr)
U^{(n)}(x) \,dx ,\nonumber\\[-8pt]\\[-8pt]
&&\eqntext{n\geq2.}
\end{eqnarray}
By (\ref{pescefritto}), (\ref{marina}) and the fact that $c_n \to
\infty$, $\lim_{n\to\infty}U^{(n)}(x) \to c_0x$ for all \mbox{$x \geq0$}.
This limit together with
(\ref{bound}) allows us to apply the dominated convergence theorem,
to get that
\begin{eqnarray*}
\lim_{n\to\infty} \int_0^\infty\frac{e^{-s(1+x) }}{1+x}
t^{(n)}(dx)
&=&
-c_0 \int_0 ^\infty\biggl( \frac{d}{dx} \biggl( \frac{e^{-s(1+x)
}}{1+x}\biggr) \biggr) x \,dx \\
&=& c_0 \int_0 ^\infty\frac{e^{-s (1+x)
}}{1+x} \,dx
\end{eqnarray*}
(in the last identity we have simply integrated by parts). In
conclusion we have shown that $g^{(n)}$ converges
point-wise to the function $g^{(\infty)}_{c_0}$ defined as in~(\ref{macedonia}).
Since in addition $\lim_{s\downarrow0 }
g^{(\infty)}_{c_0}(s)=1$, by Theorem 2 in Section~XIII.1 of
\cite{Fe2}, we conclude that $g^{(\infty)}_{c_0}$ is the Laplace
transform of some nonnegative random variable $Z _{c_0}^{(\infty)}$
and that
$Z^{(n)}$ weakly converges to $Z _{c_0}^{(\infty)}$.

Finally, Lemma \ref{limonata} allows us to determine the law of
$Z _{c_0}^{(\infty)}$. Indeed, the measure $m$ associated to
$t(dx):=c_0 \,dx$
by means of (\ref{bilancia}) is simply of the form $m(dx)=
(c_0/x)\mathbh{1}_{x\geq1}\,dx$. In particular $m^{(k) }(dx)= c_0^k
\rho_k(x) \mathbh{1}_{x\geq k} \,dx $ with~$\rho_k$ defined in
(\ref{vonnegut}). It remains then to apply (\ref{barrocciaio}).
\begin{remark}
It is useful to observe that if the initial scale $d^{(1)}$
was different from one than necessarily $g(s)\neq
g^{(1)}(s)$. However, and that is the reason why we could fix
$d^{(1)}=1$, the limit (\ref{limarrosto}) is invariant under rescaling
the variable $s$ by a constant, that is, (\ref{limarrosto}) for $g$
implies the same limit for~$g^{(1)}$.
\end{remark}


\subsection{\texorpdfstring{Asymptotic of the first point law in the HCP: Proof of Theorem \protect\ref{teo3}}{Asymptotic of the first point law in the HCP: Proof of Theorem 2.24}}\label{secteo3}
We first prove the result for the special case $\nu={\delta}_0$.
We set
\[
\ell^{(n)}(s)=\bbE\bigl[ \exp\bigl\{ -s X^{(n)}_0/ d^{(n)}\bigr\}
\bigr] ,\qquad s\in\bbR_+ .
\]
%
Recall the notation of Section \ref{secpower} and in
particular the definition of the constants $a_n=d^{(n+1)}/d^{(n)}$.
By applying to each epoch the key
identity (\ref{pietrabis}), we get the recursive system,
\[
\ell^{(n)}(s) = \ell^{(n-1)}( s/a_{n-1}) \exp\biggl\{ \frac{
1}{1+\gamma}
\bigl[ h^{(n-1)} (s/a_{n-1} ) - h^{(n-1)}(0)\bigr]\biggr\} ,\qquad n
\geq2 .
\]
Since $a_j a_{j+1} \cdots a_{n-1} = d^{(n)}/d^{(j)}$, by combining the
above recursive identities we get
%
%
\begin{eqnarray}\label{camoscio}\quad
\ell^{(n)}(s)= \ell^{(1)}\bigl( s /d^{(n)}\bigr) \exp\biggl\{ \frac
{1}{1+\gamma} \sum
_{j=1}^{n-1}
\bigl[ h^{(j)}\bigl(s d^{(j)}/d^{(n)}\bigr) -h^{(j)}(0) \bigr]
\biggr\},\nonumber\\[-8pt]\\[-8pt]
&&\eqntext{n \geq2 .}
\end{eqnarray}
We now use the integral representation (\ref{nuvoletta}) to get
%
%
\begin{eqnarray}\label{nuvolettabis}\qquad
h^{(j)}\bigl( s d^{(j)}/d^{(n)}\bigr) = \int_{[d^{(j)}-1, d^{(j+1)}-1)} (1+x)^{-1}
e^{-({s}/{d^{(n)}})(1+x) } t^{(1)} (dx) ,\nonumber\\[-8pt]\\[-8pt]
&&\eqntext{j\geq1 .}
\end{eqnarray}
This allows us to write
%
%
\begin{eqnarray}\label{secchio}
F^{(n)}(s):\!&=& \sum_{j=1}^{n-1}
h^{(j)}\bigl(s d^{(j)}/d^{(n)}\bigr) \nonumber\\[-8pt]\\[-8pt]
&=&
\int_{[0, d^{(n)}-1)} (1+x)^{-1}
e^{-s(1+x)/d^{(n)}} t^{(1)}(dx)
.\nonumber
\end{eqnarray}
Setting $U^{(1)}(x)= t^{(1)}( [0,x])$, we can use
integration by parts and the change of variable
$y=(1+x)/d^{(n)}$ to conclude that
\[
F^{(n)}(s)= [e^{-s}] \frac{U^{(1)}(d^{(n)}-1)}{d^{(n)}}+
\int_{[1/d^{(n)},1)
}e^{-sy }\biggl[ \frac{s}{y}+ \frac{1}{y^2} \biggr] \frac{ U^{(1)}( d^{(n)}
y -1 ) }{ d^{(n)} }\,dy .
\]
In particular, we can write
\begin{eqnarray*}
F^{(n)}(s)-F^{(n)}(0)&=&(e^{-s}-1) \frac
{U^{(1)}(d^{(n)}-1)}{d^{(n)}}\\
&&{}+
\int_{[1/d^{(n)},1) }s e^{-sy } \frac{ U^{(1)}( d^{(n)} y -1 ) }{
d^{(n)} y }\,dy\\
&&{}+
\int_{[1/d^{(n)},1) }\frac{e^{-sy } -1 }{y} \frac{ U^{(1)}( d^{(n)}
y -1 ) }{
d^{(n)} y }\,dy .
\end{eqnarray*}
We have already observed that (\ref{limarrosto}) together with a
Tauberian theorem implies the limit (\ref{marina}). Since $d^{(n)}
\to\infty$, we can then apply the dominated convergence theorem to
conclude that
%
%
\begin{eqnarray}\label{rumorimobili}
&&\lim_{n \to\infty}
F^{(n)}(s)-F^{(n)}(0)\nonumber\\[-8pt]\\[-8pt]
&&\qquad= c_0 \biggl( e^{-s}-1 +
\int_{(0,1)}s e^{-sy } \,dy+ \int_{(0,1) }\frac{e^{-sy } -1 }{y}
\,dy\biggr) .\nonumber
\end{eqnarray}
Collecting (\ref{camoscio}),
(\ref{secchio}) and (\ref{rumorimobili}), we conclude that for any
$s \in\bbR_+ $ the sequence $(\ell^{(n)}(s))_{n\geq1}$
converges to
\[
\exp\biggl\{-
\frac{c_0}{1+\gamma} \int_{(0,1)} \frac{1-e^{-sy} }{y} \,dy \biggr\}
.
\]
Since the latter is continuous at $s=0$ we get the desired
weak convergence (cf. Theorem 3.3.6 in~\cite{D}).

Now we prove the result for a general $\nu$. By translation
invariance, for any $x\in\bbR$,
$\mathbb{P}_{\operatorname{Ren}(\delta_x,\mu)} (X^{(n+1)}_0= x
)=\mathbb{P}_{\operatorname{Ren}(\delta_0,\mu)}
(X^{(n+1)}_0=0)$. Hence, for any bounded continuous
function $f$,
\begin{eqnarray*}
\mathbb{E}_{\operatorname{Ren}(\nu,\mu)} \bigl( f
\bigl(X^{(n)}_0/d^{(n)}\bigr)\bigr) & = &
\int\nu(dx) \mathbb{E}_{\operatorname{Ren}(\delta_x,\mu)} \bigl( f
\bigl(X^{(n)}_0/d^{(n)}\bigr)\bigr) \\
& = &\int\nu(dx) \mathbb{E}_{\operatorname{Ren}(\delta_0,\mu)} \bigl(
f\bigl(\bigl(X^{(n)}_0-x\bigr)/d^{(n)}\bigr)\bigr),
\end{eqnarray*}
and the result follows from the case $\nu=\delta_0$ considered
above once we use the assumption $\lim_{n\to\infty}d^{(n)}=+\infty
$. This completes the proof.

\subsection{\texorpdfstring{Asymptotic of the survival probability: Proof of Theorem \protect\ref{teo4}}{Asymptotic of the survival probability: Proof of Theorem 2.25}} \label{secteo4}

This section is dedicated to the proof of Theorem \ref{teo4}. We use
the notation and definitions of Section \ref{secpower}. We start
with point (i).\vadjust{\goodbreak}

\subsubsection{Proof of \textup{(i)}}

As in the proof of Theorem \ref{teo3} it is enough to consider the
case $\nu={\delta}_0$.
Recall the definition of $\mu^{(n)}$ introduced before Theorem
\ref{teo2}. By a simple induction argument based on Theorem
\ref{teo1}(ii), if the initial law $\cQ$ is
$\operatorname{Ren}(\delta_0,\mu)$, then the law of $\xi^{(j)}(0)$, that
is, the SPP at the
beginning of the $j$th epoch, conditional to the event $\{0 \in
\xi^{(j)}(0) \}$ is $\operatorname{Ren} (\delta_0,\mu^{(j)})$. Hence, by
conditioning and by using the Markov property, we get
\begin{eqnarray*}
\mathbb{P}_{\mathcal{Q}} \bigl( X^{(n+1)}_0=0 \bigr) & = &
\mathbb{P}_{\mathcal{Q}} \bigl( X^{(1)}_0=0\bigr) \prod_{j=1}^{n}
\mathbb{P}_{\mathcal{Q}} \bigl( X^{(j+1)}_0=0 \tc X^{(j)}_0=0
\bigr) \\
& = &\prod_{j=1}^{n} \mathbb{P}_{\operatorname{Ren} (\delta_0,\mu^{(j)})}
\bigl(X^{(j+1)}_0=0 \bigr) .
\end{eqnarray*}
In the last line, we also used the trivial equality
$\mathbb{P}_{\mathcal{Q}} ( X^{(1)}_0=0 )=1$. Theorem~\ref{teo1bis}(ii) ensures that
\[
\mathbb{P}_{\operatorname{Ren} (\delta_0,\mu^{(j)})} \bigl(X^{(j+1)}_0=0
\bigr)= e^{-{h^{(j)}(0)}/({1+\gamma})}\qquad \forall j \geq1,
\]
where, thanks to (\ref{nuvoletta}),
\[
h^{(j)}(0)=\mu^{(j)}\bigl(\bigl[d^{(j)},d^{(j+1)}\bigr)\bigr)=
\int_{[d^{(j)}-1, d^{(j+1)}-1)} (1+x)^{-1}t^{(1)}(dx) .
\]
It follows that
%
%
\begin{eqnarray} \label{eqmizzica}
\mathbb{P}_{\mathcal{Q}} \bigl(X^{(n+1)}_0= 0 \bigr) &=& \exp\Biggl\{
- \frac{1}{1+\gamma} \sum_{j=1}^n h^{(j)}(0) \Biggr\} \nonumber\\[-8pt]\\[-8pt]
&=& \exp\biggl\{ -
\frac{1}{1+\gamma} \int_{[0, d^{(n+1)}-1)} \frac{t^{(1)}(dx)}{1+x}
\biggr\} .\nonumber
\end{eqnarray}
If
$U^{(1)}(x)=t^{(1)}([0,x])$, and using integration by parts one gets
\[
\int_{[0, d^{(n+1)}-1)} \frac{t^{(1)}(dx)}{1+x} =
\frac{U^{(1)}(d^{(n+1)}-1)}{d^{(n+1)}} + \int_{[0, d^{(n+1)}-1)}
\frac{U^{(1)}(x)}{(1+x)^2} \,dx .
\]
As in (\ref{marina}) our assumption implies that $\lim_{y \to
\infty} \frac{U^{(1)}(y)}{y} = c_0$. Since $d^{(n)} \to\infty$ we
get immediately that $\frac{U^{(1)}(d^{(n+1)}-1)}{d^{(n+1)}} =
c_0+o(1)$. On the other hand, if $A = \sqrt{\ln(d^{(n+1)})}$ and
using again that $\lim_{y \to\infty} \frac{U^{(1)}(y)}{y} = c_0$,
we have
\begin{eqnarray*}
\int_{[0, d^{(n+1)}-1)} \frac{U^{(1)}(x)}{(1+x)^2} \,dx
& = &\int_{[0,A)} \frac{U^{(1)}(x)}{(1+x)^2} \,dx + \int_{[A,
d^{(n+1)}-1)} \frac{U^{(1)}(x)}{c_0(1+x)} \frac{c_0}{1+x} \,dx \\
& \leq &
U^{(1)}(A) + \bigl(1+o(1)\bigr) \int_{[A, d^{(n+1)}-1)} \frac{c_0}{1+x} \,dx \\
& = &\bigl(1+o(1)\bigr)c_0 \ln\bigl(d^{(n+1)}\bigr) .
\end{eqnarray*}
Similarly,
\begin{eqnarray*}
\int_{[0, d^{(n+1)}-1)} \frac{U^{(1)}(x)}{(1+x)^2} \,dx &\geq&
\int_{[A, d^{(n+1)}-1)} \frac{U^{(1)}(x)}{c_0(1+x)} \frac{c_0}{1+x}
\,dx \\
&=&
\bigl(1+o(1)\bigr)c_0 \ln\bigl(d^{(n+1)}\bigr).
\end{eqnarray*}
In conclusion\vspace*{1pt} $\int_{[0, d^{(n+1)}-1)} \frac{t^{(1)}(dx)}{1+x} = (1+o(1))c_0
\ln(d^{(n+1)})$. Result (i) of Theorem~\ref{teo4} follows from
(\ref{eqmizzica}).

\subsubsection{Proof of \textup{(ii)}}

The second part of Theorem \ref{teo4} follows from part~(i) using the universal coupling introduced in Section \ref{secdomaindyn}.

We distinguish between two cases.
Assume first that $\gamma=0$. This implies ${\lambda}_\ell=0$. In turn,
site $0$ cannot be erased from any ring of its left domain. Hence,
the event $\{0 \in\xi^{(n)}(\infty)\}$ depends only on the rings of the
domains on the right of $0$. Therefore
\[
\mathbb{P}_{\operatorname{Ren}(\mu\tc0)} \bigl( 0 \in\xi^{(n)}(\infty)
\bigr) =
\mathbb{P}_{\operatorname{Ren}(\delta_0,\mu)} \bigl( X^{(n+1)}_0=0
\bigr)
\]
and the expected result follows at once from point (i)
(with $\gamma=0$).

Now assume that $\gamma>0$. Then, by Lemma \ref{sepeff}
we can write
\[
\mathbb{P}_{\operatorname{Ren}(\mu\tc0)} \bigl( 0 \in\xi^{(n)}(\infty)
\bigr)
= \mathbb{P}^*_{\operatorname{Ren}(\delta_0,\mu)} \bigl( X^{(n+1)}_0=0
\bigr) \times
\mathbb{P}_{\operatorname{Ren}(\delta_0,\mu)} \bigl( 0 \in
\xi^{(n)}(\infty)
\bigr),
\]
where $\mathbb{P}^*$ denotes the probability measure with respect to
the hierarchical coalescent process built with ${\lambda
}^{(n,*)}_r={\lambda}^{(n)}_\ell$ and
${\lambda}^{(n,*)}_\ell={\lambda}^{(n)}_r$ (i.e., the mirror with
respect to the origin
of the hierarchical coalescence process built with~${\lambda}^{(n)}_r$ and
${\lambda}^{(n)}_\ell$). The identity ${\lambda}^{(n)}_\ell= \gamma
{\lambda}^{(n)}_r$ implies $
{\lambda}^{(n,*)}_\ell= \frac{1}{\gamma} {\lambda}^{(n,*)}_r$.
Hence, by applying
twice the result of part (i), we get
\begin{eqnarray*}
\mathbb{P}_{\operatorname{Ren}(\mu\tc0)} \bigl( 0 \in\xi^{(n)}(\infty)
\bigr) &=&
\bigl(1/d^{(n+1)}\bigr)^{({c_0}/({1+{1}/{\gamma}}))(1+o(1))}
\bigl(1/d^{(n+1)}\bigr)^{({c_0}/({1+\gamma}))(1+o(1))}\\
&=&
\bigl(1/d^{(n+1)}\bigr)^{c_0(1+o(1))},
\end{eqnarray*}
and the proof is complete.

\subsection{\texorpdfstring{Convergence of moments in the HCP: Proof of Proposition \protect\ref{propkmoment}}{Convergence of moments in the HCP: Proof of Proposition 2.22}} \label{seckmoment}

The proof of Proposition \ref{propkmoment} will be divided in
various steps.
First we will prove the result for $f(x)=x^k$, and then for a generic
function $f$ satisfying $|f(x)| \leq c(1+x^k)$.
The parameter $k \geq1$ is fixed once for all.

In what follows, we will use the notation and the definitions of
Theorem~\ref{teo2}. In particular $\mu$ and $\nu$ are probability
measures on $[1,\infty)$ and~$\mathbb{R}$, respectively, $X^{(n)}$ is
a random variable with law $\mu^{(n)}$ chosen here as $X^{(n)}=
x^{(n)} _2
(0)- x^{(n)}_1 (0)$, $Z^{(n)}=X^{(n)}/a^{n-1}$ and,
$Z^{(\infty)}=Z_1^{(\infty)}$ is the weak limit of~$Z^{(n)}$ proven
in Theorem~\ref{teo2}. Recall that $d^{(n)}=a^{n-1}$ and in
particular $d^{(1)}=1$. Also, $\mathbb{E}$ stands for the
expectation with respect to the hierarchical coalescent process
starting indifferently from $\cQ=\operatorname{Ren}(\nu,\mu)$,
$\cQ=\operatorname{Ren}(\mu)$ or $\cQ=\operatorname{Ren}_{\mathbb
{Z}}(\mu)$.
Following Section~\ref{secproofteo2}, for any $n \geq1$ and any $s
\geq0$ we introduce
$g^{(n)}(s)=\mathbb{E}(e^{-sZ^{(n)}})$, the Laplace
transform of $Z^{(n)}$, and
$h^{(n)}(s)=\mathbb{E}(e^{-sZ^{(n)}} \mathbh{1}_{1 \leq Z^{(n)}
< a})$. Thanks to Theorem \ref{teo1} [see also
(\ref{hofreddo1})] we have
%
%
\begin{equation} \label{eqrecursione}
1-g^{(n)}(as)=\bigl(1-g^{(n-1)}(s)\bigr)e^{h^{(n-1)}(s)} \qquad\forall s \geq0,
\forall n \geq2 .
\end{equation}

\vspace*{6pt}

\textit{Notation warning.} In the sequel for any pair of $C^\infty$
functions $f,g$
the symbol $D^k f(x)$ will stand for the $k$th derivative of $f$
computed at the point~$x$ while the symbol $D_x^k f(g(x))$ will denote
the $k$th derivative w.r.t~$x$ of~$f(g(x))$.\looseness=1

The above recursive identity will be very useful in our computations.
Note that by Lebesgue's theorem, for any $n$ and any $k$,
$\mathbb{E}([Z^{(n)} ]^k) = \lim_{s \to0} (-1)^kD^k\times g^{(n)}(s) \in
[0,\infty]$. We shall write, for simplicity,
$D^k g^{(n)}(0):=\lim_{s \to0} D^k g^{(n)}(s)$. It is not difficult to
prove by induction on $n$ that $|D^k g^{(n)}(0)| < \infty$, by taking
the $k$th derivative of both sides of (\ref{eqrecursione}), using the
Leibniz rule and the fact that $\bbE( [Z^{(1)} ]^k)<\infty$. In turn
%
%
\begin{equation} \label{eqkmom}
\mathbb{E}\bigl(\bigl[Z^{(n)} \bigr]^k\bigr) < \infty\qquad\forall n \geq1 .
\end{equation}
As a technical preliminary we prove that the above bound holds
uniformly in $n$.
\begin{Lemma} \label{lemtech1}
Assume that $\mu$ as finite $k$th moment, that is,
\mbox{$\mathbb{E}([Z^{(1)} ]^k) < \infty$}. Then
\[
\sup_{n \geq1} \mathbb{E}\bigl(\bigl[Z^{(n)} \bigr]^k\bigr) < \infty.
\]
\end{Lemma}

\begin{pf}
It is not restrictive to take $n\geq2$. By (\ref{eqkmom}),
$\mathbb{E}([Z^{(n)} ]^k)$ is well defined. Moreover, $\mathbb
{E}([Z^{(n)} ]^k) =
(-1)^k D^k g^{(n)}(0)$. Hence, since $x^k e^{-x} \leq B:=k^k e^{-k}$
for $x \geq0$, we have
\begin{eqnarray*}
\mathbb{E}\bigl(\bigl[Z^{(n)} \bigr]^k\bigr) &=& \mathbb{E} \bigl(\bigl[Z^{(n)} \bigr]^k
e^{-Z^{(n)}} \bigr) + (-1)^k
\bigl( D^k g^{(n)}(0) - D^k g^{(n)}(1) \bigr) \\
&\leq& B + \int_0^1
\bigl|
D^{k+1}g^{(n)}(u)\bigr| \,du .
\end{eqnarray*}
The above bound and Lemma \ref{lemtech2} below imply that
\[
\mathbb{E}\bigl(\bigl[Z^{(n)} \bigr]^k\bigr) \leq\frac{3}{2}A+B + \frac
{2ea}{(a-1)a^{n-1}} \int_0^1
\biggl|D^{k+1}g^{(1)}\biggl( \frac{u}{a^{n-1}} \biggr) \biggr|\, du
\]
for some constant $A$ that depends on $k$ and on
$\mathbb{E}([Z^{(1)} ]^k)$. Now by definition of $g^{(1)}$ and Fubini's
theorem, we get that
\begin{eqnarray*}
\int_0^1 \biggl|D^{k+1}g^{(1)}\biggl( \frac{u}{a^{n-1}} \biggr)
\biggr|\,
du &=& \mathbb{E} \biggl( \int_0^1 [Z^{(1)} ]^{k+1} \exp\biggl\{-
\frac{uZ^{(1)}}{a^{n-1}} \biggr\} \,du \biggr)
\\
&=&
a^{n-1} \mathbb{E} \biggl( \bigl[Z^{(1)} \bigr]^k \biggl( 1-\exp\biggl\{-
\frac{Z^{(1)}}{a^{n-1}} \biggr\} \biggr) \biggr)\\
&\leq&
a^{n-1} \mathbb{E} \bigl( \bigl[Z^{(1)} \bigr]^k \bigr) .
\end{eqnarray*}
Therefore,
\[
\mathbb{E}\bigl(\bigl[Z^{(n)} \bigr]^k\bigr) \leq\frac{3}{2}A+B + \frac{2ea}{a-1}
\mathbb{E} \bigl( \bigl[Z^{(1)} \bigr]^k
\bigr),
\]
and the expected result follows.
\end{pf}
\begin{Lemma} \label{lemtech2}
Assume that $\mu$ has finite $k$th moment. Then there exists a~positive constant $A$ (that depends on $k$ and
$\mathbb{E}([Z^{(1)} ]^k)$ but does not depend on $n$) such that
\[
\bigl|D^{k+1}g^{(n+1)}(u) \bigr| \leq A(1+u) + \frac{2ea}{(a-1)a^n}
\biggl|D^{k+1}g^{(1)}\biggl( \frac{u}{a^n} \biggr) \biggr|
\qquad\forall n \geq1, \forall u > 0 .
\]
\end{Lemma}
\begin{pf}
Iterating (\ref{eqrecursione}) we get
%
%
\begin{eqnarray} \label{eqiteration}
&&1-g^{(n+1)}(s)\nonumber\hspace*{-35pt}\\[-8pt]\\[-8pt]
&&\qquad=\biggl(1-g^{(1)}\biggl( \frac{s}{a^{n}} \biggr)
\biggr) \exp
\Biggl\{ \sum_{j=0}^{n-1} h^{(j+1)} \biggl( \frac{s}{a^{n-j}} \biggr)
\Biggr\}\qquad \forall s \geq0, \forall n\geq1 .\nonumber\hspace*{-35pt}
\end{eqnarray}
Hence, by the Leibniz formula,
%
%
\begin{eqnarray}
\label{damminome}
&&D^{k+1}g^{(n+1)}(s) \nonumber\\
&&\qquad = \sum_{\ell= 0}^{k+1}
\pmatrix{k+1\cr\ell} \biggl[D_s^{k+1-\ell}\biggl(g^{(1)}\biggl( \frac
{s}{a^n} \biggr)
-1 \biggr)\biggr]
\nonumber\\
&&\qquad\quad\hspace*{15.1pt}{}\times\Biggl[D_s^\ell\Biggl(\exp\Biggl\{ \sum_{j=0}^{n-1} h^{(j+1)} \biggl( \frac
{s}{a^{n-j}} \biggr) \Biggr\} \Biggr) \Biggr]\nonumber\\
&&\qquad= \frac{1}{a^{n(k+1)}}D^{k+1}g^{(1)}\biggl( \frac{s}{a^n} \biggr)
\exp\Biggl\{ \sum_{j=0}^{n-1} h^{(j+1)} \biggl( \frac{s}{a^{n-j}}
\biggr) \Biggr\} \\
&&\qquad\quad{} +
\sum_{\ell= 1}^{k} \pmatrix{k+1\cr\ell}
\frac{1}{a^{n(k+1-\ell)}} \biggl[D^{k+1-\ell} g^{(1)} \biggl(
\frac{s}{a^n} \biggr)\biggr]\nonumber\\
&&\hspace*{27.8pt}\qquad\quad{}\times\Biggl[ D_s^{\ell}
\Biggl(\exp\Biggl\{ \sum_{j=0}^{n-1} h^{(j+1)} \biggl( \frac
{s}{a^{n-j}} \biggr) \Biggr\}\Biggr)\Biggr] \nonumber\\
&&\qquad\quad{} + \biggl(g^{(1)}\biggl( \frac{s}{a^n} \biggr) -1 \biggr)
D_s^{k+1}\Biggl(\exp
\Biggl\{ \sum_{j=0}^{n-1} h^{(j+1)}
\biggl( \frac{s}{a^{n-j}} \biggr) \Biggr\}\Biggr) .\nonumber
\end{eqnarray}
%
In order to bound $D_s^\ell(\exp
\{ \sum_{j=0}^{n-1} h^{(j+1)} ( \frac{s}{a^{n-j}} )
\})$, one has to deal with
\[
\sum_{j=0}^{n-1} D_s^\ell h^{(j+1)} \biggl( \frac{s}{a^{n-j}}
\biggr) =
\sum_{j=0}^{n-1} \frac{1}{a^{\ell(n-j)}}D^\ell h^{(j+1)} \biggl(
\frac{s}{a^{n-j}} \biggr).
\]
%
By definition of $h^{(j+1)}$, we have
\[
\bigl|D^\ell h^{(j+1)}(s)\bigr| =
\mathbb{E} \bigl( \bigl[Z^{(j+1)} \bigr]^\ell e^{-sZ^{(j+1)}} \mathbh{1}_{1
\leq
Z^{(j+1)} < a} \bigr) \leq a^\ell.
\]
Therefore
%
\[
\Biggl| \sum_{j=0}^{n-1} D_s^{\ell} h^{(j+1)} \biggl( \frac{s}{a^{n-j}}
\biggr) \Biggr| \leq\sum_{j=0}^{n-1} \biggl( \frac{a}{a^{n-j}}
\biggr)^{\ell} = \sum_{j=0}^{n-1} ( a^{-\ell} )^j \leq
\frac{a^\ell}{a^\ell-1} \qquad\forall\ell\geq1.
\]
In turn, for any $\ell=1,2,\ldots,k+1$, since $h^{(j+1)}(u) \leq
h^{(j+1)}(0)$ for any $u$ and any~$j$,
%
%
\begin{eqnarray} \label{eqexp}
\Biggl|D_s^{(\ell)} \Biggl(\exp\Biggl\{ \sum_{j=0}^{n-1} h^{(j+1)}
\biggl( \frac{s}{a^{n-j}} \biggr) \Biggr\}\Biggr) \Biggr| &\leq& C \exp
\Biggl\{
\sum_{j=0}^{n-1} h^{(j+1)} \biggl( \frac{s}{a^{n-j}} \biggr) \Biggr\}
\nonumber\hspace*{-35pt}\\[-8pt]\\[-8pt]
&\leq& C \exp\Biggl\{ \sum_{j=0}^{n-1} h^{(j+1)}( 0) \Biggr\} \nonumber\hspace*{-35pt}
\end{eqnarray}
for some constant $C$ depending only on $k$, where we used the fact
that for any $F$ smooth enough and any $\ell\geq1$, it holds
\[
D_x^\ell e^{F(x)} = P_\ell
(DF(x),D^2 F(x),\ldots,D^\ell F(x)) e^{F(x)},
\]
where $P_\ell$ is a polynomial in the variables $X_1,\ldots,X_\ell$
of total degree $\ell$, whose coefficients belong to
$\{0,1,\ldots,(\ell+1)!\}$.

Now observe that, for any $\ell=1,2,\ldots,k$, by definition of
$g^{(1)}$ and since \mbox{$Z^{(1)} \geq1$},
%
%
\begin{equation}\label{cic}
\bigl| D^{k+1-\ell}g^{(1)} (u) \bigr| = \mathbb{E} \bigl(
\bigl[Z^{(1)} \bigr]^{k+1-\ell} e^{-uZ^{(1)}} \bigr) \leq\mathbb{E}
\bigl(\bigl[Z^{(1)} \bigr]^k\bigr) \qquad\forall u \geq0.\hspace*{-35pt}
\end{equation}
On the other hand, since $1-e^{-x} \leq x$ for $x \geq0$ and since
$Z^{(1)}\geq1$, one has
%
%
\begin{equation}\label{ciac}
\biggl| g^{(1)}\biggl( \frac{s}{a^n} \biggr) -1 \biggr| \leq
\frac{s}{a^n} \mathbb{E}\bigl(Z^{(1)}\bigr) .
\end{equation}
Hence, by (\ref{damminome}), (\ref{eqexp}), (\ref{cic}), (\ref{ciac})
and using the facts that $a>1$ and $h^{(j+1)}(u) \leq
h^{(j+1)}(0)$ for any $u$ and any $j$,
we end up with
%
%
\[
\bigl| D^{k+1}g^{(n+1)}(s) \bigr| \leq\biggl( C'(1+s) + \frac{1}{a^{n}}
\biggl|D^{k+1}g^{(1)}\biggl( \frac{s}{a^n} \biggr) \biggr| \biggr) \frac{1}{a^n}
\exp\Biggl\{ \sum_{j=0}^{n-1} h^{(j+1)} ( 0 ) \Biggr\}
\]
for some constant $C'$ depending only on $k$ and
$\mathbb{E}([Z^{(1)} ]^k)$. The expected result of Lemma
\ref{lemtech2} follows from Claim \ref{claim1} below.
\begin{claim} \label{claim1}
For any $n \geq1$ it holds
\[
\frac{1}{a^n} \exp\Biggl\{ \sum_{j=0}^{n-1} h^{(j+1)} ( 0 ) \Biggr\}
\leq
\frac{2ea}{a-1} .
\]
\end{claim}
\begin{pf}
Fix $s \in[0,1]$. Since $Z^{(1)} \geq1$ and $s \in[0,1]$, we have
\[
1-g^{(1)}\biggl( \frac{s}{a^n} \biggr) = 1 - \mathbb{E} \bigl(
e^{-({s}/{a^n}) Z^{(1)}} \bigr) \geq1 - e^{-{s}/{a^n}} \geq
\frac{s}{2a^n} ,
\]
where in the last line we used that $1-e^{-x} \geq\frac{x}{2}$ for
$x \in[0,1]$. Hence, we deduce from (\ref{eqiteration}) that
\begin{eqnarray*}
1 &\geq&1-g^{(n+1)}(s)=\biggl(1-g^{(1)}\biggl( \frac{s}{a^n} \biggr)
\biggr) \exp
\Biggl\{ \sum_{j=0}^{n-1} h_{j+1} \biggl( \frac{s}{a^{n-j}} \biggr)
\Biggr\}\\[-2pt]
&\geq&\frac{s}{2 a^n} \exp\Biggl\{ \sum_{j=0}^{n-1} h^{(j+1)} \biggl(
\frac{s}{a^{n-j}} \biggr) \Biggr\} .
\end{eqnarray*}
Now, by definition of $h^{(j)}$ and since $e^{-x}-1 \geq-x$ for any $x
\geq0$,
\begin{eqnarray*}
&&\sum_{j=0}^{n-1} h^{(j+1)} \biggl( \frac{s}{a^{n-j}} \biggr) \\[-2pt]
&&\qquad =
\sum_{j=0}^{n-1} \mathbb{E} \bigl( \bigl[e^{-({s}/{a^{n-j}})
Z^{(j+1)}}-1\bigr] \mathbh{1}_{1\leq Z^{(j+1)} < a} \bigr)
+ \sum_{j=0}^{n-1} h^{(j+1)} (0) \\[-2pt]
&&\qquad \geq -\sum_{j=0}^{n-1} \frac{s}{a^{n-j}}\mathbb{E}
\bigl(
Z^{(j+1)} \mathbh{1}_{1\leq Z^{(j+1)} < a} \bigr) +
\sum_{j=0}^{n-1} h^{(j+1)} (0) \\[-2pt]
&&\qquad \geq - \sum_{j=0}^{n-1} \frac{s}{a^{n-j-1}} + \sum_{j=0}^{n-1}
h^{(j+1)} (0) \\[-2pt]
&&\qquad= - \frac{s(a^n-1)}{a^{n-1}(a-1)} + \sum_{j=0}^{n-1}
h^{(j+1)} (0) .
\end{eqnarray*}
We deduce that
\[
\frac{1}{a^n} \exp\Biggl\{ \sum_{j=0}^{n-1} h^{(j+1)} ( 0 ) \Biggr\}
\leq
\frac{2}{s} \exp\biggl\{ \frac{s(a^n-1)}{a^{n-1}(a-1)} \biggr\}.
\]
Optimizing over $s \in[0,1]$ (choose
$s=\frac{a^{n-1}(a-1)}{a^n-1}$) finally leads to
\[
\frac{1}{a^n} \exp\Biggl\{ \sum_{j=0}^{n-1} h^{(j+1)} ( 0) \Biggr\}
\leq
\frac{2e(a^n-1)}{a^{n-1}(a-1)} \leq\frac{2ea}{a-1} .
\]
The claim follows.
\end{pf}

The proof of Lemma \ref{lemtech2} is complete.
\end{pf}

We can now prove Proposition \ref{propkmoment} for the special
choice $f(x)=x^k$.
\begin{Proposition} \label{propkmoment1}
\hspace*{5pt}Assume that $\mu$ as finite $k$th moment, that is,\break
$\mathbb{E}([Z^{(1)} ]^k)<\infty$. Then, in the same setting of
Proposition \ref{propkmoment}, $Z^{(\infty)}$ has also finite
$k$th moment. Moreover,
\[
\lim_{n \to\infty} \mathbb{E}\bigl(\bigl[Z^{(n)} \bigr]^k\bigr) =
\mathbb{E}\bigl(\bigl[Z^{(\infty)} \bigr]^k\bigr) .
\]
\end{Proposition}
\begin{pf}
Thanks to Lemma \ref{lemtech1} $\sup_{n \in\mathbb{N}_+}
\mathbb{E}([Z^{(n)} ]^k)<\infty$. Fix a decreasing sequence of
positive numbers $(s_m)_{m \in\mathbb{N}}$ that converges to $0$.
Since $x^k e^{-s_mx}$ is a~continuous bounded function on
$\mathbb{R}_+$, Theorem \ref{teo2} implies that\break $\lim_{n\to
\infty}\mathbb{E} ( [Z^{(n)} ]^k\times e^{-s_mZ_n} )
=\mathbb{E} ( [Z^{(\infty)} ]^k e^{-s_mZ^{(\infty)}} )$.
Hence, by Levi's theorem,
\[
\mathbb{E}\bigl(\bigl[Z^{(\infty)} \bigr]^k\bigr) = \lim_{m \to\infty}
\mathbb{E}
\bigl( \bigl[Z^{(\infty)} \bigr]^k e^{-s_mZ^{(\infty)}} \bigr) = \lim
_{m \to
\infty} \lim_{n\to\infty}\mathbb{E} \bigl( \bigl[Z^{(n)} \bigr]^k
e^{-s_mZ^{(n)}}
\bigr) < \infty.
\]
Hence $Z^{(\infty)}$ has finite $k$th moment.

Next, for any $s>0$, we write
\begin{eqnarray*}
\mathbb{E}\bigl(\bigl[Z^{(n)} \bigr]^k\bigr) & = &
\mathbb{E}\bigl( \bigl[Z^{(n)} \bigr]^k e^{-sZ^{(n)}} \bigr) + (-1)^k
\bigl( D^k g^{(n)}(0) - D^k g^{(n)}(s) \bigr) \\
& = &\mathbb{E} \bigl( \bigl[Z^{(n)} \bigr]^k e^{-sZ^{(n)}} \bigr) + (-1)^{k+1}
\int_0^{s} D^{k+1}g^{(n)}(u)\, du .
\end{eqnarray*}
Hence, thanks to Lemma \ref{lemtech2},
\begin{eqnarray*}
&&\bigl| \mathbb{E}\bigl(\bigl[Z^{(n)} \bigr]^k\bigr) - \mathbb{E}\bigl(\bigl[Z^{(\infty)} \bigr]^k\bigr)
\bigr| \\
&&\qquad \leq
\bigl| \mathbb{E} \bigl( \bigl[Z^{(n)} \bigr]^k e^{-sZ^{(n)}} \bigr) -
\mathbb{E}\bigl(\bigl[Z^{(\infty)} \bigr]^k\bigr) \bigr| + \int_0^s
\bigl|D^{k+1}g^{(n)}(u)\bigr|\, du \\
&&\qquad \leq \bigl| \mathbb{E} \bigl( \bigl[Z^{(n)} \bigr]^k e^{-sZ^{(n)}} \bigr) -
\mathbb{E}\bigl(\bigl[Z^{(\infty)} \bigr]^k\bigr) \bigr| + As(1+s) \\
&&\qquad\quad{} + \frac{2ea}{(a-1)a^{n-1}}
\int_0^s \biggl|D^{k+1}g^{(1)}\biggl( \frac{u}{a^{n-1}}
\biggr)\biggr|\,
du,
\end{eqnarray*}
where $A$ is a positive constant that depends on $k$ and
$\mathbb{E}([Z^{(1)} ]^k)$ but does not depend on $n$. Note that,
by definition of $g^{(1)}$, Fubini's theorem and then the dominated
convergence theorem,
\begin{eqnarray*}
&&\frac{1}{a^{n-1}}\int_0^s \biggl| D^{k+1}g^{(1)}\biggl(
\frac{u}{a^{n-1}} \biggr)\biggr| \,du \\
&&\qquad =
\frac{1}{a^{n-1}} \mathbb{E} \biggl( \int_0^s \bigl[Z^{(1)} \bigr]^{k+1}
e^{-(uZ^{(1)})/a^{n-1}}\,du \biggr) \\
&&\qquad = \mathbb{E} \bigl( \bigl[Z^{(1)} \bigr]^{k} \bigl[1-
e^{-(sZ^{(1)})/a^{n-1}} \bigr]
\bigr)
\mathop{\hbox to 1cm{\rightarrowfill}}_{n\to\infty} 0.
\end{eqnarray*}
By applying Theorem \ref{teo2} $\mathbb{E} ( [Z^{(n)} ]^k
e^{-sZ^{(n)}} )
\to\mathbb{E} ( [Z^{(\infty)} ]^k e^{-sZ^{(\infty)}}
)$ when $n$
tends to infinity. Therefore,
\begin{eqnarray*}
&&\lim_{n \to\infty} \bigl| \mathbb{E}\bigl(\bigl[Z^{(n)} \bigr]^k\bigr) -
\mathbb{E}\bigl(\bigl[Z^{(\infty)} \bigr]^k\bigr) \bigr| \\
&&\qquad\leq\bigl| \mathbb{E} \bigl(
\bigl[Z^{(\infty)} \bigr]^k e^{-sZ^{(\infty)}} \bigr) - \mathbb
{E}\bigl(\bigl[Z^{(\infty)} \bigr]^k\bigr) \bigr| +
As(1+s) \qquad\forall s>0 .
\end{eqnarray*}
The proof is completed by taking the limit as $s\downarrow0$.
\end{pf}
\begin{pf*}{Proof of Proposition \ref{propkmoment}}
Let $f$ be such that $|f(x)| \leq C+Cx^k$. For any $L \geq0$ we
define $f_L(x)=f(x)$ if $|x| \leq L$, and $f_L(x)=f(L)$ if $|x| \geq
L$. Note that by Proposition \ref{propkmoment1} $\mathbb{E} (
f(Z^{(\infty)}) ) < \infty$. Also, $|f(x) - f_L(x)| \leq
2C(1+x^k)\mathbh{1}_{|x| \geq L}$, and $f_L$ is bounded by
construction. It follows that
\begin{eqnarray*}
&&
\bigl| \mathbb{E}\bigl(f\bigl(Z^{(n)}\bigr)\bigr) - \mathbb{E}\bigl(f\bigl(Z^{(\infty)}\bigr)\bigr)
\bigr|\\
&&\qquad \leq
\bigl| \mathbb{E}\bigl(f\bigl(Z^{(n)}\bigr)\bigr) - \mathbb{E}\bigl(f_L\bigl(Z^{(n)}\bigr)\bigr) \bigr| +
\bigl| \mathbb{E}\bigl(f_L\bigl(Z^{(n)}\bigr)\bigr) - \mathbb{E}\bigl(f_L\bigl(Z^{(\infty)}\bigr)\bigr)
\bigr| \\
&&\qquad\quad{} +
\bigl| \mathbb{E}\bigl(f_L\bigl(Z^{(\infty)}\bigr)\bigr) - \mathbb{E}\bigl(f\bigl(Z^{(\infty)}\bigr)\bigr)
\bigr| \\
&&\qquad \leq
2C \mathbb{E}\bigl(\bigl(1 + \bigl[Z^{(n)} \bigr]^k\bigr) \mathbh{1}_{Z^{(n)} \geq L}
\bigr)
+ \bigl| \mathbb{E}\bigl(f_L\bigl(Z^{(n)}\bigr)\bigr) - \mathbb{E}\bigl(f_L\bigl(Z^{(\infty)}\bigr)\bigr)
\bigr| \\
&&\qquad\quad{} + 2C \mathbb{E}\bigl(\bigl(1 + \bigl[Z^{(\infty)} \bigr]^k\bigr) \mathbh
{1}_{Z^{(\infty)}
\geq L} \bigr) .
\end{eqnarray*}
Since $f_L$ is bounded and continuous, $\lim_{n \to\infty}
| \mathbb{E}(f_L(Z^{(n)})) - \mathbb{E}(f_L(Z^{(\infty)}))
| =0$. On the other hand, taking $L$ among the points of
continuity of the distribution function of $Z^{(\infty)}$, using
that $x \mapsto x^k \mathbh{1}_{x < L}$ and $x \mapsto\mathbh{1}_{x
< L}$ are bounded and Proposition~\ref{propkmoment1}, we have
\begin{eqnarray*}
&&\lim_{n \to\infty} \mathbb{E}\bigl(\bigl(1 + \bigl[Z^{(n)} \bigr]^k\bigr)
\mathbh{1}_{Z^{(n)} \geq L} \bigr) \\
&&\qquad = \lim_{n \to
\infty}\bigl\{ 1 - \mathbb{E}\bigl(\mathbh{1}_{Z^{(n)} < L} \bigr)
+
\mathbb{E}\bigl(\bigl[Z^{(n)} \bigr]^k\bigr) - \mathbb{E}\bigl(\bigl[Z^{(n)}
\bigr]^k \mathbh{1}_{Z^{(n)} < L} \bigr)\bigr\} \\
&&\qquad = 1 - \mathbb{E}\bigl(\mathbh{1}_{Z^{(\infty)} < L} \bigr) +
\mathbb{E}\bigl(\bigl[Z^{(\infty)} \bigr]^k\bigr) - \mathbb{E}
\bigl(\bigl[Z^{(\infty)} \bigr]^k \mathbh{1}_{Z^{(\infty)} < L} \bigr)\\
&&\qquad = \mathbb{E}\bigl(\bigl(1 + \bigl[Z^{(\infty)} \bigr]^k\bigr) \mathbh{1}_{Z^{(\infty
)} \geq L}
\bigr) .
\end{eqnarray*}
Therefore,
\[
\lim_{n \to\infty} \bigl| \mathbb{E}\bigl(f\bigl(Z^{(n)}\bigr)\bigr) -
\mathbb{E}\bigl(f\bigl(Z^{(\infty)}\bigr)\bigr) \bigr| \leq4C \mathbb{E}\bigl(\bigl(1 +
\bigl[Z^{(\infty)} \bigr]^k\bigr) \mathbh{1}_{Z^{(\infty)} \geq L} \bigr) .
\]
Now, since $\mathbb{E}([Z^{(\infty)} ]^k) < \infty$
and by
Lebesgue's theorem, the right-hand side of the latter tends to $0$
when $L$ tends to infinity. This achieves the proof.~%
\end{pf*}

\begin{appendix}

\section{\texorpdfstring{Proof of Lemma \lowercase{\protect\ref{prelimmainth}}}{Proof of Lemma 2.17}}
\label{appA}

We provide the proof of Lemma \ref{prelimmainth}.
Part (i) follows immediately from Lem\-ma~\ref{gola}.

As far as part (ii) is concerned, if the mean of $\mu$ is finite, it
is trivial to check that limit (\ref{limarrosto}) holds with
$c_0=1$. Indeed, by the Dominated Convergence theorem, both $-g'(s)$
and $(1-g(s))/s$ converge to the mean as $s\downarrow0$. Let us now
assume that the mean is infinite and that for some ${\alpha}\in[0,1]$
$\bar F(x):= \mu((x,\infty))=x^{-{\alpha}} L(x)$ for
some slowly
varying function $L$. Notice that $\bar F(1-)=1$. Let
\begin{eqnarray*}
Z_A(s)&=& \int_{[As,\infty)} ( e^{-y}-y e^{-y} ) y^{-{\alpha}} L
(y/s) \,dy ,\\
W_A(s) &=& \int_{[A s,\infty)} e^{-y} y^{-{\alpha}} L (y/s) \,dy .
\end{eqnarray*}
Using integration by parts and the change of variables $y=sx$, given
$A>1$ we can write
%
%
\begin{eqnarray}\label{polipo1}
-s g'(s)& = &-\int_{[1,\infty)} sx e^{-s x} \,d \bar F (x) \nonumber\\
& = &-\int_{[1,A)} sx e^{-s x} \,d \bar F (x)
-s x e^{-s x} \bar
F(x) \bigg|_{A-}^\infty\nonumber\\
&&{}+ \int_{[A,\infty)} (e^{-sx}-s x e^{-sx} ) \bar
F(x) s \,dx\\
& = &-\int_{[1,A)} sx e^{-s x} \,d \bar F (x)+ s A e^{-s A } \bar F
(A-) +s^{\alpha} Z_A(s) \nonumber\\
&=& \cE+s^{\alpha} Z_A(s) ,\nonumber
\end{eqnarray}
where the error term $\cE$ satisfies $|\cE| \leq sA $.
Similarly, we can
write
%
%
\begin{eqnarray}\label{polipo2}
1 -g(s)&=&1+\int_{[1,\infty)} e^{-s x} \,d \bar F (x) \nonumber\\
&=&
1+\int_{[1,A)} e^{-s x} \,d \bar F (x)-e^{-s A}\bar F (A-) + s^{\alpha}
W_A(s)\\
&=&\cE'+ s^{\alpha} W_A (s) \nonumber
\end{eqnarray}
%
and via a Taylor expansion we get $|\cE'|\leq C(A) s $ for a suitable
positive constant $C(A)$ depending only on $A$. Since the mean is
infinite, the monotone convergence theorem and De l'Hopital rule
imply that
%
%
\begin{equation}\lim_{s \downarrow0} \bigl(1-g(s)\bigr)/s= \lim_{s
\downarrow0} - g'(s)= \infty.
\end{equation}
Comparing the above limits with (\ref{polipo1}) and (\ref{polipo2})
we deduce that both~$s^{\alpha} Z_A(s)$ and $s^{\alpha} W_A(s)$ must
diverge as $s$
goes to zero. In particular, limit~(\ref{limarrosto}) is equivalent to the limit
%
%
\begin{equation}\label{polentaX}
\lim_{A \uparrow\infty} \lim_{s\downarrow0 }
\frac{Z_A(s)}{W_A(s) }={\alpha} .
\end{equation}

As proved in
\cite{Fe2} (see Section VIII.9 there), $L$ is slowly varying at
$\infty$ if and
only if it is of the form
%
%
\begin{equation}\label{grufalo}
L(x)= a(x) \exp\biggl\{ \int_{1} ^x \frac{ \e(y)}{y} \,dy \biggr\},
\end{equation}
where $\e(x)\to0$ and $a(x) \to c < \infty$ as $x \to\infty$.
In particular, given ${\delta}>0$,
for any $x$ large enough $ x^{-{\delta}} \leq L(x) \leq x^{{\delta}}$.
Since in (\ref{grufalo})
$\e(x) \to0$ and $a(x) \to c < \infty$, for any $ {\delta}>0$ there
exists $A>0$ such that $c/2 \leq a(x) \leq2c$ and $|\e(x)| \leq
{\delta}
$ for $x \geq A$. Thus, for any $s<1/A$ the integral representation
(\ref{grufalo}) implies that
%
%
\begin{equation}\label{ranaevolpe}
\frac{1}{4} ( y^{-{\delta}} \wedge y^{{\delta}} ) \leq\frac{
L(y/s)}{L (1/s) } \leq4 ( y^{-{\delta}} \lor y ^{\delta}),\qquad y
\geq
As .
\end{equation}

We now distinguish two cases:\vspace*{8pt}

$\bullet$ Case ${\alpha} \in[0,1)$. Choose ${\delta}>0$
such that
${\alpha}+{\delta}<1$, $A>1$ and $s \geq1/A$. The Dominated
Convergence theorem
together with (\ref{ranaevolpe}) implies that
%
%
\begin{eqnarray}
\label{maradona1}
\lim_{s \downarrow0} Z_A (s) / L(1/s)&=& \int_0^\infty( e^{-y}-y
e^{-y} ) y^{-{\alpha}} dy ,\\
\label{maradona2}
\lim_{s \downarrow0} W_A (s) / L(1/s)&=& \int_0^\infty e^{-y}
y^{-{\alpha}} dy .
\end{eqnarray}
At this point, (\ref{polentaX}) follows from (\ref{maradona1}),
(\ref{maradona2}) and a trivial calculation.

$\bullet$ Case ${\alpha} =1$.
It is convenient to write
\[
Z_A(s)= W_A(s)-T_A(s) ,\qquad T_A(s):=\int_{[A s,\infty)} e^{-y}
L (y/s)\,
dy .
\]
Then, (\ref{polentaX}) follows if we can prove that
%
%
\begin{equation}\label{polentabis}
\lim_{A \uparrow\infty} \limsup_{s\downarrow0 }
\frac{T_A(s)}{W_A(s) }=0 .
\end{equation}
Given ${\delta}>0$ we take $A>1$ and $s \leq1/A$ assuring
(\ref{ranaevolpe}). Then we can bound
\[
\frac{T_A(s)}{W_A(s)}\leq\frac{ 4 \int_0^\infty e^{-y} (y^{{\delta}}
\lor y^{-{\delta}} ) \,dy }{ (1/4e) \int_{As}^1 ({1 }/{y})
(y^{-{\delta}}
\wedge y^{{\delta}} )}= \frac{C}{ \int_{As}^1 y^{{\delta}-1} \,dy }=
\frac{{\delta}{C}}{(1-(As)^{\delta})} .
\]
The above bound trivially implies (\ref{polentabis}).

\section{\texorpdfstring{An example of interval law not satisfying~(\lowercase{\protect\ref{limarrosto}})}{An example of interval law not satisfying (2.14)}}\label{controcampo}

We provide here an example of a law which does not satisfy (\ref
{limarrosto}) and therefore does not fulfill the hypothesis under which
our main Theorem \ref{teo2} holds. Furthermore we
have numerically analyzed the
set of identities (\ref{brumsporco}) with $t^{(1)}$ corresponding to
this choice
for the initial distribution. The results for the corresponding
function $U^{(n)}$, displayed in Figure \ref{figsimulazione},
%
%
\begin{figure}

\includegraphics{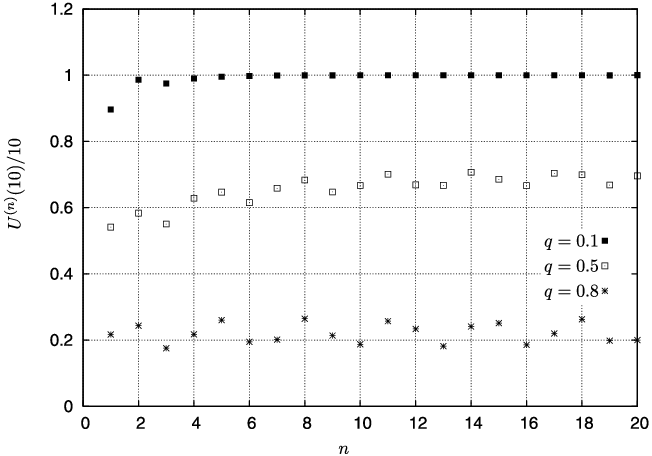}

\caption{We consider an HCP with $d^{(n)}=2^{n-1}$ (the relevant
choice to describe East model) with initial law specified in
Proposition \protect\ref{B1} and parameter $q:=1-p=0.1, 0.5,0.8$. In the first
case the limit (\protect\ref{limarrosto}) exists with $c_0=1$. Instead, for $q=0.5, 0.8$
as proven in Proposition \protect\ref{B1} the limit (\protect\ref{limarrosto}) does not exists. We
plot here $U^{(n)}(x)/x$ for $x=10$ as a~function of $n$. The data
indicate clearly that for $p=0.1$ $U^{(n)}(10)/10$ converges to $1$
as we have proven [see Theorem \protect\ref{teo2}
and especially the comment below formula (\protect\ref{miserve})].
Instead of the other two choices of the parameters,
$U^{(n)} (10)/10$ has an oscillating behavior which strongly indicates
the nonexistence of the
limit for $U^{(n)}(x)$, hence for $g^{(n)}(s)$. We have checked that an
analogous behavior
occurs for different choices of $x$. Note that if we were instead
considering for the same initial distribution but a different choice of
$d^{(n)}$ (satisfying the basic hypothesis $d^{(n)}\to\infty$ for
$n\to\infty$) we would get the same behavior. Indeed if we consider for
example the choice relevant for the Paste-all model, $d^{(n)}=n$, then
the plot of $U^{(n)}(10)/10$ would exactly be the same as above but
with $n$ replaced by $\log_2(n)$ in the $x$-axis (and in this case our
data would cover $2^{20}$ epochs).} \label{figsimulazione}
\end{figure}
strongly suggest that in this case the measure $\mu^{(n)}$ does not
have a well-defined limiting behavior as $n\to\infty$.
\begin{Proposition}
\label{B1}
Let $G$ be a geometric random variable with parameter $p=1-e^{-\lambda
}$, $\lambda\in(0,1)$. Define $X = e^G$
and $g(s)= \mathbb{E} ( e^{-sX} )$, $s \geq0$.
Then, $\lim_{s \to0} \frac{sg'(s)}{1-g(s)}$ does not exist. More precisely,
for any $\alpha\in[0,1)$ and any $n \in\mathbb{N}$, set
$s_n=e^{-n-\alpha}$.
Then $\lim_{n \to\infty} \frac{s_ng'(s_n)}{1-g(s_n)} =: L_\alpha$ exists,
and ${\alpha} \rightarrow L_\alpha$ is a nonconstant function.
\end{Proposition}

Note that the constraint ${\lambda} \in(0,1)$ is equivalent to the fact
that $X$ has infinite mean.
\begin{pf*}{Proof of Proposition \ref{B1}}
Fix $\alpha\in[0,1)$ and set $s_n=e^{-n-\alpha}$.
Since $\mathbb{P}(G=k)=p(1-p)^{k-1}$ for $k\geq1$, we have
$\bar{F}(x)= \mathbb{P}(X \geq x) = e^{-\lambda\lceil\ln x \rceil+
\lambda}
= x^{-\lambda} e^{ \lambda\{ \ln x \}}$
where $\lceil z \rceil= z + 1-\{z\}$ is the ceiling function of $z$
(i.e., the smallest integer greater than or equal to $z$), and $\{z\}
\in(0,1]$ is the fractional part of $z$. Note that $\bar{F}(e)=1$.

Then, using an integration by parts and the change of variables $u=sx$,
we have
\begin{eqnarray*}
-g'(s)
& = &
\mathbb{E} ( Xe^{-sX} ) = - \int_{[e,\infty)} xe^{-sx}
\,d\bar{F}(x)\\
&=&
e^{1-se} + \int_{[e,\infty)} (1-sx)e^{-sx} \bar{F}(x) \,dx \\
& = &
e^{1-se} + \frac{1}{s}\int_{[es,\infty)} (1-u)e^{-u} \bar{F}
\biggl(\frac{u}{s} \biggr) \,du \\
& = &
e^{1-se} + s^{\lambda-1} \int_{[es,\infty)} (1-u)u^{-\lambda}
e^{-u} e^{\lambda\{ \ln u - \ln s \}} \,du .
\end{eqnarray*}

Similarly
\[
1-g(s)= 1 + \int_{[e,\infty)} e^{-sx} \,d\bar{F}(x)
= 1-e^{-se} +s^\lambda\int_{[es,\infty)} u^{-\lambda} e^{-u}
e^{\lambda\{ \ln u - \ln s \}} \,du .
\]
Since $\{\ln u - \ln s_n \} = \{\ln u + \alpha\}$ for any $n$, it
follows that [recall that $\lambda\in(0,1)$]
\begin{eqnarray*}
\lim_{n \to\infty} \frac{-s_ng'(s_n)}{1-g(s_n)}
& = &
\lim_{n \to\infty}
\frac{s_ne^{1-s_ne} + s_n^{\lambda} \int_{[es_n,\infty)}
(1-u)u^{-\lambda} e^{-u} e^{\lambda\{ \ln u + \alpha\}}\,
du}{1-e^{-s_ne} + s_n^\lambda\int_{[es_n,\infty)} u^{-\lambda}
e^{-u} e^{\lambda\{ \ln u + \alpha\}} \,du} \\
& = &
\frac{\int_{(0,\infty)} (1-u)u^{-\lambda} e^{-u} e^{\lambda\{ \ln
u + \alpha\}} \,du}
{\int_{(0,\infty)} u^{-\lambda} e^{-u} e^{\lambda\{ \ln u + \alpha
\}} \,du}=:L_{\alpha} .
\end{eqnarray*}

Suppose $L_{\alpha}$ to be equal to $1-C$ for all ${\alpha}$. Then,
by the change of variable $v= u/{\beta}$ where $ {\beta}=e^{-{\alpha
}}$, we can write
%
%
\begin{eqnarray}\label{ciuccio}
1- L_{\alpha} &=& \frac{{\beta} \int
_0^\infty v^{-{\lambda}+1}e^{-{\beta}{v}} e^{{\lambda} \{\ln v\}
}\,dv}{ \int_0^\infty v^{-{\lambda}}e^{-{\beta}{v}} e^{{\lambda} \{
\ln v\} }\,dv}\nonumber\\[-8pt]\\[-8pt]
&=&
\frac{(B+1) \int_0^\infty v^{-{\lambda}+1}e^{-v} e^{-B v}
e^{{\lambda} \{\ln v\}
}\,dv}{ \int_0^\infty v^{-{\lambda}}e^{-B v}e^{-v} e^{{\lambda} \{\ln
v\} }\,dv}=C,\nonumber
\end{eqnarray}
where $B={\beta}-1$. Consider now the functions $f$ and $g$ on $\bbD
=\{z
\in\bbC\dvtx|z|<1\}$ defined as
\begin{eqnarray*}
f(z)&=&\int_0^\infty v^{-{\lambda}+1}e^{-v} e^{-z v} e^{{\lambda} \{
\ln v\}
}\,dv ,\\
g(z)&=&\int_0^\infty v^{-{\lambda}}e^{-z v}e^{-v} e^{{\lambda} \{\ln
v\} }\,dv .
\end{eqnarray*}
By Fubini's theorem and the series expansion of the exponential
function, one gets that $f$ and $g$ are holomorphic functions on
$\bbD$. Hence, the same holds for the function
$ H (z)= (1+z) f(z)- C g(z)$. Due to the last identity in (\ref{ciuccio})
we get that $H$ is zero on a subinterval of
the real line. Due to a theorem of complex analysis, the zeros of a
nonconstant holomorphic function are isolated points. As a
byproduct, we get that $H(z)= 0$ for all $z\in\bbD$.

Writing the power expansion of $H(z)$ around $z=0$ and using that
$H\equiv0 $, we get
%
\[
\int_0^\infty v^{-{\lambda}+1+n} e^{-v} e^{{\lambda} \{\ln v\} }\,dv=
(C-1) \int_0^\infty v^{-{\lambda}+n} e^{-v} e^{{\lambda} \{\ln v\}
}\,dv \qquad
\forall n \geq0 .
\]
Note that it must be $C>1$.
By iteration we get
%
%
\begin{equation}\label{daichecelafai}\int_0^\infty v^{-{\lambda}+n}
e^{-v} e^{{\lambda} \{\ln v\} }\,dv \leq
(C-1)^n,\qquad n\geq0.
\end{equation}
On the other hand the above left-hand side is larger than
$ \G(n+1-{\lambda}):=\int_0^\infty v^{-{\lambda}+n} e^{-v} \,dv$.
Iterating the identity $\G(z+1)=z \G(z) $ we get that
\[
\G(n+1-{\lambda}) = (n-{\lambda})(n-{\lambda}-1) \cdots(1-{\lambda
}) \G(1-{\lambda}) ,
\]
which leads to a contradiction with
(\ref{daichecelafai}).
\end{pf*}

\section{$\bbZ$-stationary SPPs}\label{puntini}
$\bbZ$-stationary SPPs and stationary SPPs have many common
features. In this Appendix we point out some properties of
$\bbZ$-stationary SPPs, whose proof (only sketched here) follows by
suitably adapting the arguments used in the continuous case.

Let us suppose that $\cQ$ is the law of a $\bbZ$-stationary SPP,
nonempty a.s. We derive here\vadjust{\goodbreak} some identities relating $\cQ$ to the
conditional probability measure $\cQ_0:=\cQ(\cdot|0\in\xi)$. These
identities are similar to the ones relating the law of a stationary
SPP to its Palm distribution~\cite{DV,FKAS}. Since all
random sets are included in $\bbZ$ it is more natural to work with
the subspaces of $\cN$ defined as
%
%
\begin{eqnarray}
\cN^{\bbZ}&=&\{ \xi\in\cN\dvtx\xi\subset\bbZ\} , \\
\cN^{\bbZ}_0&=&\{ \xi\in\cN^\bbZ\dvtx0\in\xi\} .
\end{eqnarray}
Moreover, we prefer to write ${\tau}_x \xi$ instead of $\xi-x$.

Similarly to (\ref{scimmiazimbo}) we get a simple relation
characterizing $\cQ$ by means of~$\cQ_0$:
\begin{Lemma} Given a nonnegative measurable function $f$ on $\cN
^\bbZ$ it holds
%
%
\begin{equation}\label{zimbozeta}
\int\cQ(d\xi) f(\xi) = \cQ(0\in\xi) \int\cQ_0(d \xi) \sum
_{x=0} ^{x_1(\xi)-1 } f( {\tau}_x\xi) .
\end{equation}
\end{Lemma}
\begin{pf}
From the $\bbZ$-stationarity of $\cQ$ it is simple to
derive for all measurable functions $g \dvtx\cN^{\bbZ}_0\to[0,\infty)$
and $t\in\bbN_+$ that
%
%
\begin{equation}\label{fkas128}
\int\cQ_0(d\xi) g(\xi) =\frac{1}{t \cQ(0\in\xi) }\int\cQ(d\xi)
\sum_{y\in\xi\cap(0, t]} g({\tau}_y \xi).
\end{equation}
Given a measurable map $v\dvtx\bbZ\times\cN^\bbZ_0\to[0,\infty)$,
setting $g(\xi)= \sum_{x\in\bbZ} v(x,\xi)$ in the above identity,
we get
%
%
\begin{equation}\label{wordpress2010}\qquad
\sum_{x\in\bbZ}\int\cQ_0 (d \xi) v(x, \xi) = \frac{1}{t \cQ
(0\in
\xi) } \int\cQ(d\xi) \sum_{x\in\bbZ}\sum_{y\in\xi\cap(0,t]}
v(x, {\tau}_y \xi) .
\end{equation}
Reasoning as in the proof of (1.2.10) in~\cite{FKAS}, for any
measurable function $w\dvtx\bbZ\times\bbZ\times\cN^\bbZ\to
[0,\infty)$ we get
%
%
\begin{eqnarray}\label{fkas1210}
&&\sum_{x\in\bbZ}\sum_{y\in\bbZ}\int\cQ(d\xi) w (x,y, {\tau}_y
\xi)
\mathbh{1}(y\in\xi) \nonumber\\[-8pt]\\[-8pt]
&&\qquad=\sum_{x\in\bbZ}\sum_{y\in\bbZ}\int
\cQ(d\xi) w(y,x, {\tau}_y \xi) \mathbh{1}(y\in\xi) .\nonumber
\end{eqnarray}
Combining (\ref{wordpress2010}) with (\ref{fkas1210}) where
$w(x,y,\xi)= v(x,\xi) \mathbh{1}(y\in(0,t]) /t $ we get
%
%
\begin{equation}\label{fkas129}\quad
\sum_{x\in\bbZ} \int\cQ_0(d\xi)v(x,\xi)= \frac{1}{\cQ(0\in
\xi)}\sum_{x\in\bbZ} \int\cQ(d\xi)
v(x,{\tau}_x \xi)\mathbh{1}(x\in\xi) .
\end{equation}
At this point, we take
\[
v(x, \xi) := \mathbh{1} \bigl( x= x_0 ({\tau}_{-x}
\xi) \bigr) f({\tau}_{-x}\xi)
\]
[if $\xi= \varnothing$ we set $v(x, \xi)=0$].
Note that $v(x, {\tau}_x \xi) =\mathbh{1} ( x= x_0 ( \xi) ) f(\xi) $,
thus implying together with (\ref{fkas129}) that
%
%
\begin{eqnarray}\label{rita1}
\int\cQ(d\xi) f(\xi)&=&\sum_{x\in\bbZ} \int\cQ(d\xi)
v(x,{\tau}_x \xi)\mathbh{1}(x\in\xi)\nonumber\\[-8pt]\\[-8pt]
&=& \cQ(0\in\xi)\sum_{x\in
\bbZ} \int
\cQ_0(d\xi)v(x,\xi) .\nonumber
\end{eqnarray}
In order to understand the last integral, take $\xi\in
\cN^{\bbZ}_0$. Then $ x= x_0 ({\tau}_{-x} \xi)$ if and only if
$0\leq-x
< x_1(\xi)$. Therefore, changing at the end $x$ into $-x$, we get
%
%
\begin{equation}\label{rita2}\qquad
\sum_{x\in\bbZ} v(x,\xi)=
\mathop{\sum_{x\in\bbZ:}}_{0 \leq-x< x_1 (\xi)}
f({\tau}_{-x} \xi)= \sum_{x=0}^{x_1 (\xi)-1} f({\tau}_x\xi) \qquad
\forall\xi\in\cN^{\bbZ}_0 .
\end{equation}
Combining (\ref{rita1}) and (\ref{rita2}) we get the thesis.
\end{pf}

Taking $f=1$ in (\ref{zimbozeta}) we deduce
that $x_1 (\xi)$ must have finite mean w.r.t.~$\cQ_0 $. In
particular, if $\cQ_0$ is the law of the renewal SPP on $\bbZ$
containing the origin and with domain length $\mu$ [i.e.,
$\cQ_0$ is the law of $\operatorname{Ren}_0 (\mu)$], then $\mu$ must have
finite mean. On the other hand, given $\mu$ probability measure on
$\bbN_+$ with finite mean, identity (\ref{zimbozeta})
uniquely determines the probability measure~$\cQ$ if~$\cQ_0$ is
defined as the law of $\operatorname{Ren}(\mu\tc0)$. One can then prove
that the so-defined $\cQ$ is the law of a $\bbZ$-stationary SPP
and that $\cQ_0= \cQ(\cdot|0\in\xi)$. Finally, as in the continuous
case, relation (\ref{zimbozeta}) allows us to derive a simple
description of $\bbZ$-stationary renewal SPPs similar to the one
mentioned after Definition~\ref{fataturchina}. We leave the details
to the interested reader.

\section{Exchangeable SPPs}\label{puntiniscambio}

We endow the
space ${\Omega}=(0,\infty)^{\bbN_+}$ of
sequences of positive numbers with the product topology, and we denote
by $\cB$ its Borel
$\s$-field. We write a~generic element of ${\Omega}$ as ${\omega
}=({\omega}_n \dvtx n \in
\bbN_+)$. Let $\cE_n$ be the $\s$-subfield generated by the events
that are invariant under permutations of $\bbZ$ fixing all points $x
\in\bbN_+$ with $x>n$. Let $\cE:= \bigcap_{n=1}^\infty\cE_n$ be the
\textit{exchangeable} $\s$-field. Since~${\Omega}$ is a standard Borel set,
given a~probability measure $Q$ on ${\Omega}$ there exists a~regular
conditional probability associated to $\cE$, that is, a measurable map
$\rho_Q\dvtx{\Omega} \times\cB\to[0,1]$ satisfying the following
properties:
\begin{longlist}
\item for each $A\in\cB$, $\rho_Q(\cdot,A)$ is a version of
$P(A|\cE)$;

\item for $Q$-a.e. ${\omega}\in{\Omega}$, $\rho_Q({\omega},
\cdot)$ is a
probability measure on $({\Omega}, \cB)$.
\end{longlist}

Due to de Finetti's theorem, if $Q$ is an exchangeable probability
measure on~${\Omega}$, then for $Q$-a.e. ${\omega}$ the\vadjust{\goodbreak}
measure $\rho_Q ({\omega}, \cdot)$ is a product probability measure on
${\Omega}$. The inverse implication is trivially true; hence de
Finetti's theorem provides a characterization of the exchangeable
probability measures on ${\Omega}$.

Suppose that $\cQ$ is a left-bounded exchangeable SPP containing
the
origin (see Definition \ref{baratto}). By definition, $\cQ$ has
support on the subspace
$\Xi\subset\cN$ given by the configurations $\xi\in\cN$ empty on
$(-\infty,0)$, containing the origin and given by a sequence of
points $x_k(\xi)$, $k\in\bbN$, diverging to $+\infty$. We can
define the measurable injective map $\Psi\dvtx\Xi\to{\Omega}$, with
$\Psi
(\xi)= {\omega}$ and ${\omega}_n= x_n(\xi)- x_{n-1}(\xi)$. We call
$Q$ the
measure $\cQ\circ\Psi^{-1}$. Trivially, $Q$ is an exchangeable
measure on ${\Omega}$; hence we can apply de Finetti's theorem and get
$ Q(A) = \int_{\Omega} Q(d{\omega}) \rho_Q({\omega}, A)$ for all
\mbox{$A\in\cB$}, where
$\rho_Q({\omega},\cdot)$ is a product probability measure. Since
trivially $\rho_Q({\omega},\cdot)$ has support on $\Psi(\Xi)$, the
pull-back of $\rho_Q({\omega}, \cdot)$ is a well-defined probability
measure on $\Xi$ corresponding to the law of $\operatorname{Ren}({\delta}_0,
\mu_{\omega})$. As
byproduct, we get
%
%
\begin{equation}\label{decomposizione} \cQ(\cA)= \int_{\Omega}
Q(d{\omega})
\operatorname{Ren}({\delta}_0,
\mu_{\omega}) [\cA] ,\qquad \cA\subset\cN\mbox{ measurable} .
\end{equation}

The above decomposition allows us to extend our results stated in
Section~\ref{annarella} to right exchangeable SPPs containing the
origin. We give only some comments, leaving the details to the
interested reader. Consider, for example, the HCP starting from $\cQ$,
that is, $\xi^{(1)}(0)$ has law $\cQ$.
By applying inductively Theorem \ref{teo1} we get that, given $n\geq
1$ and $t\in[0,\infty]$,
the law of~$\xi^{(n)}(t)$ conditioned to the fact that $0\in\xi^{(n)}(t)$
has the integral representation\looseness=-1
\[
\int_{{\Omega}} Q(d{\omega}) \operatorname{Ren}\bigl({\delta}_0,[\mu_{\omega
}]^{(n)}_t \bigr)
\]\looseness=-1
for a suitable probability measure $[\mu_{\omega}]^{(n)}_t$ on
$(0,\infty)$.

In particular, if each $\mu_{\omega}$ satisfies the limit
(\ref{limarrosto}) for some constant $c_0({\omega})$, we get the
following: fixed $k\geq1$, the rescaled random variable
$[x_k^{(n)}(0)- x_{k-1}^{(n) } (0)]/ d_{\min} ^{(n)}$, defined
for the HCP starting with law $\cQ$ and conditioned to the event
$\{0\in\xi^{(n)}(0)\}$, weakly converges to a random variable whose
Laplace transform $g^{(\infty)}$ is given by
\[
g^{(\infty)} (s)= \int_{{\Omega}} Q(d{\omega}) g^\infty
_{c_0({\omega}) } (s) ,
\]
where $g^\infty_{(c_0)}$ has been defined in (\ref{macedonia}). Note
that new limit laws emerge in this way.


Let us now pass to stationary exchangeable SPPs. One can formulate
de Finetti's theorem also for exchangeable laws on the space
${\Omega}'=(0,\infty)^\bbZ$ of two-sided sequences of positive numbers.
At the end we get that a stationary SPP, nonempty a.s. and with
finite intensity, is exchangeable if and only if its Palm
distribution $\cQ_0$ satisfies
%
%
\begin{equation}\label{putrefatto} \cQ_0(\cA)= \int_{{\Omega}'}
Q(d{\omega})
\operatorname{Ren}_0( \mu_{\omega}) [\cA] ,\qquad \cA\subset\cN\mbox{
measurable} ,
\end{equation}
where (i) $\mu_{\omega}$ is a probability measure on $(0,\infty)$; (ii)
for any $\cA\subset\cN$ measurable the map ${\Omega}'\ni{\omega
}\to
\operatorname{Ren}_0( \mu_{\omega}) [\cA]$ is measurable (thus implying
that the
map ${\omega}\to\mu_{\omega}$ is measurable); (iii) $Q$ is the
image of the
law $\cQ$ of the SPP under the map $ \cN^\infty_0 \to
(0,\infty)^\bbZ$, mapping $\xi$ in $( x_{k}(\xi)-x_{k-1}(\xi)\dvtx
k\in
\bbZ)$ [recall~(\ref{camilloha3figlie})].

Using (\ref{putrefatto}) and (\ref{scimmiazimbo}) we conclude that
%
%
\begin{equation}\label{eccoloarriva!} \cQ(\cA)= \int_{{\Omega}'}
Q(d{\omega})
\operatorname{Ren}( \mu_{\omega}) [\cA] ,\qquad \cA\subset\cN\mbox{
measurable} .
\end{equation}
The above decomposition of $\cQ$ allows us to extend our limit theorems
to the HCP starting with law $\cQ$, that is, from a stationary
exchangeable SPPs.
In particular,
$\xi^{(n)}(t)$ will be a stationary exchangeable SPP for all $n\geq
1$ and all $t\in[0,\infty]$. In addition, for $k\not=1$, as $n\to
\infty$ the rescaled random variable $[x_k^{(n)}(0)- x_{k-1}^{(n)
} (0)]/ d_{\min} ^{(n)}$ weakly converges to the random variable~$Z^{(\infty)} _1$ introduced in Theorem \ref{teo2}.

\section{A combinatorial lemma on exchangeable probability measures}\label{apppp}
The next combinatorial lemma has been used in Section
\ref{fondamenta}.

\begin{Lemma}\label{lemexchangeable}
Let $m_k$ be an exchangeable probability measure on $S^k$,
$S=(0,\infty)$; that is, $m_k$ is left invariant by any permutation of
the coordinates $(s_1,\ldots,s_k) \in S^k$. Call $m$ the marginal of
$m_k$ along a coordinate (it does not depend on the coordinate).
Then, for any
bounded function $f\dvtx S\to\mathbb{R}$, and any bounded function
$g\dvtx S\to
(0,\infty)$, it holds
\begin{eqnarray*}
&&\mbox{\textup{(a)}}\quad\mathbb{E}_{m_k} \Biggl( \frac{g (s_1) }{g (s_1) + \cdots+
g(s_{k}) } \prod_{i=2}^{k-1} \frac{ g( s_i) }{ \sum_{j=i}^{k-1} g(
s_j) } \prod_{i=1}^{k} f(s_i)\Biggr) = \frac{ \mathbb{E}_m
(f)^k}{k \cdot(k-2)!},
\\
&&\mbox{\textup{(b)}}\quad\hspace*{130.3pt}\mathbb{E}_{m_k} \Biggl( \prod_{i=1}^{k} \frac{ g( s_i) f(s_i)}{
\sum_{j=i}^{k} g( s_j) } \Biggr) = \frac{ \mathbb{E}_m
(f)^k}{k!} .
\end{eqnarray*}
\end{Lemma}
\begin{pf}
We will give only the proof of point (a) which is a bit harder.
The proof of point (b) follows essentially the same lines; details
are left to the reader.

Since the law $m_k$ is left invariant by any permutations of the
coordinates $(s_1,\ldots,s_k) \in S^k$, we have
\begin{eqnarray*}
&&\mathbb{E}_{m_k} \Biggl( \frac{g (s_1) }{g (s_1) + \cdots+
g(s_{k}) }
\prod_{i=2}^{k-1} \frac{ g( s_i) }{ \sum_{j=i}^{k-1} g( s_j) } \prod
_{i=1}^{k} f(s_i)\Biggr) \\
&&\qquad = \frac{1}{k!} \sum_{\sigma\in\mathcal{S}_k}
\mathbb{E}_{m_k} \Biggl( \frac{g (s_{\sigma(1)}) }{g (s_{\sigma(1)})
+ \cdots+ g(s_{\sigma(k)}) }
\prod_{i=2}^{k-1} \frac{ g( s_{\sigma(i)}) }{ \sum_{j=i}^{k-1} g(
s_{\sigma(j)}) } \prod_{i=1}^{k} f\bigl(s_{\sigma(i)}\bigr) \Biggr) \\
&&\qquad = \frac{1}{k!} \mathbb{E}_{m_k} \Biggl( f(s_1)\cdots
f(s_k)\sum_{\sigma\in\mathcal{S}_k} \frac{g (s_{\sigma(1)}) }{g
(s_1) + \cdots+ g(s_{k}) } \prod_{i=2}^{k-1} \frac{ g(
s_{\sigma(i)}) }{ \sum_{j=i}^{k-1} g( s_{\sigma(j)}) } \Biggr),
\end{eqnarray*}
where $\mathcal{S}_k$ stands for the symmetric group of
$\{1,\ldots,k\}$.
Hence
the result will follow from the identity
%
%
\begin{equation} \label{eqidentity}
\sum_{\sigma\in\mathcal{S}_k} \frac{g (s_{\sigma(1)}) }{g (s_1) +
\cdots+ g(s_{k}) } \prod_{i=2}^{k-1} \frac{ g( s_{\sigma(i)})}{
\sum_{j=i}^{k-1} g( s_{\sigma(j)}) } = k-1
\end{equation}
and the product structure of $m_k$. Now we prove
(\ref{eqidentity}). Divide the sum in~(\ref{eqidentity})
depending on the value of $\sigma(1)$ and $\sigma(k)$
\[
\mbox{l.h.s. of  (\ref{eqidentity})} = \sum_{i_1=1}^k
\frac{g (s_{i_1}) }{g (s_1) + \cdots+ g(s_{k}) }
\mathop{\sum_{i_k=1}}_{i_k \neq i_1}^k
\mathop{\mathop{\sum_{\sigma\in\mathcal{S}_k:}}_{
\sigma(1)=i_1 }}_{\sigma(k)=i_k} \prod_{i=2}^{k-1} \frac{ g(
s_{\sigma(i)})}{ \sum_{j=i}^{k-1} g( s_{\sigma(j)}) } .
\]
The thesis will follow from the fact that, for any $i_1, i_k$, the
last sum in the latter is equal to 1. Equivalently, we need to prove
that, for any $n$,
%
%
\begin{equation} \label{eqidentity2}
\sum_{\sigma\in\mathcal{S}_n} \prod_{i=1}^{n}
\frac{g(s_{\sigma(i)})}{ \sum_{j=i}^{n} g( s_{\sigma(j)} ) } = 1 .
\end{equation}
This is done by induction. Indeed, the thesis is trivial for $n=1$.
Assume that (\ref{eqidentity2}) holds at rank $n-1$. Then,
\begin{eqnarray*}
&&\sum_{\sigma\in\mathcal{S}_n} \prod_{i=1}^{n}
\frac{g(s_{\sigma(i)})}{ \sum_{j=i}^{n} g( s_{\sigma(j)} ) } \\
&&\qquad =
\sum_{i_1=1}^n \mathop{\sum_{\sigma\in\mathcal{S}_n:}}_{ \sigma
(1)=i_1} \prod_{i=1}^{n} \frac{g(s_{\sigma(i)})}{ \sum_{j=i}^{n} g(
s_{\sigma(j)} ) } \\
&&\qquad = \sum_{i_1=1}^n \frac{g(s_{i_1})}{g(s_1)+g(s_2)+\cdots+g(s_n)}
\mathop{\sum_{\sigma\in\mathcal{S}_n:}}_{
\sigma(1)=i_1} \prod_{i=2}^{n}
\frac{g(s_{\sigma(i)})}{ \sum_{j=i}^{n} g( s_{\sigma(j)} ) } .
\end{eqnarray*}
Note that,
by the induction hypothesis, the second sum is equal to $1$ (for
any $i_1$). Hence,
\[
\sum_{\sigma\in\mathcal{S}_n} \prod_{i=1}^{n}
\frac{g(s_{\sigma(i)})}{ \sum_{j=i}^{n} g( s_{\sigma(j)} ) } =
\sum_{i_1=1}^n \frac{g(s_{i_1})}{g(s_1)+g(s_2)+\cdots+g(s_n)} =1 .
\]
This ends the proof of (\ref{eqidentity2}) and thus of point (a).
As already mentioned the proof of point (b) is easier [only
(\ref{eqidentity2}) has to be used]; details are left to the
reader.
\end{pf}
\end{appendix}

\makeatletter\write@toc@ignorecontentsline\makeatother
\section*{Acknowledgments}
We thank Marco Ribezzi Crivellari and Fran\c{c}ois Si\-menhaus for
useful discussions and the Laboratoire de Probabilit\'{e}s et
Mod\`{e}\-les Al\'{e}atoires of the University Paris VII and the
Department of Mathematics of the University of Roma Tre for the support and
the kind hospitality.
\makeatletter\write@toc@restorecontentsline\makeatother


%

%
\printaddresses


\begin{thebibliography}{23}

\bibitem{A}
\begin{barticle}[mr]
\bauthor{\bsnm{Aldous},~\bfnm{David~J.}\binits{D.~J.}}
(\byear{1999}).
\btitle{Deterministic and stochastic models for coalescence (aggregation and
  coagulation): A review of the mean-field theory for probabilists}.
\bjournal{Bernoulli}
\bvolume{5}
\bpages{3--48}.
\bid{doi={10.2307/3318611}, issn={1350-7265}, mr={1673235}}
\end{barticle}
\endbibitem

\bibitem{Be2}
\begin{barticle}[mr]
\bauthor{\bsnm{Bertoin},~\bfnm{Jean}\binits{J.}}
(\byear{2001}).
\btitle{Eternal additive coalescents and certain bridges with exchangeable
  increments}.
\bjournal{Ann. Probab.}
\bvolume{29}
\bpages{344--360}.
\bid{doi={10.1214/aop/1008956333}, issn={0091-1798}, mr={1825153}}
\end{barticle}
\endbibitem

\bibitem{Be}
\begin{bbook}[mr]
\bauthor{\bsnm{Bertoin},~\bfnm{Jean}\binits{J.}}
(\byear{2006}).
\btitle{Random Fragmentation and Coagulation Processes}.
\bseries{Cambridge Studies in Advanced Mathematics}
\bvolume{102}.
\bpublisher{Cambridge Univ. Press}, \baddress{Cambridge}.
\bid{doi={10.1017/CBO9780511617768}, mr={2253162}}
\end{bbook}
\endbibitem

\bibitem{B}
\begin{bbook}[mr]
\bauthor{\bsnm{Billingsley},~\bfnm{Patrick}\binits{P.}}
(\byear{1968}).
\btitle{Convergence of Probability Measures}.
\bpublisher{Wiley}, \baddress{New York}.
\bid{mr={0233396}}
\end{bbook}
\endbibitem

\bibitem{BDG}
\begin{barticle}[auto:STB|2011-03-03|12:04:44]
\bauthor{\bsnm{Bray},~\bfnm{A.~J.}\binits{A.~J.}},
  \bauthor{\bsnm{Derrida},~\bfnm{B.}\binits{B.}} \AND
  \bauthor{\bsnm{Gordr{\`e}che},~\bfnm{C.}\binits{C.}}
(\byear{1994}).
\btitle{Non-trivial algebraic decay in a soluble model of coarsening}.
\bjournal{Europhys. Lett.}
\bvolume{27}
\bpages{175--180}.
\end{barticle}
\endbibitem

\bibitem{DV}
\begin{bbook}[mr]
\bauthor{\bsnm{Daley},~\bfnm{D.~J.}\binits{D.~J.}} \AND
  \bauthor{\bsnm{Vere-Jones},~\bfnm{D.}\binits{D.}}
(\byear{1988}).
\btitle{An Introduction to the Theory of Point Processes}.
\bpublisher{Springer}, \baddress{New York}.
\bid{mr={0950166}}
\end{bbook}
\endbibitem

\bibitem{DBG}
\begin{barticle}[auto:STB|2011-03-03|12:04:44]
\bauthor{\bsnm{Derrida},~\bfnm{B.}\binits{B.}},
  \bauthor{\bsnm{Bray},~\bfnm{A.~J.}\binits{A.~J.}} \AND
  \bauthor{\bsnm{Godr{\`e}che},~\bfnm{C.}\binits{C.}}
(\byear{1994}).
\btitle{Non-trivial exponents in the zero temperature dynamics of the 1d Ising
  and Potts model}.
\bjournal{J. Phys. A}
\bvolume{27}
\bpages{L357--L361}.
\end{barticle}
\endbibitem

\bibitem{DGY1}
\begin{barticle}[auto:STB|2011-03-03|12:04:44]
\bauthor{\bsnm{Derrida},~\bfnm{B.}\binits{B.}},
  \bauthor{\bsnm{Godr{\`e}che},~\bfnm{C.}\binits{C.}} \AND
  \bauthor{\bsnm{Yekutieli},~\bfnm{I.}\binits{I.}}
(\byear{1990}).
\btitle{Stable distributions of growing and coalescing droplets}.
\bjournal{Europhys. Lett.}
\bvolume{12}
\bpages{385--390}.
\end{barticle}
\endbibitem

\bibitem{DGY2}
\begin{barticle}[auto:STB|2011-03-03|12:04:44]
\bauthor{\bsnm{Derrida},~\bfnm{B.}\binits{B.}},
  \bauthor{\bsnm{Godr{\`e}che},~\bfnm{C.}\binits{C.}} \AND
  \bauthor{\bsnm{Yekutieli},~\bfnm{I.}\binits{I.}}
(\byear{1991}).
\btitle{Scale invariant regime in the one dimensional models of growing and
  coalescing droplets}.
\bjournal{Phys. Rev. A}
\bvolume{44}
\bpages{6241--6251}.
\end{barticle}
\endbibitem

\bibitem{D}
\begin{bbook}[mr]
\bauthor{\bsnm{Durrett},~\bfnm{Richard}\binits{R.}}
(\byear{1996}).
\btitle{Probability: Theory and Examples}, \bedition{2nd} ed.
\bpublisher{Duxbury Press}, \baddress{Belmont, CA}.
\bid{mr={1609153}}
\end{bbook}
\endbibitem

\bibitem{EJ}
\begin{barticle}[auto:STB|2011-03-03|12:04:44]
\bauthor{\bsnm{Eisinger},~\bfnm{S.}\binits{S.}} \AND
  \bauthor{\bsnm{Jackle},~\bfnm{J.}\binits{J.}}
(\byear{1991}).
\btitle{A hierarchically constrained kinetic ising model}.
\bjournal{Z.~Phys. B}
\bvolume{84}
\bpages{115--124}.
\end{barticle}
\endbibitem

\bibitem{FMRT1}
\begin{bmisc}[auto:STB|2011-03-03|12:04:44]
\bauthor{\bsnm{Faggionato},~\bfnm{A.}\binits{A.}},
  \bauthor{\bsnm{Martinelli},~\bfnm{F.}\binits{F.}},
  \bauthor{\bsnm{Roberto},~\bfnm{C.}\binits{C.}} \AND
  \bauthor{\bsnm{Toninelli},~\bfnm{C.}\binits{C.}}
(\byear{2010}).
\bhowpublished{Aging through hierarchical coalescence in the East model.
  Preprint. Available at
  \href{http://arxiv.org/abs/arXiv:1012.4912}{arXiv:1012.4912}.}
\end{bmisc}
\endbibitem

\bibitem{FMRT2}
\begin{bmisc}[auto:STB|2011-03-03|12:04:44]
\bauthor{\bsnm{Faggionato},~\bfnm{A.}\binits{A.}},
  \bauthor{\bsnm{Martinelli},~\bfnm{F.}\binits{F.}},
  \bauthor{\bsnm{Roberto},~\bfnm{C.}\binits{C.}} \AND
  \bauthor{\bsnm{Toninelli},~\bfnm{C.}\binits{C.}}
(\byear{2011}).
\bhowpublished{Unpublished manuscript}.
\end{bmisc}
\endbibitem

\bibitem{Fe2}
\begin{bbook}[mr]
\bauthor{\bsnm{Feller},~\bfnm{William}\binits{W.}}
(\byear{1971}).
\btitle{An Introduction to Probability Theory and Its Applications. {V}ol.
  {II}}, \bedition{2nd} ed.
\bpublisher{Wiley}, \baddress{New York}.
\bid{mr={0270403}}
\end{bbook}
\endbibitem

\bibitem{FKAS}
\begin{bbook}[mr]
\bauthor{\bsnm{Franken},~\bfnm{Peter}\binits{P.}},
  \bauthor{\bsnm{K{\"o}nig},~\bfnm{Dieter}\binits{D.}},
  \bauthor{\bsnm{Arndt},~\bfnm{Ursula}\binits{U.}} \AND
  \bauthor{\bsnm{Schmidt},~\bfnm{Volker}\binits{V.}}
(\byear{1982}).
\btitle{Queues and Point Processes}.
\bpublisher{Wiley}, \baddress{Chichester}.
\bid{mr={0691670}}
\end{bbook}
\endbibitem

\bibitem{GM}
\begin{barticle}[mr]
\bauthor{\bsnm{Gallay},~\bfnm{T.}\binits{T.}} \AND
  \bauthor{\bsnm{Mielke},~\bfnm{A.}\binits{A.}}
(\byear{2003}).
\btitle{Convergence results for a coarsening model using global linearization}.
\bjournal{J. Nonlinear Sci.}
\bvolume{13}
\bpages{311--346}.
\bid{doi={10.1007/s00332-002-0543-8}, issn={0938-8974}, mr={1982018}}
\end{barticle}\vadjust{\goodbreak}
\endbibitem

\bibitem{K}
\begin{bbook}[mr]
\bauthor{\bsnm{Kallenberg},~\bfnm{Olav}\binits{O.}}
(\byear{2005}).
\btitle{Probabilistic Symmetries and Invariance Principles}.
\bpublisher{Springer}, \baddress{New York}.
\bid{mr={2161313}}
\end{bbook}
\endbibitem

\bibitem{LS}
\begin{barticle}[mr]
\bauthor{\bsnm{Limic},~\bfnm{Vlada}\binits{V.}} \AND
  \bauthor{\bsnm{Sturm},~\bfnm{Anja}\binits{A.}}
(\byear{2006}).
\btitle{The spatial {$\Lambda $}-coalescent}.
\bjournal{Electron. J. Probab.}
\bvolume{11}
\bpages{363--393 (electronic)}.
\bid{issn={1083-6489}, mr={2223040}}
\end{barticle}
\endbibitem

\bibitem{MKK}
\begin{bbook}[mr]
\bauthor{\bsnm{Matthes},~\bfnm{Klaus}\binits{K.}},
  \bauthor{\bsnm{Kerstan},~\bfnm{Johannes}\binits{J.}} \AND
  \bauthor{\bsnm{Mecke},~\bfnm{Joseph}\binits{J.}}
(\byear{1978}).
\btitle{Infinitely Divisible Point Processes}.
\bpublisher{Wiley}, \baddress{New York}.
\bid{mr={0517931}}
\end{bbook}
\endbibitem

\bibitem{Pe}
\begin{bincollection}[mr]
\bauthor{\bsnm{Pego},~\bfnm{Robert~L.}\binits{R.~L.}}
(\byear{2007}).
\btitle{Lectures on dynamics in models of coarsening and coagulation}.
In \bbooktitle{Dynamics in Models of Coarsening, Coagulation, Condensation and
  Quantization}.
\bseries{Lect. Notes Ser. Inst. Math. Sci. Natl. Univ. Singap.}
\bvolume{9}
\bpages{1--61}.
\bpublisher{World Scientific}, \baddress{Hackensack, NJ}.
\bid{doi={10.1142/9789812770226_0001}, mr={2395779}}
\end{bincollection}
\endbibitem

\bibitem{Pr}
\begin{bbook}[auto:STB|2011-03-03|12:04:44]
\beditor{\bsnm{Privman},~\bfnm{V.}\binits{V.}}, ed.
(\byear{1997}).
\btitle{Non Equilibrium Statistical Mechanics in One Dimension}.
\bpublisher{Cambridge Univ. Press}, \baddress{New York}.
\end{bbook}
\endbibitem

\bibitem{SE}
\begin{barticle}[auto:STB|2011-03-03|12:04:44]
\bauthor{\bsnm{Sollich},~\bfnm{P.}\binits{P.}} \AND
  \bauthor{\bsnm{Evans},~\bfnm{M.~R.}\binits{M.~R.}}
(\byear{2003}).
\btitle{Glassy dynamics in the asymmetrically constrained kinetic Ising chain}.
\bjournal{Phys. Rev. E}
\bvolume{68}
\bpages{031504}.
\end{barticle}
\endbibitem

\end{thebibliography}
\end{document}